\documentclass[reqno]{amsart}
\pdfoutput=1
\usepackage{amsmath, amssymb,amsthm}
\usepackage{txfonts}
\usepackage{enumerate}
\usepackage{color}
\usepackage[normalem]{ulem}

\newtheorem{theorem}{Theorem}

\newtheorem{corollary}[theorem]{Corollary}
\newtheorem{definition}[theorem]{Definition}
\newtheorem{lemma}[theorem]{Lemma}
\newtheorem{proposition}[theorem]{Proposition}
\newtheorem{remark}[theorem]{Remark}
\newtheorem{example}[theorem]{Example}

 \newcommand{\eps}{\varepsilon}

  \newcommand{\F}{\mathcal{F}}
 
 \newcommand{\M}{\mathsf{M}}
 
 \newcommand{\W}{\mathcal{W}}
 \newcommand{\X}{\mathcal{X}}

 \renewcommand{\phi}{\varphi}

\newcommand{\E}{\mathbb{E}}
\renewcommand{\P}{\mathbb{P}}
\newcommand{\N}{\mathbb{N}}
\newcommand{\Q}{\mathbb{Q}}
\newcommand{\R}{\mathbb{R}}

\newcommand{\law}{\mathrm{Law}}

\newcommand{\Xs}{\mathcal{X}}
\newcommand{\Ys}{\mathcal{Y}}

\newcommand{\Ysf}{\mathsf{Y}}
\newcommand{\sX}{\mathsf{X}}

\newcommand{\cF}{\mathcal{F}}
\newcommand{\cG}{\mathcal{G}}

\renewcommand{\W}{{\mathbb W}}

\usepackage{bbm}
\newcommand{\I}{\mathbbm 1}
\newcommand{\1}{\mathbbm 1}
\newcommand{\indic}{\1}

\renewcommand{\eps}{\varepsilon}
\renewcommand{\epsilon}{\varepsilon}

\newcommand{\eg}{e.g.}

\DeclareMathOperator{\id}{Id}
\DeclareMathOperator{\proj}{proj}
\DeclareMathOperator*{\argmin}{arg\,min}

\newcommand{\fourIdx}[5]{%
\setbox1=\hbox{\ensuremath{^{#1}}}%
 \setbox2=\hbox{\ensuremath{_{#2}}}%
 \setbox5=\hbox{\ensuremath{#5}}%
 \hspace{\ifnum\wd1>\wd2\wd1\else\wd2\fi}%
 \ensuremath{\copy5^{\hspace{-\wd1}\hspace{-\wd5}#1\hspace{\wd5}#3}%
 _{\hspace{-\wd2}\hspace{-\wd5}#2\hspace{\wd5}#4}%
 }}

\numberwithin{equation}{section}
\numberwithin{theorem}{section}

\newcommand{\TRT}{\mathsf{JOIN}}

\newcommand{\RST}{\mathsf{RST}}

\newcommand{\RMST}{\mathsf{RMST}}

\newcommand{\TRST}{\mathsf{JOIN}}

\newcommand{\Opt}{\mathsf{Opt}}

\newcommand{\SG}{\mathsf{SG}}
\newcommand{\CS}{\mathsf{CS}}
\newcommand{\MCS}{\mathsf{MCS}}

\newcommand{\cl}{\textsc{cl}}
\newcommand{\op}{\textsc{op}}

\renewcommand{\subset}{\subseteq}
\renewcommand{\supset}{\supseteq}

\newcommand{\ol}[1]{\overline{#1}}
\newcommand{\ul}[1]{\underline{#1}}
\newcommand{\cat}{\otimes}
\renewcommand{\S}[1]{S^{\cat {#1}}}

\newcommand{\CR}{{C(\R_+)}}
\newcommand{\CRx}{{C_x(\R_+)}}
\newcommand{\CRo}{{C_0(\R_+)}}
\newcommand{\CRR}{{C(\R_+)}}

\newcommand{\leb}{\mathcal L}

\newcommand\numberthis{\addtocounter{equation}{1}\tag{\theequation}}
\renewcommand{\llcorner}{\upharpoonright}

\usepackage{asymptote}
\begin{asydef}
usepackage("amsmath");
texpreamble("\newcommand{\ul}[1]{\underline{#1}}\newcommand{\ol}[1]{\overline{#1}}\newcommand{\indic}[1]{\boldsymbol{1}_{\{\ensuremath{#1}\}}}");
\end{asydef}

\begin{document}

\title{The geometry of multi-marginal Skorokhod Embedding}

\author{Mathias Beiglb\"ock} \author{Alexander M.~G.~Cox} \author{Martin
  Huesmann} \thanks{ The first author gratefully acknowledges support
   by the FWF-grants p26736 and Y782, the third author gratefully acknowledges support by the German  Research  Foundation  through  the  Hausdorff
Center for Mathematics and the Collaborative Research Center 1060.}  \date{\today}
\begin{abstract}

  The Skorokhod Embedding Problem (SEP) is one of the classical problems in the study of stochastic processes, with applications in many different fields (cf.~ the surveys \cite{Ob04,Ho11}). Many of these applications have natural multi-marginal extensions leading to the \emph{(optimal) multi-marginal Skorokhod problem} (MSEP). Some of the first papers to consider this problem are \cite{Ho98b, BrHoRo01b, MaYo02}. However, this turns out to be difficult using existing techniques: only recently a complete solution was be obtained in \cite{CoObTo15} establishing an extension of the Root construction, while other instances are only partially answered or remain wide open.

In this paper, we extend the theory developed in \cite{BeCoHu14} to the multi-marginal setup which is comparable to the extension of the optimal transport problem to the multi-marginal optimal transport problem.  As for the one-marginal case, this viewpoint turns out to be very powerful. In particular, we are able to show that all classical optimal embeddings have natural multi-marginal counterparts. Notably these different constructions are linked through a joint geometric structure and the classical solutions are recovered as particular cases.

Moreover, our results also have consequences for the study of the martingale transport problem as well as the peacock problem.

\smallskip

\end{abstract}
\maketitle

\section{Introduction} 

The Skorokhod Embedding problem (SEP) is a classical problem in probability, dating back to the 1960s (\cite{Sk61,Sk65}). Simply stated, the aim is to represent a given probability as the distribution of Brownian motion at a chosen stopping time. Recently, motivated by applications in probability, mathematical finance, and numerical methods, there has been renewed, sustained interest in solutions to the SEP (cf.\ the two surveys \cite{Ob04,Ho11}) and its multi-marginal extension, the multi-marginal SEP: Given marginal measures $\mu_0,\ldots,  \mu_n$ of finite variance and a Brownian motion with $B_0\sim \mu_0$, construct stopping times $\tau_1\leq \ldots\leq\tau_n$ s.t. %satisfying 
\begin{align}\label{MSkoSol}\tag{MSEP}
\text{$ B_{\tau_i} \sim \mu_i$ for all $1\leq i\leq n$ and $\E[\tau_n]<\infty.$ }
\end{align}
It is well known that a solution to \eqref{MSkoSol} exists iff the marginals are in convex order ($\mu_0\preceq_c\ldots \preceq_c \mu_n$) and have finite second moment; under this condition Skorokhod's original results give the existence of solutions of the induced one period problems, which can then be pasted together to obtain a solution to \eqref{MSkoSol}.

It appears to be significantly harder to develop genuine extensions of one period solutions: many of the classical solutions to the SEP exhibit additional desirable characteristics and optimality properties which one would like to extend to the multi-marginal case.  However the original derivations of these solutions make significant use of the particular structure inherent to certain problems, often relying on explicit calculations, which make extensions difficult if not impossible. The first paper which we are aware of to attempt to extend a classical construction to the multi-marginal setting is \cite{BrHoRo01b}, who generalised the Az\'{e}ma-Yor embedding (\cite{AY79}) to the case with two marginals. This work was further extended by Henry-Labord\`ere, Ob\l\'oj, Spoida, and Touzi \cite{HLOb12, ObSp14}, who were able to extend to arbitrary (finite) marginals, under a non-trivial assumption on the measures.  Using an extension of the stochastic control approach in \cite{GaHeTo12} Claisse, Guo, and Henry-Labord\`ere \cite{ClGuHL15} constructed a two marginal extension of the Vallois embedding.  Recently, Cox, Obloj, and Touzi \cite{CoObTo15} were able to characterise the solution to the general multi-marginal Root embedding through the use of an optimal stopping formulation.

\subsection*{Mass transport approach and general multi-marginal embedding}
%\subsubsection*{Mass Transport Approach}
In this paper, we develop a new approach to
the multi-marginal Skorokhod problem, based on insights from the field of optimal transport.

Following the seminal paper of Gangbo and McCann \cite{GaMc96} the mutual interplay of optimality and geometry of optimal transport plans has been a cornerstone of the field. As shown for example in \cite{CoFrKl13,Pa15} this in not limited to the two-marginal case but extends to the multi-marginal case where it turns out to be much harder though. %driving force in the field. 
Recently, similar ideas have been shown to carry over to a more probablistic context, to optimal transport problems satisfying additional linear constraints \cite{BeJu12, Za14, GhKiLi15} and in fact to the classical Skorokhod embedding problem \cite{BeCoHu14}.

Building on these insights, we extend the mass transport viewpoint developed in \cite{BeCoHu14} to the %full fledged 
multi-marginal Skorokhod embedding problem.  This allows us to give multi-marginal extensions of all the classical optimal  solutions to the Skorokhod problem in full generality, which we exemplify by several examples. In particular the classical solutions of Az\'ema-Yor, Root, Rost, Jacka, Perkins, and Vallois can be recovered as special cases.
 In addition, the approach allows us to derive a number of new solutions to \eqref{MSkoSol} which have further applications to e.g.\ martingale optimal transport and the peacock problem. 
A main contribution of this paper is that in many different cases, solutions to the multi-marginal SEP share a common \emph{geometric structure}. In all the cases we consider, this geometric information will in fact be enough to characterise the optimiser uniquely, which highlights the flexibility of our approach.

Furthermore, our approach to the Skorokhod embedding problem is very general and does not rely on fine properties of Brownian motion. Therefore, exactly as in \cite{BeCoHu14} the results of this article carry over to sufficiently regular Markov processes, \eg{} geometric Brownian motion, three-dimensional Bessel process and Ornstein-Uhlenbeck processes, and Brownian motion in $\R^d$ for $d >1$. As the arguments are precisely the same as in  \cite{BeCoHu14}, we refer to  \cite[Section 8]{BeCoHu14} for details.

\subsection*{Related Work}
Interest in the multi-marginal Skorokhod problem comes from a number of directions and we describe some of these here:

\begin{itemize}
\item {\bf Maximising the running maximum: the Az\'{e}ma-Yor embedding}

  \noindent
  Suppose $(M_t)_{t\geq0}$ is a martingale and write $\bar M_t:=\sup_{s\leq t}M_s$.  The relationship between the laws of $M_1$ and $\bar M_1$ has been studied by Blackwell and Dubins \cite{BlDu63}, Dubins and Gilat \cite{DuGi78} and Kertz and R\"osler \cite{KeRo90}, culminating in a complete classification of all possible joint laws by Rogers \cite{Ro93}. In particular given the law of $M_1$, the set of possible laws of $\bar M_1$ admits a maximum w.r.t.\ the stochastic ordering, this can be seen through the Az\'{e}ma-Yor embedding.  Given initial and terminal laws of the martingale, Hobson \cite{Ho98a} gave a sharp upper bound on the law of the maximum based on an extension of the Az\'{e}ma-Yor embedding to Brownian motion started according to a non-trivial initial law. These results are further extended in \cite{BrHoRo01b} to the case of martingales started in $0$ and constrained to a specified marginal at an intermediate time point, essentially based on a further extension of the Az\'{e}ma-Yor construction. The natural aim is to solve this question in the case of arbitrarily many marginals. Assuming that the marginals have ordered barycenter functions this case is included in the work of Madan and Yor \cite{MaYo02}, based on iterating the Az\'{e}ma-Yor scheme. More recently, the stochastic control approach of \cite{GaHeTo12} (for one marginal) is extended by Henry-Labord\`ere, Ob\l\'oj, Spoida, and Touzi \cite{HLOb12, ObSp14} to marginals in convex order satisfying an additional assumption (\cite[Assumption $\circledast$]{ObSp14}\footnote{As shown by an example in \cite{ObSp14} this condition is necessary to carry out their explicit construction.}). Together with the Dambis-Dubins-Schwarz Theorem,  Theorem~\ref{thm:AzemaYor} below provides a solution to this problem in full generality.

\item {\bf Multi-Marginal Root embedding}

\noindent
In a now classical paper, Root \cite{Ro69} showed that for any centred distribution with finite second moment, $\mu$, there exists a (right) \emph{barrier} $\mathcal R$, i.e.\ a Borel subset of $\R_+\times\R$ such that $(t,x)\in \mathcal R$ implies $(s,x)\in \mathcal R$ for all $s\geq t$, and for which $B_{\tau_{\mathcal R}}\sim \mu$, $\tau_{\mathcal R}=\inf\{t: (t,B_t)\in \mathcal R\}$. This work was further generalised to a large class of Markov processes by Rost \cite{Ro76}, who also showed that this construction was optimal in that it minimised $\E[h(\tau)]$ for convex functions $h$.

More recent work on the Root embedding has focused on attempts to characterise the stopping region. A number of papers do this either through analytical means (\cite{Mc91,CoWa11,CoWa13,GaObRe15}) or through connections with optimal stopping problems (\cite{DeAngelis:15}). Recently the connection to optimal stopping problems has enabled Cox, Ob\l\'oj, and Touzi \cite{CoObTo15} to extend these results to the multi-marginal setting Moreover, they prove that this solution enjoys a similar optimality property to the one-marginal Root solution. The principal strategy is to first prove the result in the case of locally finitely supported measures by means of a time reversal argument. The proof is then completed in the case of general measures by a delicate limiting procedure. 

As a consequence of the theoretical results in this paper, we will be able to prove similar results. In particular, the barrier structure as well as the optimality properties are recovered in Theorem~\ref{thm:Root1}. Indeed, as we will show below, the particular geometric structure of the Root embedding turns out to be archetypal for a number of multi-marginal counterparts of classical embeddings.

\item {\bf Model-independent Finance}

\noindent
 An important application field for the results in this paper, and one of the motivating factors behind the recent resurgence of interest in the SEP, relates to model-independent finance. In mathematical finance, one models the price process $S$ as a martingale under a risk-neutral measure, and specifying prices of call options at maturity $T$ is equivalent to fixing the distribution $\mu$ of $S_T$. Understanding no-arbitrage price bounds for a functional $\gamma$, can often be seen to be equivalent to finding the range of $\E[\gamma(B)_\tau]$ among all solutions to the Skorokhod embedding problem for $\mu$. This link between SEP and model-independent pricing and hedging was pioneered by Hobson~\cite{Ho98a} and has been an important question ever since. A comprehensive overview is given in \cite{Ho11}. 

 However, the above approach uses only market data for the maturity time $T$, while in practice market data for many intermediate maturities may also be available, and this corresponds to the multi-marginal SEP. While we do not pursue this direction of research in this article we emphasize that our approach yields a systematic method to address this problem.  In particular, the general framework of super-replication results for model-independent framework now includes a number of important contributions, see \cite{DoSo12, GTT15, BeCoHuPePr15, HoOb15}, and most of these papers allow for information at multiple intermediate times.

\item {\bf Martingale optimal transport}

\noindent
Optimal transport problems where the transport plan must satisfy additional martingale constraints have recently been investigated, \eg{} in \cite{HoNe12, BeHePe12, BeJu12, GaHeTo12, DoSo12, CaLaMa17}. Besides having a natural interpretation in finance, such martingale transport problems are also of independent mathematical interest, for example -- similarly to classical optimal transport -- they have consequences for the investigation of martingale inequalities (see \eg{} \cite{BoNu13,HeObSpTo12,ObSp14}).  As observed in \cite{BeHeTo15} one can gain insight into the martingale transport problem between two probabilities $\mu_1$ and $\mu_2$ by relating it to a Skorokhod embedding problem which may be considered as continuous time version of the martingale transport problem. Notably this idea can be used to recover the known solutions of the martingale optimal transport problem in a unified fashion (\cite{HuSt16}). It thus seems natural that an improved understanding of an $n$-marginal martingale transport problem can be obtained based on the multi-marginal Skorokhod embedding problem. Indeed this is exemplified in Theorem~\ref{thm:MMMM1} below, where we use a multi-marginal embedding to establish an $n$-period version of the martingale monotone transport plan, and recover similar results to recent work of Nutz, Stebegg, and Tan \cite{NuStTa17}.

\item {\bf Construction of peacocks}

\noindent
Dating back to the work of Madan--Yor \cite{MaYo02}, and studied systematically in the book of Hirsch, Profeta, Roynette and Yor \cite{HiPr11}, given a family of probability measures $(\mu_t)_{t \in [0,T]}$ which are increasing in convex order, a \emph{peacock} (from the acronym PCOC ``Processus Croissant pour l'Ordre Convexe'') is a martingale such that $M_t \sim \mu_t$ for all $t \in [0,T]$.  The existence of such a process is granted by Kellerer's celebrated theorem, and typically there is an abundance of such processes. Loosely speaking, the peacock problem is to give constructions of such martingales.  Often such constructions are based on Skorokhod embedding or particular martingale transport plans, and often one is further interested in producing solutions with some additional optimality properties; see for example the recent works \cite{HeTaTo16, Ju14a, KaTaTo15,Ho16}.

Given the intricacies of multi-period martingale optimal transport and Skorokhod embedding, it is necessary to make additional assumptions on the underlying marginals and desired optimality properties are in general not preserved in a straight forward way during the inherent limiting/pasting procedure.  We expect that an improved understanding of the multi-marginal Skorokhod embedding problem will provide a first step to tackle these range of problems in a systematic fashion.
\end{itemize}

\subsection{Outline of the Paper}

We will proceed as follows. In Section~\ref{sec:probint}, we will describe our main results. Our main technical tool is a `monotonicity principle', Theorem~\ref{thm:monprin2}. This result allows us to deduce the geometric structure of optimisers. Having stated this result, and defined the notion of `stop-go pairs', which are important mathematical embodiment of the notion of `swapping' stopping rules for a candidate optimiser, we will be able to deduce our main consequential results. Specifically, we will prove the multi-marginal generalisations of the Root, Rost and Az\'{e}ma-Yor embeddings, using their optimality properties as a key tool in their construction. The Rost construction is entirely novel, and the solution to the Az\'{e}ma-Yor embedding generalises existing results, which have only previously been given under a stronger assumption on the measures. We also give a multi-marginal generalisation of an embedding due to Hobson \& Pedersen; this is, in some sense, the counterpart of the Az\'{e}ma-Yor embedding; classically, this is better recognised as the embedding of Perkins \cite{Pe86}, however for reasons we give later, this embedding has no multi-marginal extension.  Moreover the proofs of these results will share a common structure, and it will be clear how to generalise these methods to provide similar results for a number of other classical solutions to the SEP.

In Section~\ref{sec:probint}, we also use our methods to give a multi-marginal martingale monotone transport plan, using a construction based on a SEP-viewpoint.

The remainder of the paper is then dedicated to proving the main technical result, Theorem~\ref{thm:monprin2}. In Section~\ref{sec:Si}, we introduce our technical setup, and prove some preliminary results. As in \cite{BeCoHu14}, it will be important to consider the class of \emph{randomised} multi-stopping times, and we define these in this section, and derive a number of useful properties. It is technically convenient to consider randomised multi-stopping times on a canonical probability space, where there is sufficient additional randomisation, independent of the Brownian motion, however we will prove in Lemma~\ref{lem:rmst} that any sufficiently rich probability space will suffice. A key property of the set of randomised multi-stopping times embedding a given sequence of measures is that this set is compact in an appropriate (weak) topology, and this will be proved in Proposition~\ref{prop:RMSTcompact}; an important consequence of this is that optimisers of the multi-marginal SEP exist under relatively mild assumptions on the objective (Theorem~\ref{thm:exists}).

In Section~\ref{sec:CSMCSSGPairs} we introduce the notions of color-swap pairs, and multi-colour swap pairs. These will be the fundamental constituents of the set of `bad-pairs', or combinations of stopped and running paths that we do not expect to see in optimal solutions. In this section we define these pairs, and prove some technical properties of the sets.

In Section~\ref{sec:monotonicity} we complete the proof of Theorem~\ref{thm:monprin2}. In spirit this follows the proof of the corresponding result in \cite{BeCoHu14}, and we only provide the details here where the proof needs to adapt to account for the multi-marginal setting.

\subsection{Frequently used notation}\label{sec:notation}

\begin{itemize}
\item The set of (sub-)probability measures on a space $\mathsf{X}$ is denoted by $\mathcal P(\mathsf{X})$ / $\mathcal P^{\leq 1}(\mathsf{X})$.
\item $\Xi^d=\{(s_1,\ldots,s_d):0\leq s_1\leq\ldots\leq s_d\}$ denotes the $d$-dimensional simplex. 
\item The $d$-dimensional Lebesgue measure will be denoted by $\leb^d$.
\item For a measure $\xi$ on $\mathsf{X}$ we write $f(\xi)$ for the push-forward of $\xi$ under $f:\mathsf{X}\to \mathsf{Y}$.
\item We use $\xi(f)$ as well as $\int f~ d\xi$ to denote the integral of a function $f$ against a measure $\xi$.
\item $\CRx$ denotes the continuous functions starting in $x$; $\CRR=\bigcup_{x\in\R}\CRx$.  For $\omega\in\CR$ we write $\theta_s\omega$ for the path in $\CRo$ defined by $(\theta_s\omega)_{t\geq 0}=(\omega_{t+s}-\omega_s)_{t\geq 0}$.
\item $\W$ denotes Wiener measure; $\W_\mu$ denotes law of Brownian motion started according to a probability $\mu$; $\F^0$  ($\F^a$) the natural (augmented) filtration on $\CRo$.
\item For $d\in\N$ we set $\overline{\CR}=\CR\times[0,1]^d$, $\bar\W=\W\otimes\leb^d$, and $\bar \cF=(\bar \cF_t)_{t\geq 0}$ the usual augmentation of $(\cF_t^0\otimes\mathcal B([0,1]^d))_{t\geq 0}$. To keep notation manageable, we suppress $d$ from the notation since the precise number will always be clear from the context.
\item $\sX$ is a Polish space equipped with a Borel probability measure $m$. We set $\Xs:=\sX\times \CRo$, $\P=m\otimes \W$, $\cG^0=(\cG_t^0)_{t\geq 0}=(\mathcal{B}(\sX)\otimes \F_t^0)_{t\geq 0}$, $\cG^a$ the usual augmentation of $\cG^0$. 
\item For $d\in\N$ we set $\bar\X=\X\times [0,1]^d$, $\bar\P=\P\otimes\leb^d$, and $\bar \cG=(\bar \cG_t)_{t\geq 0}$ the usual augmentation of $(\cG_t^0\otimes\mathcal B([0,1]^d))_{t\geq 0}$. Again, we suppress $d$ from the notation since the precise number will always be clear from the context.
\item The set of stopped paths started at $0$ is denoted by  $S =\{(f,s): f:[0,s] \to \R \mbox{ is continuous, $f(0)=0$}\}$ and we define $r:\CRo\times\R_+\to S$ by $r(\omega, t):= (\omega_{\upharpoonright [0,t]},t)$. The set of stopped paths started in $\sX$ is $S_\sX =(\sX,S)=\{(x,f,s): f:[0,s] \to \R \mbox{ is continuous, $f(0)=0$}, x\in \sX\}$ and we define $r_\sX:\sX\times\CRo\times\R_+\to S_\sX$ by $r_\sX(x,\omega, t):= (x,\omega_{\upharpoonright [0,t]},t)$, i.e.\ $r_\sX=(\id,r)$. 
\item We use $\oplus$ for the concatenation of paths: depending on the context the arguments may be elements of $S$, $\CRo$ or $\CRo\times\R_+$. Specifically, $\oplus: \mathsf{Y} \times \mathsf{Z} \to \mathsf{Z}$, where $\mathsf{Y}$ is either $S$ or $\CRo\times\R_+$, and $\mathsf{Z}$ may be any of the three spaces. For example, if $(f,s) \in S$ and $\omega \in \CRo$, then $(f,s) \oplus \omega$ is the path 
\begin{equation}
  \omega'(t) = 
  \begin{cases}
  f(t) & t \le s\\
  f(s) + \omega(t-s) & t > s
  \end{cases} \label{eq:concatenation}.
\end{equation}
\item As well as the simple concatenation of paths, we introduce a concatenation operator which keeps track of the concatenation time: if $(f,s),(g,t) \in S$, then $(f,s)\cat(g,t) = (f\oplus g,s,s+t)$.  We denote the set of elements of this form as $\S{2}$, and inductively, $\S{i}$ in the same manner.
\item Elements of $\S{i}$ will usually be denoted by $(f,s_1,\ldots,s_i)$ or $(g,t_1,\ldots,t_i)$. We define $r_i:\CRo\times\Xi^i\to \S{i}$ by $r_i(\omega, s_1,\ldots,s_i)):= (\omega_{\upharpoonright [0,s_i]},s_1,\ldots,s_i)$. Accordingly, the set of $i$-times stopped paths started in $\sX$ is $\S{i}_\sX=(\sX,\S{i})$. Elements of $\S{i}_\sX$ are usually denoted by $(x,f,s_1,\ldots,s_i)$ or $(y,g,t_1,\ldots,t_i)$. In case of $\sX=\R$ we often simply write $(f,s_1,\ldots,s_i)$ or $(g,t_1,\ldots,t_i)$ with the understanding that $f(0),g(0)\in\R$. In case that there is no danger of confusion we will also sometimes write $\S{i}_\R=\S{i}$. 
The operators $\oplus, \cat$ generalise in the obvious way to allow elements of $\S{i}_\sX$ to the left of the operator.
\item For $(x,f,s_1,\ldots,s_i)\in\S{i}_\sX, (h,s)\in S$ we often denote their concatenation by $(x,f,s_1,\ldots,s_i)|(h,s)$ which is the same element as $(x,f,s_1,\ldots,s_i)\otimes (h,s)$ but comes with the probabilistic interpretation of conditioning on the continuation of $(f,s_1,\ldots,s_i)$ by $(h,s)$. In practice, this means that we will typically expect the $(h,s)$ to be absorbed by a later $\oplus$ operation.
\item The map $\X\times\Xi^i\ni(x,\omega,s_1,\ldots,s_i)\mapsto (x,\omega_{\llcorner [0,s_i]},s_1,\ldots,s_i)\in \S{i}_\sX$ will (by slight abuse of notation) also be denoted by $r_i$. 
\item We set $ \tilde r_i: \X\times\Xi^i \to \S{i}_\sX\times C(\R_+), (x,\omega,s_1,\ldots,s_i)\mapsto ((x,\omega_{\llcorner [0,s_i]},s_1,\ldots,s_i), \theta_{s_i}\omega)$. Then $\tilde r_i$ is clearly a homeomorphism with inverse map
$$\tilde r^{-1}_i:((x,f,s_1,\ldots,s_i),\omega)\mapsto (x,f\oplus \omega,s_1,\ldots,s_i).$$
Hence, $\xi^i=\tilde r_i^{-1}(\tilde r_i(\xi^i))~.$ For $1\leq i<d$ we can extend $\tilde r_i$ to a map $\tilde r_{d,i}: \X\times\Xi^d\to \S{i}_\sX\times C(\R_+)\times \Xi^{d-i}$ by setting
$$\tilde r_{d,i}(x,\omega,s_1,\ldots,s_d)=((x,\omega_{\llcorner [0,s_i]},s_1,\ldots,s_i),\theta_{s_i}\omega, (s_{i+1}-s_i,\ldots,s_d-s_i)).$$ 
\item For $\Gamma_i \subseteq \S{i}$  we set $\Gamma_i^<:=\{(f,s_1,\ldots, s_{i-1},s_i): \exists (\tilde f,s_1,\ldots,s_{i-1}, \tilde s)\in \Gamma, s_{i-1}\leq s_i< \tilde s \mbox{ and $f\equiv \tilde f$ on $[0,s]$}\},$ where we set $s_0=0.$
\item For   $(f,s_1,\ldots,s_i) \in \S{i}$ we write $\ol{f} = \sup_{r \le s_i} f(r)$, and $\ul{f} = \inf_{r \le s_i} f(r)$.
\item   For $1\leq i<n$ and $F$  a function on $\S{n}$ resp.\ $\CRo\times \Xi^n$ and $(f,s_1,\ldots,s_i)\in \S{i}$ we set 
  \begin{align*}
    F^{(f,s_1,\ldots,s_i)\cat} (\eta,t_{i+1},\ldots,t_n)
    & := F(f\oplus \eta,s_1,\ldots,s_i,s_i+t_{i+1},\ldots,s_i+t_n)\\
    & \quad = F\left( (f,s_1,\ldots,s_i)\cat (\eta,t_{i+1},\ldots,t_n)\right),
  \end{align*}
  where $(\eta,t_{i+1},\ldots,t_n)$ may be an element of $\S{n-i}$, or $\CRo\times \Xi^{n-i}$. We similarly define
  \begin{align*}
    F^{(f,s_1,\ldots,s_i)\oplus} (\eta,t_{i+1},\ldots,t_n)
    & := F(f\oplus \eta,s_1,\ldots,s_{i-1},s_i+t_{i+1},\ldots,s_i+t_n)\\
    & \quad = F\left( (f,s_1,\ldots,s_i)\oplus (\eta,t_{i+1},\ldots,t_n)\right),
  \end{align*}
  where $(\eta,t_{i},\ldots,t_n)$ may be an element of $\S{n-i+1}$, or $\CRo\times \Xi^{n-i+1}$.
\item For any $j$-tuple $1\leq i_1 <\ldots < i_j\leq d$ we denote by
$\proj_{\X\times (i_1,\ldots,i_j)}$  the projection  from
$\X\times \R^d$ to $\X\times\R^j$ defined by
$$(x,\omega,y_1,\ldots,y_d)\mapsto (x,\omega,y_{i_1},\ldots,y_{i_j})$$
and correspondingly, for $\xi\in\mathcal P(\X\times \R^d)$, $\xi^{(i_1,\ldots,i_j)}=\proj_{\X\times
  (i_1,\ldots,i_j)}(\xi).$  When $j=0$, we understand this as simply
the projection onto $\X$. If $(i_1,\ldots,i_j)=(1,\ldots,j)$ we simply write $\xi^{(1,\ldots,j)}=\xi^i$. 
\end{itemize}

\section{Main Results}

\subsection{Existence and Monotonicity Principle} \label{sec:probint}

In this section we present our key results and provide an interpretation in probabilistic terms. To move closer to classical probabilistic notions, in this section, we slightly deviate from the notation used in the rest of the article. We consider a Brownian motion $B$ on some generic probability space and recall that, for each $1\leq i \leq n$,
$$\S{i}:=\{(f,s_1,\ldots,s_i):0\leq s_1\leq\ldots\leq s_i, f\in C([0,s_i])  \}.$$
We note that $\S{i}$ carries a natural Polish topology.
 For a function $\gamma:\S{n} \to \R$ which is Borel and a sequence $(\mu_i)_{i=0}^n$ of centered probability measures on $\R$, increasing in convex order, we are interested in the optimization problem
\begin{align}\label{eq:Pgamma} 
 P_\gamma= \inf\{\E[\gamma( (B_s)_{s\leq \tau_n}, \tau_1, \ldots, \tau_n)]: \tau_1, \dots, \tau_n \text{ satisfy \eqref{MSkoSol}}\} \tag{$\mathsf{OptMSEP}$}.
\end{align}
We denote the set of all minimizers of \eqref{eq:Pgamma} by $\Opt_\gamma$. Take another Borel measurable function $\gamma_2:\S{n}\to\R$. We will be also interested in the secondary optimization problem
\begin{align}\label{eq:secopt}
 P_{\gamma_2|\gamma}=\inf\{\E[\gamma_2( (B_s)_{s\leq \tau_n}, \tau_1, \ldots, \tau_n)]: (\tau_1,\ldots,\tau_n)\in\Opt_\gamma\}. \tag{$\mathsf{OptMSEP}_2$}
\end{align}
Both optimization problems, \eqref{eq:Pgamma} and \eqref{eq:secopt},will not depend on the particular choice of the underlying
probability space, provided that $(\Omega,\F,(\F_t)_{t\geq 0},\P)$ is
sufficiently rich that it supports a Brownian motion $(B_t)_{t \ge 0}$
starting with law  $\mu_0$, and an independent, uniformly distributed random variable $Y$, which is $\F_0$-measurable (see Lemma~\ref{lem:rmst}). We will from now on assume that we are working in this setting. On this space, we denote the filtration generated by the Brownian motion by $\F^B$.

We will usually assume that \eqref{eq:Pgamma} and \eqref{eq:secopt} are
  \emph{well-posed} in the sense that $\E\Big[\gamma\big((B_s)_{s\leq
    \tau_n},\tau_1,\ldots,\tau_n\big)\Big]$ and $\E\Big[\gamma_2\big((B_s)_{s\leq
    \tau_n},\tau_1,\ldots,\tau_n\big)\Big]$ exist with values in $(-\infty,\infty]$ for all
  $\tau=(\tau_1,\ldots,\tau_n)$ which solve \eqref{MSkoSol} and is finite for one such $\tau$.

\begin{theorem}\label{thm:exists}
 Let $\gamma,\gamma_2:\S{n}\to\R$ be lsc and bounded from below. Then, there exists a minimizer to \eqref{eq:secopt}.
\end{theorem}

The condition that the functions $\gamma, \gamma_2$ are bounded below can easily be relaxed (see \cite[Theorem~4.1]{BeCoHu14}). We will prove this result in Section~\ref{sec:optimExists}.

Our main result is the monotonicity principle, Theorem~\ref{thm:monprin2}, which is a \emph{geometric} characterisation of optimizers $\hat\tau=(\hat\tau_1,\ldots,\hat\tau_n)$ of \eqref{eq:secopt}. The version we state here is weaker than the result we will prove in Section \ref{sec:monotonicity} but easier to formulate and still sufficient for our intended applications.

For two families of increasing stopping times $(\sigma_j)_{j=i}^n$ and $(\tau_j)_{j=i}^n$ with $\tau_i=0$ we define
$$k:=\inf\{j\geq i: \tau_{j+1}\geq \sigma_j\}$$
and stopping times
\begin{align*}
\tilde \sigma_j = \left\{\begin{array}{l}
                   \tau_j \text{ if } j\leq k\\
\sigma_j \text{ if } j>k
                  \end{array}\right.
\end{align*}
and analogously 
\begin{align*}
\tilde \tau_j = \left\{\begin{array}{l}
                   \sigma_j \text{ if } j\leq k\\
\tau_j \text{ if } j>k
                  \end{array}\right..
\end{align*}
Note that $(\tilde \sigma_j)_{j=i}^n$ and $(\tilde\tau_j)_{j=i}^n$ are again two families of increasing stopping times, since 
\begin{align}\label{eq:tildetau}
& \tilde\tau_i=\sigma_i\nonumber\\
\leq~ & \tilde\tau_{i+1}=\1_{\tau_{i+1}\geq\sigma_i}\tau_{i+1} + \1_{\tau_{i+1}<\sigma_i} \sigma_{i+1}\nonumber \\ \leq ~ &\tilde\tau_{i+2}=\1_{\tau_{i+1}\geq\sigma_i}\tau_{i+2} + \1_{\tau_{i+1}<\sigma_i} (\1_{\tau_{i+2}\geq\sigma_{i+1}}\tau_{i+2} + \1_{\tau_{i+2}<\sigma_{i+1}} \sigma_{i+2}) \\
\leq~ & \tilde\tau_{i+3}=\ldots,\nonumber
\end{align}
 and similarly for $\tilde\sigma_j$.

\begin{definition}
  A pair $((f,s_1,\ldots, s_{i-1},s), (g,t_1,\ldots,t_{i-1},t))\in\S{i}\times\S{i}$ constitutes an $i$-th stop-go pair, written $((f,s_1,\ldots, s_{i-1},s), (g,t_1,\ldots, t_{i-1},t))\in\SG_i$, if $f(s)=g(t)$ and for all families of $\cF^B$-stopping times $\sigma_{i}\leq \ldots\leq \sigma_n$, $0=\tau_i\leq\tau_{i+1}\leq \ldots\leq \tau_n$  satisfying $0<\E[\sigma_j]<\infty$ for all $i\leq j\leq n$ and $0\le\E[\tau_j]<\infty$ for all $i<j\leq n$
\begin{align*}
&\E[\gamma\left(( (f\oplus B)_{u})_{u\leq s + \sigma_n}, s_1, \ldots, s_{i-1}, s+ \sigma_{i},s+ \sigma_{i+1},\ldots, s+ \sigma_{n} \right)]
\\
+&\ \E[\gamma\left(( (g\oplus B)_{u})_{u\leq t + \tau_n}, t_1\ , \ldots, t_{i-1}\ , t\phantom{+ \sigma_{i} }\ , t\ + \tau_{i+1}\ ,\ldots, t\ + \tau_{n}\ \right)] &\ 
\\
>\ &\E[\gamma\left(( (f\oplus B)_{u})_{u\leq s + \tilde\sigma_n}, s_1, \ldots, s_{i-1}, s \phantom{+ \sigma_{i}},s+ \tilde\sigma_{i+1},\ldots, s+ \tilde\sigma_{n} \right)]
\numberthis \label{eq:SG}\\
+&\ \E[\gamma\left(( (g\oplus B)_{u})_{u\leq t + \tilde\tau_n}, t_1 , \ldots, t_{i-1} , t + \tilde\tau_{i} \ , t\ + \tilde\tau_{i+1}\ ,\ldots, t\ + \tilde\tau_{n}\ \right)],
\end{align*}
whenever both sides are well defined and the left hand side is finite. (See Figure~\ref{fig:BadPairs}.)

\begin{figure}[ht]
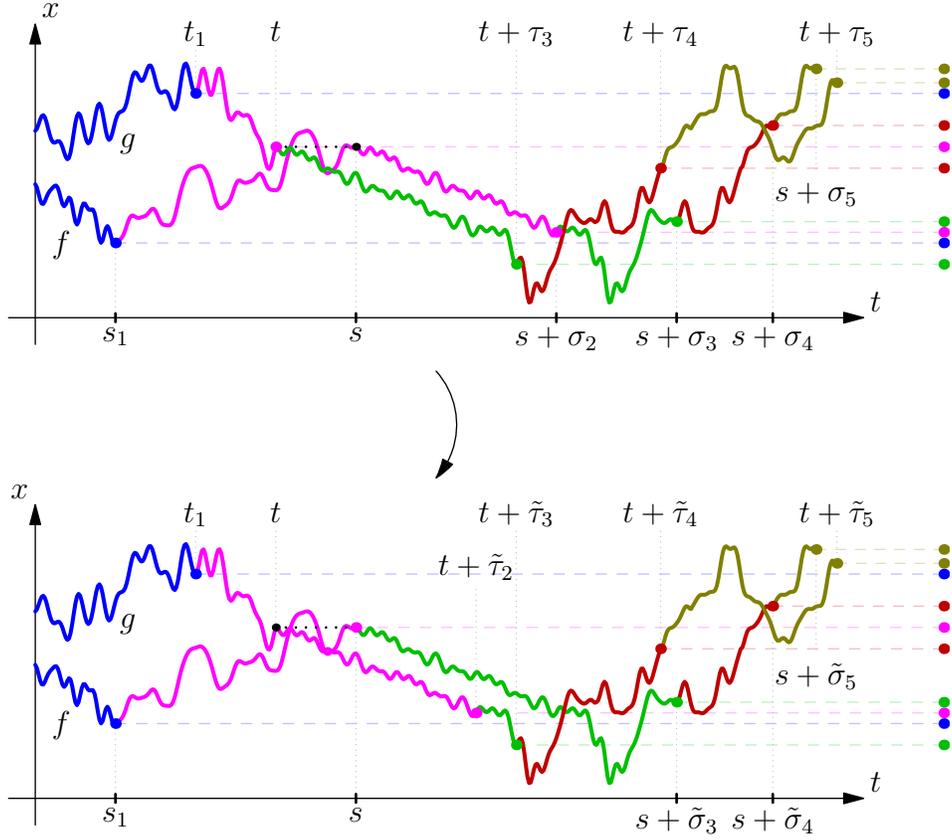

\centering
  \begin{asy}[width=0.99\textwidth]
    import graph;
    import stats;
    import patterns;

    pen ax = black+0.5;
    real eps = 0.25;
    real eps2 = 2.0; // Shift between figures
    real eps3 = 0.04; // Dash sizes

    real xmin = 0;
    real xmax = 2.5;
    real tmax = 7.5;

    //  real[] B1, B2, B3, B4, B5, B6, B7, B8, B9, BX; // Brownian motion
    real[] t; // Time
    int n1, n2, n3, n4, n5, n6, n7, n8, n9, nX; // Time indexes of path starts
    real sig=0.15;
    int N = 500; // Steps
    int K = 8; // ``Smoothness'' of curve. Higher -> rougher

    real dt = tmax/N;

    for(int i = 0; i<=N; ++i)
    {
      t[i] = dt*i;
    }

    path BM_Fourier(int m1, int m2, int K, real xstart, real xfin, real[] t0, real sigma, int seed) {

      srand(seed);

      int N = m2-m1+1;

      real[] gp, gm, b;

      for(int i = 0; i <= K; ++i)
      {
        gp[i] = sigma*Gaussrand();
        gm[i] = sigma*Gaussrand();
      }
      
      // real xfstart, xffin;
      
      real xfstart = 0;
      real xffin = 0;
      for(int i = 0; i <= K; ++i)
      {
        xfstart = xfstart + gp[i]/(i+1)*cos(2*pi*(i+1)*t0[m1]);
        xfstart = xfstart - gm[i]/(i+1)*cos(-2*pi*(i+1)*t0[m1]);

        xffin = xffin + gp[i]/(i+1)*cos(2*pi*(i+1)*t0[m2]);
        xffin = xffin - gm[i]/(i+1)*cos(-2*pi*(i+1)*t0[m2]);
      }
      
      b[0] = xstart;
      b[N-1] = xfin;

      for(int i = 1; i < N; ++i)
      {
        b[i] = 0;
        for(int j = 0; j <= K; ++j)
        {
          b[i] = b[i] + gp[j]/(j+1)*cos(2*pi*(j+1)*t0[m1+i]);
          b[i] = b[i] - gm[j]/(j+1)*cos(2*pi*(j+1)*t0[m1+i]);
        }
        b[i] = b[i] + (xstart -xfstart) + (t0[m1+i]-t0[m1])*(xfin-xffin-xstart+xfstart)/(t0[m2]-t0[m1]);
      }
      
      path P = (t0[m1],b[0]);

      for(int i = 1; i < N; ++i)
      {
        P = P--(t0[m1+i],b[i]);
      }
      return P;
    }
     
    n1 = 0;
    n2 = round(1.0/5*N);
    n3 = round(0.5/5*N);
    n4 = round(2.0/5*N);
    n5 = round(1.5/5*N);
    n6 = round(2.75/5*N);
    n7 = round(3.0/5*N);
    n8 = round(3.5/5*N);
    n9 = round(3.9/5*N);
    nX = round(4.1/5*N);

    path BM1, BM2, BM3, BM4, BM5, BM6, BM7, BM8, BM9, BMX; //

    real x0 = 1.25;
    real x1 = 1.75;
    real x2 = 2.1;
    real x3 = 1.6;
    real x4 = 0.7;
    real x5 = 0.8;
    real x6 = 0.5;
    real x7 = 0.9;
    real x8 = 1.4;
    real x9 = 1.8;
    real xX = 2.2;

    BM1 = BM_Fourier(n1,n2,K,x1,x2,t,sig,102);
    BM2 = BM_Fourier(n2,n5,K,x2,x3,t,sig,123);
    BM3 = BM_Fourier(n1,n3,K,x0,x4,t,sig,98);
    BM4 = BM_Fourier(n3,n4,K,x4,x3,t,sig,24);
    BM5 = BM_Fourier(n5,n6,K,x3,x5,t,sig,110);
    BM6 = BM_Fourier(n6,n7,K,x5,x6,t,sig,80);
    BM7 = BM_Fourier(n7,n8,K,x6,x7,t,sig,70);
    BM8 = BM_Fourier(n8,n9,K,x7,x8,t,sig,60);
    BM9 = BM_Fourier(n9,nX,K,x8,x9,t,sig,50);
    BMX = BM_Fourier(nX,N,K,x9,xX,t,sig,40);

    path tail = subpath(BMX,0,round((N-nX)*0.3));
    pair tail_end = shift(t[n4]-t[n5],0)*point(tail,length(tail));

    //pen BMpen1 = blue+1.5;
    //pen BMpen2 = heavygreen+1.5+dashed;
    //pen BMpen3 = magenta+1.5+dotted;
    //pen BMpen4 = heavyred+1.5+dashdotted;
    //pen BMpen5 = olive+1.5;

    pen BMpen1 = blue+1.5;
    pen BMpen2 = magenta+1.5;
    pen BMpen3 = heavygreen+1.5;
    pen BMpen4 = heavyred+1.5;
    pen BMpen5 = olive+1.5;

    // Shift to lower picture.
    transform T = shift(0,-eps2-xmax+xmin);

    void labdash(real x, real z, string s)
    {
      pair p = (x,0);
      pair q = (x,z);
      label(s,p,S,basealign(1));
      draw(p--q,mediumgray+dotted+0.5);
      draw(shift(0,eps3)*p--shift(0,-eps3)*p,black+1.0);
    }

    void labdash_top(real x, real z, string s)
    {
      pair p = (x,xmax);
      pair q = (x,z);
      label(s,p,(0,1),basealign(1));
      draw(p--q,mediumgray+dotted+0.5);
    }

    void labdash_low(real x, real z, string s)
    {
      pair p = T*(x,0);
      pair q = T*(x,z);
      label(s,p,S,basealign(1));
      draw(p--q,mediumgray+dotted+0.5);
      draw(shift(0,eps3)*p--shift(0,-eps3)*p,black+1.0);
    }

    void labdash_top_low(real x, real z, string s)
    {
      pair p = T*(x,xmax);
      pair q = T*(x,z);
      label(s,p,(0,1),basealign(1));
      draw(p--q,mediumgray+dotted+0.5);
    }

    real eps4 = 1.0;
    real ops = 0.25; // Pen opacity for lines

    // Draw axes in top plot

    draw((-eps,0)--(tmax+eps,0),ax,Arrow);

    draw((0,xmin-eps)--(0,xmax+eps),ax,Arrow);

    label("$t$",(tmax+eps,0),NE);
    label("$x$",(0,xmax+eps),NE);
    //label("$\sigma_1=\sigma_2$",(point(BM,n1).x,xmin),S);
    //label("$\sigma_3$",(point(BM,n3).x,xmin),S);

    // Draw dots on RHS
    
    draw((t[n2],x2)--(tmax+eps4,x2),(BMpen1+0.5+dashed)+opacity(ops));
    dot((tmax+eps4,x2),BMpen1+4.0);

    draw((t[n3],x4)--(tmax+eps4,x4),(BMpen1+0.5+dashed)+opacity(ops));
    dot((tmax+eps4,x4),BMpen1+4.0);
    
    draw((t[n5],x3)--(tmax+eps4,x3),(BMpen2+0.5+dashed)+opacity(ops));
    dot((tmax+eps4,x3),BMpen2+4.0);
    
    draw((t[n6]+(t[n4]-t[n5]),x5)--(tmax+eps4,x5),(BMpen2+0.5+dashed)+opacity(ops));
    dot((tmax+eps4,x5),BMpen2+4.0);
    
    draw((t[n7],x6)--(tmax+eps4,x6),(BMpen3+0.5+dashed)+opacity(ops));
    dot((tmax+eps4,x6),BMpen3+4.0);
    
    draw((t[n8]+(t[n4]-t[n5]),x7)--(tmax+eps4,x7),(BMpen3+0.5+dashed)+opacity(ops));
    dot((tmax+eps4,x7),BMpen3+4.0);
    
    draw((t[n9],x8)--(tmax+eps4,x8),(BMpen4+0.5+dashed)+opacity(ops));
    dot((tmax+eps4,x8),BMpen4+4.0);
    
    draw((t[nX]+(t[n4]-t[n5]),x9)--(tmax+eps4,x9),(BMpen4+0.5+dashed)+opacity(ops));
    dot((tmax+eps4,x9),BMpen4+4.0);
    
    draw((t[N],xX)--(tmax+eps4,xX),(BMpen5+0.5+dashed)+opacity(ops));
    dot((tmax+eps4,xX),BMpen5+4.0);
    
    draw(tail_end--(tmax+eps4,tail_end.y),(BMpen5+0.5+dashed)+opacity(ops));
    dot((tmax+eps4,tail_end.y),BMpen5+4.0);

    // Labels/lines for stopping times in top plot

    labdash_top(t[n2],x2,"$t_1$");
    labdash_top(t[n5],x3,"$t$");
    labdash_top(t[n7],x6,"$t+\tau_3$");
    labdash_top(t[n9],x8,"$t+\tau_4$");
    labdash_top(t[N],xX,"$t+\tau_5$");

    labdash(t[n3],x4, "$s_1$");
    labdash(t[n4],x3, "$s$");
    labdash(t[n6]+(t[n4]-t[n5]),x5, "$s+\sigma_2$");
    labdash(t[n8]+(t[n4]-t[n5]),x7, "$s+\sigma_3$");
    labdash(t[nX]+(t[n4]-t[n5]),x9, "$s+\sigma_4$");

    draw(tail_end--(tail_end.x,1.35),mediumgray+dotted+0.5);
    label("$s+\sigma_5$", (tail_end.x,1.35),S,basealign(1));

    // Draw paths in top picture

    draw(BM1,BMpen1);
    draw(Label("$g$",MidPoint,black+1.0),BM1,opacity(0));
    draw(BM3,BMpen1);
    draw(Label("$f$",MidPoint,SW,black+1.0),BM3,opacity(0));

    draw(BM2,BMpen2);
    draw(BM4--shift(t[n4]-t[n5],0)*BM5,BMpen2);

    draw(BM5--BM6,BMpen3);
    draw(shift(t[n4]-t[n5],0)*(BM6--BM7),BMpen3);

    draw(BM7--BM8,BMpen4);
    draw(shift(t[n4]-t[n5],0)*(BM8--BM9),BMpen4);

    draw(BM9--BMX,BMpen5);
    draw(shift(t[n4]-t[n5],0)*tail,BMpen5);

    //draw(subpath(BM,0,n1),BMpen1);
    //draw(subpath(BM,n1,n2),BMpen2);
    //draw(subpath(BM,n2,n3),BMpen3);
    //draw(subpath(BM,n3,N+1),BMpen4);

    // Add dots to top picture
    draw((t[n5],x3)--(t[n4],x3),dotted+1.0);

    dot((t[n5],x3));
    dot((t[n4],x3));

    dot((t[n2],x2),BMpen1+4.0);
    dot((t[n3],x4),BMpen1+4.0);

    dot((t[n5],x3),BMpen2+4.0);
    dot((t[n6]+(t[n4]-t[n5]),x5),BMpen2+4.0);

    dot((t[n7],x6),BMpen3+4.0);
    dot((t[n8]+(t[n4]-t[n5]),x7),BMpen3+4.0);
    
    dot((t[n9],x8),BMpen4+4.0);
    dot((t[nX]+(t[n4]-t[n5]),x9),BMpen4+4.0);

    dot((t[N],xX),BMpen5+4.0);
    dot(tail_end,BMpen5+4.0);

    // Arrow between plots
    draw((tmax/2,-2*eps)..(tmax*0.52,-eps-eps2/2)..(tmax/2,2*eps-eps2),ax,Arrow);

    // Draw axes and labels for bottom picture

    draw(T*((-eps,0)--(tmax+eps,0)),ax,Arrow);
    draw(T*((0,xmin-eps)--(0,xmax+eps)),ax,Arrow);
    label("$t$",T*(tmax+eps,0),NE);
    label("$x$",T*(0,xmax+eps),NW);

    // Draw dots on RHS, lower picture
    real eps4 = 1.0;
    real ops = 0.25; // Pen opacity for lines
    
    draw(T*((t[n2],x2)--(tmax+eps4,x2)),(BMpen1+0.5+dashed)+opacity(ops));
    dot(T*(tmax+eps4,x2),BMpen1+4.0);

    draw(T*((t[n3],x4)--(tmax+eps4,x4)),(BMpen1+0.5+dashed)+opacity(ops));
    dot(T*(tmax+eps4,x4),BMpen1+4.0);
    
    draw(T*( (t[n4],x3)--(tmax+eps4,x3)),(BMpen2+0.5+dashed)+opacity(ops));
    dot(T*(tmax+eps4,x3),BMpen2+4.0);
    
    draw(T*((t[n6],x5)--(tmax+eps4,x5)),(BMpen2+0.5+dashed)+opacity(ops));
    dot(T*(tmax+eps4,x5),BMpen2+4.0);
    
    draw(T*((t[n7],x6)--(tmax+eps4,x6)),(BMpen3+0.5+dashed)+opacity(ops));
    dot(T*(tmax+eps4,x6),BMpen3+4.0);
    
    draw(T*( (t[n8]+(t[n4]-t[n5]),x7)--(tmax+eps4,x7)),(BMpen3+0.5+dashed)+opacity(ops));
    dot(T*(tmax+eps4,x7),BMpen3+4.0);
    
    draw(T*( (t[n9],x8)--(tmax+eps4,x8)),(BMpen4+0.5+dashed)+opacity(ops));
    dot(T*(tmax+eps4,x8),BMpen4+4.0);
    
    draw(T*( (t[nX]+(t[n4]-t[n5]),x9)--(tmax+eps4,x9)),(BMpen4+0.5+dashed)+opacity(ops));
    dot(T*(tmax+eps4,x9),BMpen4+4.0);
    
    draw(T*( (t[N],xX)--(tmax+eps4,xX)),(BMpen5+0.5+dashed)+opacity(ops));
    dot(T*(tmax+eps4,xX),BMpen5+4.0);
    
    draw(T*( tail_end--(tmax+eps4,tail_end.y)),(BMpen5+0.5+dashed)+opacity(ops));
    dot(T*(tmax+eps4,tail_end.y),BMpen5+4.0);

    // Stopping times labels, lower picture

    labdash_top_low(t[n2],x2,"$t_1$");
    labdash_top_low(t[n5],x3,"$t$");
    labdash_top_low(t[n7],x6,"$t+\tilde{\tau}_3$");
    labdash_top_low(t[n9],x8,"$t+\tilde{\tau}_4$");
    labdash_top_low(t[N],xX,"$t+\tilde{\tau}_5$");
    
    labdash_low(t[n3],x4, "$s_1$");
    labdash_low(t[n4],x3, "$s$");
    //labdash_low((t[n6]+(t[n4]-t[n5])),x5, "$s+\tilde{\sigma}_2$");
    labdash_low((t[n8]+(t[n4]-t[n5])),x7, "$s+\tilde{\sigma}_3$");
    labdash_low((t[nX]+(t[n4]-t[n5])),x9, "$s+\tilde{\sigma}_4$");     

    draw(T*((t[n6],x5)--(t[n6],2.0)),mediumgray+dotted+0.5);
    label("$t+\tilde{\tau}_2$", T*(t[n6],2.0),(0,1),basealign(1));

    draw(T*(tail_end--(tail_end.x,1.35)),mediumgray+dotted+0.5);
    label("$s+\tilde{\sigma}_5$", T*(tail_end.x,1.35),S,basealign(1));

    // Draw paths in bottom plot
    draw(T*BM1,BMpen1);
    draw(Label("$g$",MidPoint,black+1.0),T*BM1,opacity(0));
    draw(T*BM3,BMpen1);
    draw(Label("$f$",MidPoint,SW,black+1.0),T*BM3,opacity(0));

    draw(T*(BM2--BM5),BMpen2);
    draw(T*(BM4),BMpen2);

    draw(T*(BM6),BMpen3);
    draw(T*(shift(t[n4]-t[n5],0)*(BM5--BM6--BM7)),BMpen3);

    draw(T*(BM7--BM8),BMpen4);
    draw(T*(shift(t[n4]-t[n5],0)*(BM8--BM9)),BMpen4);

    draw(T*(BM9--BMX),BMpen5);
    draw(T*(shift(t[n4]-t[n5],0)*tail),BMpen5);

    draw(T*((t[n5],x3)--(t[n4],x3)),dotted+1.0);

    dot(T*(t[n5],x3));
    //dot(T*(t[n4],x3));

    // Draw dots in bottom picture
    dot(T*(t[n2],x2),BMpen1+4.0);
    dot(T*(t[n3],x4),BMpen1+4.0);

    dot(T*((t[n4]),x3),BMpen2+4.0);
    dot(T*(t[n6],x5),BMpen2+4.0);
    
    dot(T*((t[n7]),x6),BMpen3+4.0);
    dot(T*(t[n8]+(t[n4]-t[n5]),x7),BMpen3+4.0);
    
    dot(T*(t[n9],x8),BMpen4+4.0);
    dot(T*(t[nX]+(t[n4]-t[n5]),x9),BMpen4+4.0);

    dot(T*(t[N],xX),BMpen5+4.0);
    dot(T*tail_end,BMpen5+4.0);
  \end{asy}
  \caption{\label{fig:BadPairs}We show a potential `Bad Pair'. In the top picture, we show the pair $((f,s_1, s),(g,t_1,t))$ with corresponding stopping times $\tau_3, \dots, \tau_5$ and $\sigma_2, \dots, \sigma_5$. In the bottom picture, the stopping times $\tilde{\tau}_2, \dots \tilde{\tau}_5$ and $\tilde{\sigma}_3, \dots, \tilde{\sigma_5}$ are shown. Note that the first time that the stopping rules can `revert' to their original times are $\tilde{\tau}_3$ and $\tilde{\sigma}_3$.}
\end{figure}

A pair $((f,s_1,\ldots, s_{i-1},s), (g,t_1,\ldots,t_{i-1},t)) \in\S{i}\times\S{i}$ constitutes a secondary $i$-th stop-go pair, written $((f,s_1,\ldots, s_{i-1},s), (g,t_1,\ldots,t_{i-1},t)) \in\SG_{2,i}$, if $f(s)=g(t)$ and for all families of $\F^B$-stopping times $\sigma_{i}\leq \ldots\leq \sigma_n$, $0=\tau_i\leq\tau_{i+1}\leq \ldots\leq \tau_n$  satisfying $0<\E[\sigma_j]<\infty$ for all $i\leq j\leq n$ and $0\le\E[\tau_j]<\infty$ for all $i<j\leq n$ the inequality \eqref{eq:SG} holds with $\geq$ and if there is equality we have
\begin{align*}
&\E[\gamma_2\left(( (f\oplus B)_{u})_{u\leq s + \sigma_n}, s_1, \ldots, s_{i-1}, s+ \sigma_{i},s+ \sigma_{i+1},\ldots, s+ \sigma_{n} \right)]
\\
+&\ \E[\gamma_2\left(( (g\oplus B)_{u})_{u\leq t + \tau_n}, t_1\ , \ldots, t_{i-1}\ , t\phantom{+ \sigma_{i} }\ , t\ + \tau_{i+1}\ ,\ldots, t\ + \tau_{n}\ \right)] &\ 
\\
>\ &\E[\gamma_2\left(( (f\oplus B)_{u})_{u\leq s + \tilde\sigma_n}, s_1, \ldots, s_{i-1}, s \phantom{+ \sigma_{i}},s+ \tilde\sigma_{i+1},\ldots, s+ \tilde\sigma_{n} \right)]
\numberthis \label{eq:SG2}\\
+&\ \E[\gamma_2\left(( (g\oplus B)_{u})_{u\leq t + \tilde\tau_n}, t_1 , \ldots, t_{i-1} , t + \tilde\tau_{i} \ , t\ + \tilde\tau_{i+1}\ ,\ldots, t\ + \tilde\tau_{n}\ \right)],
\end{align*}
 whenever both sides are well defined and the left hand side (of \eqref{eq:SG2}) is finite.
\end{definition}

For $0\leq i<j\leq n$ we define $\proj_{\S{i}}:\S{j}\to\S{i}$ by $(f,s_1,\ldots,s_j)\mapsto (f_{\llcorner [0,s_i]},s_1,\ldots,s_i)$ where we take $s_0=0$, $\S{0}=\R$, and $f_{\llcorner [0,0]}:=f(0)\in\R$.

\begin{definition}
 A set $\Gamma=(\Gamma_1,\ldots,\Gamma_n)$ with $\Gamma_i\subset\S{i}$ for each $i$ is called $\gamma_2|\gamma$-monotone if for each $1\leq i \leq n$ 
$$ \SG_{2,i}\cap (\Gamma_i^<\times\Gamma_i)=\emptyset,$$
where
$$\Gamma_i^<=\{(f,s_1,\ldots,s_{i-1},u): \text{ there exists } (g,s_1,\ldots,s_{i-1},s)\in\Gamma_i, s_{i-1}\leq u <s, g_{\llcorner [0,u]}=f\},$$
and $\proj_{\S{i-1}}(\Gamma_i)\subset \Gamma_{i-1}$.
\end{definition}

\begin{theorem}[Monotonicity principle]\label{thm:monprin2}
 Let $\gamma,\gamma_2:\S{n}\to\R$ be Borel measurable, $B$ be a Brownian motion on some stochastic basis $(\Omega,\F,(\F_t)_{t\geq 0},\P)$ with $B_0 \sim \mu_0$ and let $\hat\tau=(\hat\tau_1,\ldots,\hat\tau_n)$ be an optimizer of \eqref{eq:secopt}. Then there exists a $\gamma_2|\gamma$-monotone set $\Gamma=(\Gamma_1,\ldots,\Gamma_n)$ supporting $\hat\tau$ in the sense that $\P$- a.s.\ for all $1\leq i\leq n$ 
\begin{align}\label{eq:monprin2}
 ((B_s)_{s\leq\tau_i},\tau_1,\ldots,\tau_i)\in\Gamma_i.
\end{align}
\end{theorem}

\begin{remark}
We will also consider ternary or j-ary optimization problems given $j$ Borel measurable functions $\gamma_1,\ldots,\gamma_j:\S{n}\to\R$ leading to ternary or j-ary $i$-th stop-go pairs $\SG_{i,3},\ldots,\SG_{i,j}$ for $1\leq i \leq n,$ the notion of $\gamma_j|\ldots|\gamma_1$-monotone sets and a corresponding monotonicity principle. To save (digital) trees we leave it to the reader to write down the corresponding definitions.
\end{remark}

\subsection{New $n$-marginal embeddings}

\subsubsection{The $n$-marginal Root embedding}

The classical Root embedding \cite{Ro69} establishes the existence of a barrier (or right-barrier) $\mathcal R\subset \R_+\times\R$ such that the first hitting time of $\mathcal R$ solves the Skorokhod embedding problem. A barrier $\mathcal R$ is a Borel set such that $(s,x)\in\mathcal R \Rightarrow (t,x)\in\mathcal R$ for all $t>s$. Moreover, the Root embedding has the property that it minimises $\E[h(\tau)]$ for a strictly convex function $h:\R_+\to\R$ over all solutions to the Skorokhod embedding problem, cf.\ \cite{Ro76}.

 We will show that there is a unique $n$- marginal Root embedding in the sense that there are $n$ barriers $(\mathcal R^i)_{i=1}^n$ such that for each $i\leq n$ the first hitting time of $\mathcal R^i$ \emph{after} hitting $\mathcal R^{i-1}$ embeds $\mu_i$.

\begin{theorem}[$n$-marginal Root embedding, c.f.\ \cite{CoObTo15}]\label{thm:Root1}
 Put $\gamma_i:\S{n}\to\R, (f,s_1,\ldots,s_n)\mapsto h(s_i)$ for some strictly convex function $h:\R_+\to\R$ and assume that \eqref{eq:Pgamma} is well posed. Then there exist $n$ barriers $(\mathcal R^i)_{i=1}^n$ such that defining
\begin{align*}
& \tau^{Root}_1(\omega)=\inf\{t\geq 0: (t,B_t(\omega))\in \mathcal R^1\}
\end{align*}
and for $1<i\leq n$
\begin{align*}
 \tau^{Root}_i(\omega)=\inf\{t\geq \tau^{Root}_{i-1}(\omega) : (t,B_t(\omega))\in \mathcal R^i\}
\end{align*}
the multi-stopping time $(\tau^{Root}_1,\ldots,\tau^{Root}_n)$
minimises
$$\E[h(\tilde\tau_i)]$$ 
simultaneously for all $1\leq i\leq n$ among all increasing families of stopping times $(\tilde \tau_1,\ldots,\tilde\tau_n)$ such that $B_{\tilde \tau_j}\sim\mu_j$ for all $1\leq j\leq n$. This solution is unique in the sense that for any solution $\tilde \tau_1, \ldots, \tilde \tau_n$ of such a barrier-type we have $\tau^{Root}_i=\tilde \tau_i$ a.s.
\end{theorem}
\begin{proof}
Fix a permutation $\kappa$ of $\{1,\ldots,n\}$. We consider the functions $\tilde\gamma_1=\gamma_{\kappa(1)},\ldots,\tilde\gamma_n=\gamma_{\kappa(n)}$ on $\S{n}$  and the corresponding family of $n$-ary minimisation problems, ($\mathsf{OptMSEP}_n$). Let $(\tau^{Root}_1,\ldots,\tau^{Root}_n)$ be an optimiser of $P_{\tilde\gamma_n|\ldots|\tilde\gamma_1}$. By the $n$-ary version of Theorem~\ref{thm:exists}, choose an optimizer $(\tau^{Root}_1,\ldots,\tau^{Root}_n)$ of ($\mathsf{OptMSEP}_n$) and, by the corresponding version of Theorem~\ref{thm:monprin2},  a $\tilde\gamma_n|\ldots|\tilde\gamma_1$-monotone family of sets  $(\Gamma_1,\ldots,\Gamma_n)$ supporting $(\tau^{Root}_1,\ldots,\tau^{Root}_n)$. Hence for every $i\leq n$ we have  $\P$-a.s.
$$((B_s)_{s\leq\tau_i},\tau^{Root}_1,\ldots,\tau^{Root}_i)\in\Gamma_i,$$
 and 
$$(\Gamma^{<}_i\times \Gamma_i)\cap\SG_{i,n}=\emptyset.$$
We claim that, for all $1\leq i\leq n$ we have
$$\SG_{i,n}\supset\{\left((f,s_1,\ldots,s_i),(g,t_1,\ldots,t_i)\right): f(s_i)=g(t_i), s_i>t_i\}.$$
Fix $(f,s_1,\ldots,s_i),(g,t_1,\ldots,t_i)\in \S{i}$ satisfying $s_i>t_i$ and consider two families of stopping times $(\sigma_j)_{j=i}^n$ and $(\tau_j)_{j=i}^n$ on some probability space $(\Omega,\cF,\P)$ together with their modifications $(\tilde\sigma_j)_{j=i}^n$ and $(\tilde\tau_j)_{j=i}^n$ as in Section \ref{sec:probint}. Put
$$j_1:=\inf\{m\geq 1 : \kappa(m) \geq i\}$$
and inductively for $1< a\leq n-i+1$
$$ j_a:=\inf\{m\geq j_{a-1} : \kappa(m)\geq i\}.$$
Let $ l=\argmin\{a: \P[\sigma_{j_a}\neq \tilde\sigma_{j_a}]>0\}.$
By the definition of $\tilde \sigma_j$ and $\tilde\tau_j$ we have in case of $j_l=i$ the equality $\{\sigma_{j_l}\neq\tilde\sigma_{j_l}\} = \Omega$ and for $j_l>i$ it holds that
$$ \{\sigma_{j_l}\neq\tilde\sigma_{j_l}\} = \bigcap_{i\leq k <j_l} \{\sigma_k > \tau_{k+1}\}.$$
As $\tau_{k}\leq \tau_{k+1}$, in particular, we have on $\{\sigma_{j_l}\neq\tilde\sigma_{j_l}\}$ the inequality $\sigma_k > \tau_k$ for every $i \le k\leq j_l$.
The strict convexity of $h$ and $s>t$ implies
$$ \E[h(s+\sigma_{j_l})] + \E[h(t+\tau_{j_l})]>\E[h(s+\tilde\sigma_{j_l})] + \E[h(t+\tilde\tau_{j_l})]~.$$
Hence, we get a strict inequality in (the corresponding $\kappa^{-1}(j_l)$-ary version of) \eqref{eq:SG2} and the claim is proven.

Then we define for each $1\leq i\leq n$
$$\mathcal R_\cl^i:=\{(s,x)\in\R_+\times\R: \exists (g,t_1,\ldots,t_i)\in \Gamma_i, g(t_i)=x, s\geq t_i\}$$
and
$$\mathcal R_\op^i:=\{(s,x)\in\R_+\times\R: \exists (g,t_1,\ldots,t_i)\in \Gamma_i, g(t_i)=x, s > t_i\}.$$
Following the argument in the proof of Theorem 2.1 in \cite{BeCoHu14}, we define 
$\tau^1_\cl$ and $\tau^1_\op$ to be the first hitting times of $\mathcal R^1_\cl$ and $\mathcal R^1_\op$ respectively to see that actually $\tau^1_\cl\leq\tau^{Root}_1\leq\tau_\op^1$ and $\tau^1_\cl=\tau_\op^1$ a.s. by the strong Markov property. Then we can inductively proceed and define
$$\tau^i_\cl:=\inf\{t\geq \tau^{i-1}_\cl : (t,B_t)\in\mathcal R^i_\cl\}$$
and
$$\tau^i_\op:=\inf\{t\geq \tau^{i-1}_\cl : (t,B_t)\in\mathcal R^i_\op\}.$$
By the very same argument we see that $\tau^i_\cl\leq \tau_i^{Root} \leq \tau^i_\op$ and in fact $\tau^i_\cl=\tau^i_\op.$

Finally, we need to show that the choice of the permutation $\kappa$ does not matter. This follows from a straightforward adaptation of the argument of Loynes \cite{Lo70} (see also \cite[Remark 2.2]{BeCoHu14} and \cite[Proof of Lemma~2.4]{CoObTo15}) to the multi-marginal set up. Indeed, the first barrier $\mathcal R^1$ is unique by Loynes original argument. This implies that the second barrier is unique because Loynes argument is valid for a general starting distribution of the process $(t,B_t)$ in $\R_+\times\R$ and we can conclude by induction. 
\end{proof}

\begin{remark} \label{rk:opt-prob}
  \begin{enumerate}
  \item  In the last theorem, the result stays the same if we take different strictly convex functions $h_i$ for each $i$.
  \item\label{it:opt-prob-permute} Moreover, it is easy to see that the proof is simplified if one starts with the objective $\sum_{i=1}^n h_i(\tau_i)$, which removes the need for taking an arbitrary permutation of the indices at the start. Of course, to get the more general conclusion, one needs to consider these permutations.
  \end{enumerate}
\end{remark}

\begin{corollary}
Let $h:\R_+\to\R$ be a strictly convex function and let $\gamma:\S{n}\to\R, (f,s_1,\ldots,s_n)\mapsto \sum_{i=1}^n h(t_i).$ Let $\tau^{Root}=(\tau^{Root}_1,\ldots,\tau^{Root}_n)$ be the minimizer of Theorem \ref{thm:Root1}. Then it also minimizes 
$$ \E[\gamma(\tilde\tau_1,\ldots,\tilde\tau_n)]$$
among all increasing families of stopping times $\tilde\tau_1\leq\ldots\leq\tilde\tau_n$ satisfying $B_{\tilde\tau_i}\sim\mu_i$ for all $1\leq i\leq n$.
\end{corollary}

\subsubsection{The $n$-marginal Rost embedding}

The classical Rost embedding \cite{Ro76} establishes  the existence of an inverse barrier (or left-barrier) $\mathcal R\subset \R_+\times\R$ such that the first hitting time of $\mathcal R$ solves the Skorokhod embedding problem. An inverse barrier $\mathcal R$ is a Borel set such that $(t,x)\in\mathcal R \Rightarrow (s,x)\in\mathcal R$ for all $s<t$. Moreover, the Rost embedding   has the property that it maximises $\E[h(\tau)]$ for a strictly convex function $h:\R_+\to\R$ over all solutions to the Skorokhod embedding problem, cf.\ \cite{Ro76}. Similarly to the Root embedding it follows that

\begin{theorem}[$n$-marginal Rost embedding]\label{thm:Rost1}
 Put $\gamma_i:\S{n}\to\R, (f,s_1,\ldots,s_n)\mapsto -h(s_i)$ for some strictly convex function $h:\R_+\to\R$ and assume that \eqref{eq:Pgamma} is well posed. Then there exist $n$  inverse barriers $(\mathcal R^i)_{i=1}^n$ such that defining
\begin{align*}
& \tau^{Rost}_1(\omega)=\inf\{t\geq 0: (t,B_t(\omega))\in \mathcal R^1\}
\end{align*}
and for $1<i\leq n$
\begin{align*}
 \tau^{Rost}_i(\omega)=\inf\{t\geq \tau^{Rost}_{i-1}(\omega) : (t,B_t(\omega))\in \mathcal R^i\}
\end{align*}
the multi-stopping time $(\tau^{Rost}_1,\ldots,\tau^{Rost}_n)$
maximises
$$\E[h(\tau_i)]$$ 
simultaneously for all $1\leq i\leq n$ among all increasing families of stopping times $( \tau_1,\ldots,\tau_n)$ such that $B_{ \tau_j}\sim\mu_j$ for all $1\leq j\leq n$. Moreover, it also maximises
$$\sum_{i=1}^n\E\left[ h(\tau_i)\right].$$
This solution is unique in the sense that for any solution $\tilde \tau_1, \ldots, \tilde \tau_n$ of such a barrier-type we have $\tau^{Rost}_i=\tilde \tau_i$.
\end{theorem}

The proof of this theorem goes along the very same lines as the proof of Theorem \ref{thm:Root1}. The only difference is that due to the maximisation we get
$$\SG_{i,n}\supset\{(f,s_1,\ldots,s_i),(g,t_1,\ldots,t_i): f(s_i)=g(t_i), s_i<t_i\}$$
leading to inverse barriers. We omit the details.

\subsubsection{The $n$-marginal Az\'ema-Yor embedding}

For $(f,s_1,\ldots,s_n)\in \S{n}$ we will use the notation $\bar f_{s_i}:=\max_{0\leq s\leq s_i} f(s).$

\begin{theorem}[$n$-marginal Az\'ema-Yor solution]\label{thm:AzemaYor}
 There exist $n$ barriers $(\mathcal R^i)_{i=1}^n$ such that defining
\begin{align*}
& \tau^{AY}_1=\inf\{t\geq 0: (\bar B_t,B_t)\in \mathcal R^1\}
\end{align*}
and for $1<i\leq n$
\begin{align*}
 \tau^{AY}_i=\inf\{t\geq \tau^{AY}_{i-1} : (\bar B_t,B_t)\in \mathcal R^i\}
\end{align*}
the multi-stopping time $(\tau^{AY}_1,\ldots,\tau^{AY}_n)$
maximises
$$\E \left[\sum_{i=1}^n\bar B_{\tau_i}\right]$$ 
among all increasing families of stopping times $( \tau_1,\ldots,\tau_n)$ such that $B_{ \tau_j}\sim\mu_j$ for all $1\leq j\leq n$.
This solution is unique in the sense that for any solution $\tilde \tau_1, \ldots, \tilde \tau_n$ of such a barrier-type we have $\tau^{AY}_i=\tilde \tau_i$.
\end{theorem}

We emphasise that this result has not appeared previously in the literature in this generality; previously the most general result was due to \cite{HLOb12} and \cite{ObSp13}, which proved a closely related result under an additional condition on the measures, which is not necessary here.

\begin{remark}
 In fact, similarly to the $n$-marginal Root and Rost solutions $\tau^{AY}$ simultaneously solves the optimization problems
$$\sup\{\E [\bar B_{\tilde\tau_i}]: \tilde\tau_1\leq\ldots \leq \tilde\tau_n, B_{\tilde\tau_1}\sim \mu_1, \ldots,B_{\tilde\tau_n}\sim \mu_n\}$$
for each $i$ which of course implies  Theorem \ref{thm:AzemaYor} (see also Remark~\ref{rk:opt-prob}.\ref{it:opt-prob-permute}). To keep the presentation readable, we only prove the less general version.
\end{remark}

\begin{proof}
Fix a bounded and strictly increasing continuous function $\phi:\R_+\to\R_+$ and consider the  continuous functions $\gamma(f,s_1,\ldots,s_n)=-\sum_{i=1}^n \bar f_{s_i}$ and $\tilde\gamma(f,s_1,\ldots,s_n)=\sum_{i=1}^n \phi(\bar f_{s_i})f(s_i)^2$ defined on $\S{n}$.  Pick, by Theorem \ref{thm:exists}, a minimizer $\tau^{AY}$ of \eqref{eq:secopt} and, by Theorem \ref{thm:monprin2}, a $\tilde\gamma|\gamma$-monotone family of sets $(\Gamma_i)_{i=1}^n$  supporting $\tau^{AY}=(\tau^{AY}_i)_{i=1}^n$ such that for all $i\leq n$
$$\SG_{i,2}\cap(\Gamma_i^{<}\times\Gamma_i)=\emptyset.$$
We claim that
\begin{align}\label{eq:SGAY}
\SG_{i,2}\supset \{ ((f,s_1,\ldots,s_i),(g,t_1,\ldots,t_i))\in \S{i}\times \S{i}: f(s_i)=g(t_i), \bar f_{s_i}>\bar g_{t_i} \}.
\end{align}
Indeed, pick $((f,s_1,\ldots,s_i),(g,t_1,\ldots,t_i))\in \S{i}\times \S{i}$ with $f(s_i)=g(t_i)$ and $\bar f_{s_i}>\bar g_{t_i}$ and take two families of stopping times $(\sigma_j)_{j=i}^n$ and $(\tau_j)_{j=i}^n$  together with their modifications $(\tilde\sigma_j)_{j=i}^n$ and $(\tilde\tau_j)_{j=i}^n$ as in Section \ref{sec:probint}. We assume that they live on some probability space $(\Omega,\cF,\P)$ additionally supporting a standard Brownian motion $W$. Observe that (as written out in the proof of Theorem \ref{thm:Root1}) on $\{\sigma_j\neq \tilde\sigma_j\}$ it holds that $\sigma_j>\tau_j$. Hence, on this set we have $\bar W_{\sigma_j}\geq \bar W_{\tau_j}$. This implies that for  $\omega\in\{\sigma_j\neq \tilde\sigma_j\}$ (and hence $\tilde\sigma_j=\tau_j,\tilde\tau_j=\sigma_j$)
\begin{equation}\label{eq:AYinequality}
  \begin{split}
    \bar f_{s_i} 
    & \vee(f(s_i) + \bar W_{\sigma_j}(\omega))  + \bar g_{t_i}\vee(g(t_i) + \bar W_{\tau_j}(\omega))\\
    & \leq \bar f_{s_i}\vee(f(s_i) + \bar W_{\tilde\sigma_j}(\omega))  + \bar g_{t_i}\vee(g(t_i) + \bar W_{\tilde\tau_j}(\omega)),
  \end{split}
\end{equation}
with a strict inequality unless either $\bar W_{\sigma_j}(\omega)\leq \bar g_{t_i}-g(t_i)$ or $\bar W_{\tau_j}\geq \bar f_{s_i}-f(s_i).$ 
On  the set $\{\sigma_j=\tilde\sigma_j\}$ we do not change the stopping rule for the $j$-th stopping time and hence we get a (pathwise) equality in \eqref{eq:AYinequality}.
Thus, we always have a strict  inequality in \eqref{eq:SG2} unless a.s.\ either $\bar W_{\sigma_j}(\omega)\leq \bar g_{t_i}-g(t_i)$ or  $\bar W_{\tau_j}\geq \bar f_{s_i}-f(s_i)$ for all $j$. However, in that case
we have for all $j$ such that $\P[\sigma_j\neq\tilde\sigma_j]>0$ (there is at least one such $j$, namely $j=i$)
\begin{align*}
& \E\left[\phi(\bar f_{s_i})(f(s_i)+W_{\sigma_j})^2\right]  +  \E\left[\phi(\bar g_{t_i})(g(t_i) +  W_{\tau_j})^2\right]\\
 >\ & \E\left[\phi(\bar f_{s_i})(f(s_i) +  W_{\tilde\sigma_j})^2\right]  +  \E\left[\phi(\bar g_{t_i})(g(t_i) + W_{\tilde\tau_j})^2\right].
\end{align*}
Hence, $((f,s_1,\ldots,s_i),(g,t_1,\ldots,t_i))\in\SG\subset \SG_2$ in the first case and in the second case we have  $((f,s_1,\ldots,s_i),(g,t_1,\ldots,t_i))\in\SG_2$  proving \eqref{eq:SGAY}.

For each $i\leq n$ we define 
$$\mathcal R^i_{\op}:=\{(m,x):\exists (f,s_1,\ldots,s_i)\in\Gamma_i, f(s_i)=x, \bar f_{s_i}<m\}$$
and
$$\mathcal R^i_{\cl}:=\{(m,x):\exists (f,s_1,\ldots,s_i)\in\Gamma_i, f(s_i)=x, \bar f_{s_i}\leq m\}$$
with respective hitting times ($\tau^0=0$)
$$\tau^i_\op:=\inf\{t\geq \tau^{i-1}_\cl: (\bar B_t,B_t)\in\mathcal R^i_\op \}$$
 and 
$$\tau^i_\cl:=\inf\{t\geq \tau^{i-1}_\cl: (\bar B_t,B_t)\in\mathcal R^i_\cl \}.$$ We will show inductively on $i$ that firstly $\tau^i_\cl\leq \tau_i^{AY} \leq \tau^i_\op$ a.s.\ and secondly $\tau^i_\cl=\tau^i_\op$ a.s.\ proving the theorem. The case $i=1$ has been settled in \cite{BeCoHu14}. So let us assume $\tau^{i-1}_\cl=\tau^{i-1}_\op$ a.s. Then $\tau^i_\cl\leq\tau_i^{AY}$ follows from the definition of $\tau^i_\cl$.  To show that $\tau^{AY}_i\leq \tau_{\op}^i$ pick $\omega$ satisfying $((B_s(\omega))_{s\leq\tau_i^{AY}},\tau_1^{AY}(\omega),\ldots,\tau_i^{AY}(\omega))\in\Gamma_i$ and assume that $\tau^i_\op(\omega)<\tau_i^{AY}(\omega).$ Then there exists $s\in \big[\tau^i_\op(\omega),\tau_i^{AY}(\omega)\big)$ such that $f:=(B_r(\omega))_{r\leq s}$ satisfies $(\bar f, f(s))\in \mathcal{R}^i_\op$. Since $\tau^{i-1}_\cl(\omega)\leq\tau^i_\op(\omega)\leq s < \tau_i^{AY}(\omega)$ we have $(f,\tau^1_\cl(\omega),\ldots,\tau^{i-1}_\cl(\omega),s)\in \Gamma_i^{<}$. By definition of $\mathcal{R}^i_\op$, there 
exists $(g,t_
1,\ldots,t_i)\in \Gamma_i$ such that $f(s)= g(t_i)$ and $\bar g_{t_i} < \bar f_{s}$, yielding a contradiction to \eqref{eq:SGAY}.

Finally, we need to show that $\tau^i_\cl=\tau^i_\op$ a.s. 
 Before we proceed we give a short reminder of the case $i=1$ from \cite[Theorem 6.5]{BeCoHu14}. We define
\begin{equation*}
  \tilde\psi^1_0(m) = \sup\{x: \exists (m,x) \in \mathcal{R}^1_\cl\}.
\end{equation*}
From the definition of $\mathcal{R}^1_\cl$, we see that $\tilde\psi^1_0(m)$ is increasing, and we define the right-continuous function $\psi^1_+(m) = \tilde\psi^1_0(m+)$, and the left-continuous function $\psi^1_-(m) = \tilde\psi^1_0(m-)$. It follows from the definitions of $\tau^1_{\op}$ and $\tau^1_{\cl}$ that:
\begin{equation*}
  \tau_+ := \inf\{t \ge 0: B_t \le \psi^1_+(\ol{B}_t)\} \le \tau^1_{\cl} \le \tau^1_{\op} \le  \inf\{t \ge 0: B_t < \psi^1_-(\ol{B}_t)\} =: \tau_-.
\end{equation*}
As $\tilde\psi^1_0$ has at most countably many jump points (discontinuity points) it is easily checked that $\tau_- = \tau_+$ a.s.\ and hence $\tau^1_\cl=\tau^1_\op= \tau_1^{AY}$. Note also that the law $\bar \mu^1$ of $\bar B_{\tau_1^{AY}}$ can have an atom only at the rightmost point of its support. Hence, with $\pi^1:= \law(\bar B_{\tau_1^{AY}}, B_{\tau_1^{AY}})$, the measure $\pi^1_{\llcorner \{(x,y): y<x\}}$ has continuous projection onto the horizontal axis. 

Defining these quantities in obvious analogy for $j\in \{2, \ldots, n\}$, we need to prove $\tau^{i+1}_\cl=\tau^{i+1}_\op= \tau_{i+1}^{AY}$ assuming  that  $\pi^{i}$  has continuous projection onto the horizontal axis.  
To do so, we decompose $\pi^i$ into \emph{free} and \emph{trapped} particles $$ \pi^{i}_f:=\pi^i_{\llcorner \{(m,x):x > \psi^i_-(m)\}}, 
\quad \pi^{i}_t:=\pi^i_{\llcorner \{(m,x):x \leq \psi^i_-(m)\}}.$$

Here $\pi^i_f$ refers to particles which are free to reach a new maximum, while $\pi^i_t$ refers to particles which are trapped in the sense that they will necessarily hit $\mathcal R _\op^i$ (and thus also $\mathcal R _\cl^i$) before they reach a new maximum. For particles started in $\pi^i_f$ it follows precisely as above that the hitting times of $\mathcal R _\op^{i+1}$ and $\mathcal R _\cl^{i+1}$ agree.  For particles started in $\pi^i_t$ this is a consequence of Lemma~\ref{BarrierEquality}. Additionally, as above we find  that  $\pi^{i+1}_{\llcorner \{(x,y): y<x\}}$  has continuous projection onto the horizontal axis.  
\end{proof}

\begin{lemma}{\cite[Lemma~3.2]{HuSt16}}\label{BarrierEquality}
Let $\mu$ be a probability measure on $\R^2$ such that the projection onto the horizontal axis $\proj_x \mu$ is continuous (in the sense of not having atoms) and let $\psi: \R\to \R$ be a Borel function. Set 
$$R^\op:= \{(x,y): x>\psi(y)\},\quad R^\cl:= \{(x,y): x\geq\psi(y)\}.$$
Start a vertically moving Brownian motion $B$ in $\mu$ and define
$$ \tau^\op:= \inf\{ t\geq 0: (x, y+B_t)\in R^\op\}, \quad \tau^\cl:= \inf\{ t\geq 0: (x, y+B_t)\in R^\cl\}.$$
Then $\tau^\cl=\tau^\op$ almost surely.
\end{lemma}

\subsubsection{The $n$-marginal Perkins/Hobson-Pedersen embedding}

For $(f,s_1,\ldots,s_n)\in \S{n}$ we will use the notation
$\ul f_{s_i}:=\min_{0\leq s\leq s_i} f(s)$ to denote the
running minimum of the path up to time $s_i$. (Recall also that
$\bar f_{s_i}$ is the maximum of the path). In this section we will
consider a generalisation of the embeddings of Perkins \cite{Pe86} and
Hobson and Pedersen \cite{HoPe02}. The construction of Perkins to
solve the one-marginal problem with a trivial starting law can be
shown to simultaneously minimise $\E[h(\bar{B}_\tau)]$ for any
increasing function $h$, and maximise $\E[k(\ul{B}_\tau)]$ for any
increasing function $k$, over all solutions of the embedding
problem. Later Hobson and Pedersen \cite{HoPe02} described a closely
related construction which minimised $\E[h(\bar{B}_\tau)]$ over all
solutions to the SEP with a general starting law. The solution of
Perkins took the form:
\begin{equation*}
  \tau^P := \inf \{ t \ge 0: B_t \le \gamma_-(\bar{B}_t) \text{ or } B_t \ge \gamma_+(\ul{B}_t)\}
\end{equation*}
for decreasing functions $\gamma_+, \gamma_-$. Hobson and Pedersen
constructed, for the case of a general starting distribution, a
stopping time
\begin{equation*}
  \tau^{HP} := \inf \{ t \ge 0: B_t \le \gamma_-(\bar{B}_t) \text{ or } \ol{B}_t \ge G\}
\end{equation*}
where $G$ was an appropriately chosen, $\F_0$-measurable random
variable. They showed that $\tau^{HP}$ minimised $\E[h(\bar{B}_\tau)]$
for any increasing function, but it is clear that the second
minimisation does not hold in general. In \cite[Remark~2.3]{HoPe02},
the existence of a version of Perkins' construction for a general
starting law is conjectured. Below we will show that the construction
of Hobson and Pedersen can be generalised to the multi-marginal case,
and sketch an argument that there are natural generalisations of the
Perkins embedding to this situation, but argue that
there is no `canonical' generalisation of the Perkins embedding. To be
more specific, for given increasing functions $h,k$, the embedding(s)
which maximise $\E[k(\ul{B}_{\tau_n})]$ over all solutions to the
multi-marginal embedding problem which minimise
$\E[h(\bar{B}_{\tau_n})]$ will typically differ from the embeddings
which minimise $\E[h(\bar{B}_{\tau_n})]$ over all maximisers of
$\E[k(\ul{B}_{\tau_n})]$.

\begin{theorem}[$n$-marginal `Hobson-Pedersen'
  solution]\label{thm:Perkins}

 There exist $n$ left-barriers $(\mathcal R^i)_{i=1}^n$ and stopping
  times $\tau_1^*\le \tau_2^* \le \dots \le \tau_n^*$ where
  $\tau_i^*<\infty$ implies $B_{\tau^*_i} = \bar{B}_{\tau_i^*}$ such that
  \begin{align*}
    & \tau^{HP}_1 = \inf\{t\geq 0: (\bar B_t, B_t)\in \mathcal R^1\}\wedge \tau_1^*
  \end{align*}
  and for $1<i\leq n$
  \begin{align*}
    \tau^{HP}_i=\inf\{t\geq \tau^{HP}_{i-1} : (\bar B_t, B_t)\in \mathcal R^i\} \wedge \tau_i^*
  \end{align*}
  the multi-stopping time $(\tau^{HP}_1,\ldots,\tau^{HP}_n)$ minimises
  $$\E \left[\sum_{i=1}^n\bar B_{\tau_i}\right]$$ 
  among all increasing families of stopping times
  $(\tau_1,\ldots,\tau_n)$ such that $B_{ \tau_j}\sim\mu_j$ for all
  $1\leq j\leq n$.
\end{theorem}

\begin{proof}
  Fix a bounded and strictly increasing continuous function
  $\phi:\R_+\to\R_+$ and consider the  continuous
  functions $\gamma(f,s_1,\ldots,s_n)=\sum_{i=1}^n \bar f_{s_i}$,
  $\gamma_2(f,s_1,\ldots,s_n)=\sum_{i=1}^n \phi(\bar f_{s_i})
  f(s_i)^2$ defined on $\S{n}$.
  Pick, by Theorem \ref{thm:exists}, a minimizer $\tau^{HP}$ of \eqref{eq:secopt} and, by Theorem~\ref{thm:monprin2},  a
  $\gamma_2|\gamma$-monotone family of sets $(\Gamma_i)_{i=1}^n$ supporting
  $\tau^{HP}=(\tau^{HP}_i)_{i=1}^n$ such that for all $i\leq n$
  $$\SG_{i,2}\cap(\Gamma_i^{<}\times\Gamma_i)=\emptyset.$$
  By an essentially identical argument to that given in
  Theorem~\ref{thm:AzemaYor}, we have
  \begin{align}\label{eq:SGP}
    \SG_{i,2}\supset \left\{
    ((f,s_1,\ldots,s_i),(g,t_1,\ldots,t_i))\in \S{i}\times
    \S{i}: f(s_i)=g(t_i), \bar f_{s_i} < \bar g_{t_i} \right\}.
  \end{align}

  Note that, given $\tau^{HP}_i$, we can define stopping times
  $\tau^*_i:= \tau^{HP}_i$ if $B_{\tau^{HP}_i} = \bar{B}_{\tau_i^{HP}}$ and
  to be infinite otherwise. 

  For each $i\leq n$ we define 
  $$\mathcal R^i_{\op}:=\left\{(m,x):\exists
    (f,s_1,\ldots,s_i)\in\Gamma_i, f(s_i)=x, \bar f_{s_i} >m, x<m \right\}$$
  and 
   $$\mathcal R^i_{\cl}:=\left\{(m,x):\exists
    (f,s_1,\ldots,s_i)\in\Gamma_i, f(s_i)=x, \bar f_{s_i} \ge m, x<m \right\}$$
  with respective hitting times ($\tau^0=0$)
  $$\tau^i_\op:=\inf\{t\geq \tau^{i-1}_\op: (\bar B_t, B_t)\in\mathcal R^i_\op \}$$
  and 
  $$\tau^i_\cl:=\inf\{t\geq \tau^{i-1}_\cl: (\bar B_t, B_t)\in\mathcal
  R^i_\cl \}.$$
  It can be shown inductively on $i$ that firstly
  $\tau^i_\cl \wedge \tau^*_i\leq \tau_i^{HP} \leq \tau^i_\op\wedge
  \tau_i^*$
  a.s., and secondly $\tau^i_\cl\wedge \tau^*_i=\tau^i_\op\wedge \tau^*_i$ a.s., proving the
  theorem. The proofs of these results are now essentially identical to
  the proof of Theorem~\ref{thm:AzemaYor}.
\end{proof}

Of course, as before, a more general version of the statement (without the summation) can be proved, at the expense of a more complicated argument.

\begin{remark}
  The result above says nothing about the uniqueness of the
  solution. However the following argument (also used in
  \cite{BeCoHu14}) shows that any optimal solution (to both the
  primary and secondary optimisation problem in the proof of
  Theorem~\ref{thm:Perkins}) will have the same barrier form:
  specifically, suppose that $(\tau^i)$ and $(\sigma^i)$ are both
  optimal. Define a new stopping rule which, at time 0, chooses either the
  stopping rule $(\tau^i)$, or the stopping rule $(\sigma^i)$, each
  with probability $1/2$. This stopping rule is also optimal (for both
  the primary and secondary rules), and the arguments above may be
  re-run to deduce the corresponding form of the optimal
  solution.

  In fact, a more involved argument would appear to give uniqueness of
  the resulting barrier among the class of all such solutions; the
  idea is to use a Loynes-style argument as before, but applied both
  to the barrier and the \emph{rate} of stopping at the maximum. The
  difficulty here is to argue that any stopping times of the form
  given above are essentially equivalent to another stopping time
  which simply stops at the maximum according to some rate which will
  be dependent only on the choice of the lower barrier (that is, in
  the language above,
  $\P(H_x^i < \tau^*_i = \tau^{HP}_i \le H_{x+\eps}^i)$ is
  independent of the choice of $\tau^{HP}_i$ for any $x$ and $\eps>0$, where
  $H_x^i:= \inf \{t \ge \tau^{HP}_{i-1}: B_t \ge x$). By identifying
  each of the possible optimisers with a canonical form of the
  optimiser, and using a Loynes-style argument which combines two
  stopping rules of the form above by taking the maximum of the
  left-barriers, and the fastest stopping rate of the rules, one can
  deduce that there is a unique sequence of barriers and stopping rate
  giving rise to an embedding of this form. We leave the details to
  the interested reader.
\end{remark}

\begin{remark}
  We conclude by considering informally the `Perkins'-type
  construction implied by our methods. Recall that in the single
  marginal case, where $B_0 = 0$, the Perkins embedding simultaneously
  both maximises the law of the minimum, and minimises the law of the
  maximum. A slight variant of the methods above would suggest that
  one could adapt the arguments above to consider the optimiser which
  has the same primary objective as above, and also then aims to
  minimise the law of the minimum. In this case the arguments may be
  run to give stopping regions (for each marginal) which are barriers in the sense that
  it is the first hitting time of a left-barrier $\mathcal R$ which is
  left-closed in the sense that if (for a fixed $x$) a path with
  $\bar{f}_s = m, \ul{f}_s = j$ is stopped, then so too are all paths
  with $\bar{g}_s = m',  \ul{g}_s = j'$, where $(m',-j') \prec (m,-j)$
  and $\prec$ denotes the \emph{lexicographical} ordering. With this
  definition, the general outline argument given above can proceed as
  usual, however we do not do this here since the final stage of the
  argument --- showing that the closed and open hitting times of such a
  region are equal --- would appear to be much more subtle than
  previous examples, and so we leave this as an open problem for
  future work.

  However, more notable is that in the multiple marginal case (and
  indeed, already to some extent in the case of a single marginal with
  a general starting law), the Perkins optimality property is no
  longer strictly preserved. To see why this might be the case (see
  also \cite[Remark~2.3]{HoPe02}) we note that, in the case of a
  single marginal, with trivial starting law, the embedding
  constructed via the double minimisation problems always stops at a
  time when the process sets a new minimum or a new maximum. At any
  given possible stopping point, the decision to stop should depend
  both on the current minimum, and the current maximum; however when
  the process is at a current maximum, both the current position and
  the current maximum are the same. In consequence, the
  decision to stop at e.g.\  a new maximum will only depend on the value of
  the minimum, and the optimisation problem relating to maximising a
  function of the maximum will be unaffected by the choice. In
  particular, it is never important which optimisation is the primary
  optimisation problem, and which is the secondary optimisation
  problem: in terms of the barrier-criteria established above, this
  can be seen by observing that in lexicographic ordering,
  $(m',-j') \prec (m,-j)$ is equivalent to $(-j',m') \prec (-j,m)$ if
  either $m=m'$ or $j=j'$.

  On the other hand, with multiple marginals, we may have to consider
  possible stopping at times which do not correspond to setting a new
  maximum or minimum. Consider for example the case with
  $\mu_0 = \delta_0, \mu_1 = (\delta_1 + \delta_{-1})/2, \mu_2 =
  2(\delta_2 + \delta_{-2})/5+\delta_0/5$.
  In particular, the first stopping time, $\tau_1$ must be the first
  hitting time of $\{-1,1\}$, and if the process stops at $0$ at the
  second stopping time, then to be optimal, it must stop there the
  first time it hits $0$ after $\tau_1$. If we consider the
  probability that we return to 0 after $\tau_1$, before hitting
  $\{-2,2\}$, then this is larger than $\frac{1}{5}$, and we need to
  choose a rule to determine which of the paths returning to 0 we
  should stop. It is clear that, if the primary optimisation is to
  minimise the law of the maximum, then this decision would only
  depend on the running maximum, while it will depend only on the
  running minimum if the primary and secondary objectives are
  switched. In particular, the two problems give rise to different
  optimal solutions. The difference here arises from the fact that we
  are not able to assume that all paths have either the same maximum,
  or the same minimum. As a consequence, we do not, in general, expect
  to recover a general version of the Perkins embedding, in the sense
  that there exists a multi-marginal embedding which minimises the law
  of the maximum, and maximises the law of the minimum
  simultaneously. 
\end{remark}

\subsubsection{Further ``classical'' embeddings and other remarks}
By combining the ideas and techniques from the previous sections and the techniques from \cite[Section 6.2]{BeCoHu14} we can establish the existence of $n$-marginal versions of the Jacka and Vallois embeddings and their siblings (replacing the local time with a suitably regular additive functional) as constructed in \cite[Remark 7.13]{BeCoHu14}.  We leave the details to the interested reader. 

We also remark that it is possible to get more detailed descriptions of the structure of the different barriers. At this point we only note that all the embeddings presented above have the nice property that their $n$-marginal solution restricted to the first $n-1$ marginals is in fact the $n-1$ marginal solution. This is a direct consequence of the extension of the Loynes argument to $n$-marginals as shown in the proof of Theorem \ref{thm:Root1}. For a more detailed description of the barriers for the $n$-marginal Root embedding we refer to \cite{CoObTo15}.

We also observe  that, as in \cite[Section 6.3]{BeCoHu14}, it is possible to deduce multi-marginal embeddings of some of the embeddings presented in the previous sections, e.g.\ Root and Rost, in higher dimensions. We leave the details to the interested reader.

\subsubsection{A $n$-marginal version of the monotone martingale coupling.}

We next discuss the embedding giving rise to a multi-marginal version of the monotone martingale transport plan. Note that we need an extra assumption on the starting law $\mu_0$, but on $\mu_0$ only.

\begin{theorem}[$n$-marginal martingale monotone transport plan]\label{thm:MMMM1}
 Assume that $\mu_0$ is continuous (in the sense that $ \mu_0(a)=0$ for all $a\in \R$). Let $c:\R\times\R\to\R$ be three times continuously differentiable with $c_{xyy}<0$. Put $\gamma_i:\S{n}\to\R, (f,s_1,\ldots,s_n)\mapsto c(f(0),f(s_i))$ and assume that \eqref{eq:Pgamma} is well posed. Then there exist $n$ barriers $(\mathcal R^i)_{i=1}^n$ such that defining
\begin{align*}
& \tau_1=\inf\{t\geq 0: (B_t-B_0, B_t)\in \mathcal R^1\}
\end{align*}
and for $1<i\leq n$
\begin{align*}
 \tau_i=\inf\{t\geq \tau_{i-1}: (B_t-B_0, B_t)\in \mathcal R^i\}
\end{align*}
the multi-stopping time $(\tau_1,\ldots,\tau_n)$ minimises
$$\E[c(B_0,B_{\tau_i})]$$ 
simultaneously for all $1\leq i\leq n$ among all increasing families of stopping times $(\tilde \tau_1,\ldots,\tilde\tau_n)$ such that $B_{\tilde \tau_j}\sim\mu_j$ for all $1\leq j\leq n$. This solution is unique in the sense that for any solution $\tilde \tau_1, \ldots, \tilde \tau_n$ of such a barrier-type we have $\tau_i=\tilde \tau_i$.
\end{theorem}
\begin{remark}
In the final stage of writing this article we learned of the work of Nutz, Stebegg, and Tan \cite{NuStTa17} on multi-period martingale optimal transport which (among various further results) provides an $n$-marginal version of the monotone martingale transport plan. Their methods are rather different from the ones employed in this article and in particular not related to the Skorokhod problem. 
\end{remark}
\begin{proof}[Proof of Theorem \ref{thm:MMMM1}.]
The overall strategy of the proof, and in particular the first steps follow exactly the arguments encountered above.
Fix a permutation $\kappa$ of $\{1,\ldots,n\}$. We consider the functions $\tilde\gamma_1=\gamma_{\kappa(1)},\ldots,\tilde\gamma_n=\gamma_{\kappa(n)}$ on $\S{n}$  and the corresponding family of $n$-ary minimisation problems. Pick, by the n-ary version of Theorem \ref{thm:exists}, an optimizer $(\tau_1,\ldots,\tau_n)$ and, by the $n$-ary version of Theorem \ref{thm:monprin2},  a $\tilde\gamma_n|\ldots|\tilde\gamma_1$-monotone family of sets  $(\Gamma_1,\ldots,\Gamma_n)$ supporting $(\tau_1,\ldots,\tau_n),$ i.e.\ for every $i\leq n$ we have  $\P$-a.s.
$$((B_s)_{s\leq\tau_i},\tau_1,\ldots,\tau_i)\in\Gamma_i,$$
 and 
$$(\Gamma_i^{<}\times \Gamma_i)\cap\SG_{i,n}=\emptyset.$$
We claim that   for all $1\leq i\leq n$ we have
$$\SG_{i,n}\supset\{(f,s_1,\ldots,s_i),(g,t_1,\ldots,t_i): f(s_i)=g(t_i), g(0)>f(0)\}.$$
To this end, we have to consider $(f,s_1,\ldots,s_i),(g,t_1,\ldots,t_i)\in \S{i}$ satisfying $f(s_i)=g(t_i), f(s_i)-f(0)>g(t_i)-g(0)$ and consider two families of stopping times $(\sigma_j)_{j=i}^n$ and $(\tau_j)_{j=i}^n$ together with their modifications $(\tilde\sigma_j)_{j=i}^n$ and $(\tilde\tau_j)_{j=i}^n$ as in Section \ref{sec:probint}. However, since the modification of stopping times consists only of repeated swapping of the two stopping times what is effectively sufficient to prove is the following: 

For $f(s)-f(0)>g(t)-g(0)$ and any  stopping times $\rho, \sigma, \tau$, where $\rho\leq \sigma$, we have for
$$\tilde \sigma := \sigma \I_{\rho \leq \tau} +  \tau \I_{\rho >  \tau}, \quad \tilde \tau := \tau \I_{\rho \leq \tau} +  \sigma \I_{\rho >  \tau}$$
the inequality 
\begin{align}\label{therelevantone}\begin{split}
&\E[ c(f(0), f(s) + B_\sigma)] + \E [c(g(0),g(t)+  B_\tau) ]\\
>\ 
& \E [c(f(0),f(s) + B_{\tilde \sigma}) ]+ \E[ c(g(0),g(t)+  B_{\tilde\tau})],\end{split}
\end{align}
and that this inequality is strict, provided that the set $\rho>\tau$ has positive probability. 

To establish this inequality, of course only the parts were $\rho> \tau$ matters. Otherwise put, the inequality remains equally  valid if we replace all of $\sigma, \tau, \tilde \sigma, \tilde \tau$ by $\tau\vee \sigma$ on the set $\rho\leq \tau$, in which case we  have $\tilde \sigma= \tau$, $\tilde \tau=\sigma$, $\sigma\geq\tau$. Hence to prove \eqref{therelevantone} it is sufficient to show for $\alpha:=\law (B_\sigma), \beta:=\law (B_\tau)$ and $a:= f(s)=g(t)$ that 
$$\int c(f(0),a+x)\, d\alpha(x) + \int c(g(0),a+x)\, d\beta(x)
> \int c(f(0),a+x)\, d\beta(x) + \int c(g(0),a+x)\, d\alpha(x).$$
To obtain this, we claim that 
$$t\mapsto \int c(t,a+x)\, d\alpha(x) - \int c(t,a+x)\, d\beta(x)$$
is decreasing in $t$: This holds true since $c_x$ is concave and $\beta$ precedes  $\alpha$  in the convex order (strictly if $\P(\rho>\tau)>0$).

\smallskip

Having established the claim, we define for each $1\leq i\leq n$
$$\mathcal R_\cl^i:=\{(d,x)\in\R_+\times\R: \exists (g,t_1,\ldots,t_i)\in \Gamma_i, g(t_i)=x, d\geq g(t_i)-g(0)\}$$
and
$$\mathcal R_\op^i:=\{(d,x)\in\R_+\times\R: \exists (g,t_1,\ldots,t_i)\in \Gamma_i, g(t_i)=x, d> g(t_i)-g(0)\}.$$
Following the argument used above, we define 
$\tau^1_\cl$ and $\tau^1_\op$ to be the first times the process $(B_t-B_0, B_t)_{t\geq 0}$ hits $\mathcal R^1_\cl$ and $\mathcal R^1_\op$ respectively to see that actually $\tau^1_\cl\leq\tau_1\leq\tau_\op^1$. 

It remains to show that  $\tau^1_\cl=\tau_\op^1$ (This has already been shown in \cite[Prop. 3.1]{HuSt16}; we present the argument for completeness). To this end, note that the hitting time of $(B_t-B_0, B_t)_{t\geq 0}$ into a barrier can equally well be interpreted as the hitting time of $(-B_0, B_t)_{t\geq 0}$ into a transformed (i.e.\ sheared through the transformation $(d,x)\mapsto (d-x,x)$ ) barrier. 
The purpose of this alteration is that the process $(-B_0,B_t)_{t\geq 0}$ moves only vertically and we can now apply Lemma \ref{BarrierEquality} to establish that indeed  $\tau^1_\cl=\tau_\op^1$. Observe that at this stage the continuity assumption on $\mu_0$ is crucial. 

We then proceed by induction. 

As above, uniqueness and the irrelevance of the permutation follow from Loynes' argument. 
\end{proof}

A very natural conjecture is then that Theorem~\ref{thm:MMMM1} would give rise to a solution to the \emph{peacock} problem. The set of martingales $(S_t)_{t\in [0,T]}$ (more precisely the set of corresponding martingale measures) carries a natural topology and
given $D\subseteq [0,T]$ with $T\in D$ the set of martingales  with prescribed marginals $(\mu_t)_{t\in D}$ is compact (cf.\
\cite{BeHuSt15}).
By taking limits of the solutions provided above along appropriate finite discretisations $D\subseteq [0,T]$, one obtains a sequence of optimisers to the discrete problem whose limit $(S_t)_{t\in [0,T]}$ satisfies $S_t\sim \mu_t, t\in [0,T]$ and  minimizes  $\E [(S_t-S_0)^3]$ simultaneously for all $t\in [0,T]$ among all such martingales. However, since this is not the scope of the present article we leave details for future work.

We note that this also provides a continuous time extension of the martingale monotone coupling rather different from the constructions given by Henry-Labord\`ere, Tan, and Touzi \cite{HeTo13} and Juillet \cite{Ju14a}.

\section{Stopping Times and Multi-Stopping Times}\label{sec:Si}

For a Polish space $\sX$ equipped with a probability measure $m$ we define a new probability space $(\X,\cG^0,(\cG^0_t)_{t\geq 0},\P)$ with $\X:=\sX\times \CRo, \cG^0:= \mathcal B(\sX)\otimes \cF^0, \cG^0_t:=\mathcal B(\sX)\otimes \F^0_t$, $\P:= m \otimes \W$, where $\mathcal B(\sX)$ denotes the Borel $\sigma$-algebra on $\sX$, $\W$ denotes the Wiener measure, and $(\F^0_t)_{t\geq 0}$  the natural filtration. We denote the usual augmentation of $\cG^0$ by $\cG^a$. Moreover, for $*\in\{0,a\}$ we set $\cG^*_{0-}:=\mathcal B(\sX)\otimes \F^*_0.$ If we want to stress the dependence on $(\sX,m)$ we write  $\cG^a(\sX,m), \cG^a_t(\sX,m),\ldots$.

The natural coordinate process on $\X$ will be denoted by $Y$, i.e.\ for $t\geq 0$ we set
$$Y_t(x,\omega)=(x,\omega_t).$$
Note that under $\P$, in the case where $\X = \R$, the process $Y$ can be interpreted as a Brownian motion
with starting law $m$. In particular, $t\mapsto Y_t(x,\omega)$ is continuous and $\cG^0_t=\sigma(Y_s, s\leq t)$.

We recall
$$ S:= \left\{ (f,s)~:~f\in C[0,s], f(0)=0\right\}, \qquad S_\sX:=(\sX, S)$$
and introduce the maps
\begin{align}\label{eq:rx}
& r: \CRo\times \R_+\to S, \quad (\omega,t)\mapsto (\omega_{\llcorner [0,t]},t),\\
 &r_\sX: \X\times \R_+\to S_\sX,\quad (x,\omega,t)\mapsto (x,r(\omega,t)).
\end{align}
We equip $C_0(\R_+)$ with the topology of uniform convergence on compacts and $S_\sX$ with the final topology inherited from $\X\times \R_+$ turning it into a Polish space. This structure is very convenient due to the following proposition which is a particular case of \cite[Theorem IV. 97]{DeMeA}.

\begin{proposition}\label{S2F} Optional sets / functions on $\X\times\R_+$ correspond to Borel measurable sets / functions on $S_\sX$. More precisely we have: 
 \begin{enumerate}
  \item A set $D\subseteq \X\times\R_+$ is $\cG^0$-optional iff  there is a Borel set $A\subseteq S_\sX$  with $D=r_\sX^{-1}(A)$. 
\item A process $Z=(Z_t)_{t\in\R_+}$ is $\cG^0$-optional iff there is a Borel measurable $H:S_\sX\to\R$ such that $Z=H\circ r.$ 
\end{enumerate}
A $\cG^0$-optional set $A\subseteq \X\times\R_+ $ is closed in $\X\times\R_+$ iff the corresponding set $r_\sX(A)$ is closed in $S_\sX$. 
\end{proposition}

\begin{definition}\label{def:Scont}
 A $\cG^0$-optional process $Z=H\circ r_\sX$ is called $S_\sX$- continuous (resp.\ l./u.s.c.) iff $H:S_\sX\to \R$ is continuous (resp.\ l./u.s.c.).
\end{definition}

\begin{remark}\label{rem:pred ST}
 Since the process $t\mapsto Y_t(x,\omega)$ is continuous the predictable and optional $\sigma$- algebras coincide (\cite[Theorems~IV.67 (c) and IV.97]{DeMeA}). Hence, every $\cG^0$-stopping  time $\tau$ is predictable and, therefore, foretellable on the set $\{\tau > 0\}.$
\end{remark}

\begin{definition}\label{EAverage}
Let $Z:\X\to \R$ be a measurable function which is bounded or positive. Then we define $\E[Z|\cG_t^0]$ to be the unique $\cG_t^0$-measurable function satisfying 
$$\E[Z|\cG_t^0](x,\omega):=Z^M_t(x,\omega):=  \int Z(x,\omega_{\llcorner[0,t]}\oplus \omega')\, d\W(\omega').$$
\end{definition}

\begin{proposition}\label{prop:very cont}
 Let $Z\in C_b(\X)$. Then $Z^M_t$ defines an $S_\sX$- continuous martingale (see Definition \ref{def:Scont}), $Z^M_\infty=\lim_{t\to\infty}Z^M_t$ exists and equals $Z$. 
\end{proposition}
\begin{proof}
 Up to a minor change of the probability space this is \cite[Proposition 3.5]{BeCoHu14}.
\end{proof}

\subsection{Randomised stopping times}\label{sec:rst}

We set
$$ \mathsf{M}:=\{\xi\in\mathcal P^{\leq 1}(\X\times \R_+): \xi(d(x,\omega),ds)=\xi_{x,\omega}(ds)~\P(d(x,\omega)), \xi_{x,\omega} \in\mathcal P^{\leq 1}(\R_+)\}$$
and equip it with the weak topology induced by the continuous and bounded functions on $\X\times\R_+$. Each $\xi\in\mathsf{M}$ can be uniquely characterized by its cumulative distribution function $A^\xi_{x,\omega}(t):=\xi_{x,\omega}([0,t])$.
\begin{definition}\label{def:rst}
A measure $\xi\in\mathsf{M}$ is called randomized stopping time, written $\xi\in\RST$, iff the associated increasing process $A^\xi$ is optional. If we want to stress the Polish probability space $(\sX,\mathcal B(\sX), m)$ in the background, we write $\RST(\sX,m)$.
\end{definition}

We remark that randomized stopping times are a subset of the so called $\mathbf{P}$-measures introduced by Doleans \cite{Dol68} (for motivation and further remarks see \cite[Section 3.2]{BeCoHu14}).

In the sequel we will mostly be interested in a representation of randomized stopping times on an enlarged probability space. We will be interested in $(\X',\cG',(\cG_t')_{t\geq 0},\P')$ where $\X':=\X\times [0,1]$, $\P'(A_1\times A_2)=\P(A)\leb(A_2)$ ($\leb$ denoting Lebesgue measure on $\R$), $\cG'$ is the completion of $\cG^0\otimes\mathcal B([0,1])$, and $(\cG_t')_{t\geq 0}$ the usual augmentation of $(\cG_t^0\otimes \mathcal B([0,1]))_{t\geq 0}.$ 

The following characterization of randomized stopping times is essentially Theorem~3.8 of \cite{BeCoHu14}. The only difference is the presence of the $\sX$ in the starting position, however it is easily checked that this does not affect the proof.

\begin{theorem}\label{thm:equiv RST}
Let $\xi\in \mathsf{M}$. Then the following are equivalent:
\begin{enumerate}
\item There is a Borel function $A:S_\sX\to[0,1]$ such that the process $A\circ r_\sX$ is right-continuous increasing and
\begin{align}\label{AlternativeRep} \xi_{x,\omega}([0,s]):=A\circ r_\sX(x,\omega,s)\end{align}
defines a disintegration of $\xi$ wrt to $\P$. 
\item We have $\xi\in \RST(\sX,m)$, i.e.\ given a  disintegration $(\xi_{x,\omega})_{(x,\omega)\in\X}$ of $\xi$ wrt $\P$, the random variable   $ \tilde A_t(x,\omega)=\xi_{x,\omega}([0,t])$ is $\cG^a_t$-measurable for all $t\in \R_+$. 
\item  For all $f\in C_b(\R_+)$ supported on some  $ [0,t]$, $t\geq 0$ and all $g\in C_b(\X)$ 
\begin{align}\label{CheckMeasurability}\textstyle\int f(s) (g-\E[g|\cG_t^0])(x,\omega) \, \xi(dx,d\omega, ds)=0\end{align}
\item On the probability space $(\X',\cG',(\cG_t')_{t\geq 0},\P')$, the random time
  \begin{equation}
    \label{eq:rhodefn}
     \rho(x,\omega,u) :=\inf \{ t \ge 0 : \xi_{x,\omega}([0,t]) \ge u\}
  \end{equation}
  defines an $\cG'$-stopping time.
\end{enumerate}
\end{theorem}

\begin{remark}\label{rem:rst}
 An immediate consequence of \eqref{CheckMeasurability} is the closedness of $\RST$ wrt to the weak topology induced by the continuous and bounded functions on $\X\times \R_+$ (cf.\ \cite[Corollary 3.10]{BeCoHu14} and  Lemma \ref{lem:RSTclosed}).
\end{remark}

\subsection{Randomised multi-stopping times}\label{sec:rmst}

In this section, we extend the results of the last section to the case of multiple stopping. Recall the notation defined in Section~\ref{sec:notation}. In particular,
for $d\geq 1$, recall that
 $$\Xi^d:=\{(s_1,\ldots,s_d)\in\R^d_+,s_1\leq\ldots\leq s_d\}$$
  and define $\mathsf{M}^d$ to consist of all $\xi\in\mathcal P^{\leq 1}(\X\times \Xi^d)$ such that
$$ \xi(d(x,\omega),ds_1,\ldots,ds_d)=\xi_{x,\omega}(ds_1,\ldots,ds_d)~\P(d(x,\omega)), \xi_{x,\omega} \in\mathcal P^{\leq 1}(\Xi^d) .$$
Recall that $(\bar\X,\bar\cG,(\bar \cG_t)_{t\geq 0},\bar\P)$ is defined by $\bar\X=\X\times [0,1]^d,\bar\P(A_1\times A_2)=\P(A_1)\leb^d(A_2),$ where $\leb^d$ denotes the Lebesgue measure on $\R^d$ and $\bar\cG_t$ is the usual augmentation of $\cG^0_t\otimes \mathcal B([0,1]^d)$. We mostly denote $\leb^d(du)$  by $du$. For $(u_1,\ldots,u_d)\in[0,1]^d$ we often just write $(u_1,\ldots,u_d)=u.$ We suppress the $d$- index in the notation for the extended probability space. It will either be clear from the context which $d$ we mean or we explicitly write down the corresponding spaces.

\begin{definition}\label{def:RMST}
 A measure $\xi\in\M^d$ is called \textit{randomised multi-stopping time}, denoted by $\xi\in\RMST_d$, if for all $0\leq i\leq d-1$ 
\begin{equation}\label{eq:defrmst}\tilde r_{i+1,i}(\xi^{(i+1)})\in\RST(\S{i}_\sX,r_i(\xi^i)).\end{equation}
We denote the subset of all randomised multi-stopping times with total mass 1 by $\RMST_d^1$.
If we want to stress the dependence on $(\sX,m)$ we write $\RMST_d(\sX,m)$ or $\RMST_d^1(\sX,m).$
\end{definition}

Unlike for the randomised stopping times, there is no obvious analogue of (1), (2) or (3) of Theorem~\ref{thm:equiv RST} in the multi-stopping time setting. However below we prove a representation result for randomised multi-stopping times in a similar manner to (4). The following lemma (c.f. \cite[Lemma 3.11]{BeCoHu14}) then enables us to conclude that, on an arbitrary probability space, all sequences of increasing stopping times can be represented as a randomised multi-stopping time on our canonical probability space.

\begin{lemma}\label{lem:rmst}
Let $B$ be a Brownian motion on some stochastic basis $(\Omega, \mathcal H, (\mathcal H_t)_{t\geq 0},\Q)$ with right continuous filtration. Let $\tau_1,\ldots,\tau_n$ be an increasing sequence of $\mathcal H$-stopping times and consider
$$\Phi:\Omega\to \CR\times \Xi^d, \quad \bar\omega\mapsto ((B_t)_{t\geq 0},\tau_1(\bar\omega),\ldots,\tau_n(\bar\omega)).$$
Then $\xi:=\Phi(\Q)$ is a randomized multi-stopping time and for any measurable $\gamma:\S{n}\to\R$ we have
\begin{align}\label{eq:Omegaunimportant}
 \int \gamma(f,s_1,\ldots,s_n)~r_n(\xi)(d(f,s_1,\ldots,s_n))=\E_\Q[\gamma((B_t)_{t\leq\tau_n},\tau_1,\ldots,\tau_n)].
\end{align}
If $\Omega$ is sufficiently rich that it supports a uniformly distributed random variable which is $\mathcal H_0$-measurable, then for $\xi\in\RMST$ we can find an increasing family $(\tau_i)_{1\leq i \leq n}$ of $\mathcal H$-stopping times such that $\xi=\Phi(\Q)$ and \eqref{eq:Omegaunimportant} holds.
\end{lemma}
\begin{proof}
 For notational convenience we show the first part for the case $n=2$. Let $B_0\sim m$. It then follows by \cite[Lemma 3.11]{BeCoHu14} that $\tilde r_{1,0}(\xi)\in\RST(\R,m)$. Hence, we need to show that $\tilde r_{2,1}(\xi)\in\RST(S_\R,r_1(\xi))$, i.e.~ we have to show that $\xi^2_{(f,s)}$ is $r_1(\xi)$--a.s.~ a randomized stopping time, where $(\xi^2_{(f,s)})_{(f,s)}$ denotes a disintegration of $\xi^2$ wrt $r_1(\xi).$ (Here and in the rest of the proof we assume $f(0)\in\R$ and suppress the ``$x$'' from the notation).

First we show that $\tilde r_1(\xi^1)(d(f,s),d\omega)=r_1(\xi^1)(d(f,s))\W(d\omega).$ Take a measurable and bounded $F:S_\R\times \CRo\to\R$. Then, 
using the strong Markov property in the last step, we have
 \begin{align*}
  &\int F((f,s),\omega)~\tilde r_1(\xi^1)(d(f,s),d\omega) \numberthis \label{eq:disint0}\\
=& \int F((r_\sX(\tilde\omega,s),\theta_{s}\tilde\omega)~ \xi^1(d\tilde\omega,ds)\\
=&  \int_{\Omega} F(r_\sX(B(\omega),\tau_1(\omega)),\theta_{\tau_1(\omega)}B(\omega))~\Q(d\omega)\\
=& \int F((f,s),\tilde\omega)~r_1(\xi^1)(d(f,s))\W(d\tilde\omega)~.
 \end{align*}
Let $q$ be the projection from $S_\R\times \CRo\times\R_+$ to $S_\R\times\CRo$, and $p$ be the projection from $\X\times\Xi^2\to\X\times\R_+$, $p(\omega,s_1,s_2)=(\omega,s_1).$ Then, $q\circ \tilde r_{2,1}=\tilde r_1\circ p$. Recalling that $\xi^1=p(\xi^2)$ there is a disintegration of $\tilde r_{2,1}(\xi^{1,2})$ wrt $\tilde r_1(\xi^1)$ which we denote by
$$\xi^2_{(f,s_1),\omega}(ds_2)\in\mathcal P^{\le 1}(\R_+).$$
Then, we set $\xi^2_{(f,s_1)}(d\omega,ds_2):=\xi^2_{(f,s_1),\omega}(ds_2)\W(d\omega).$ Since $d\tilde r_1(\xi^1)=dr_i(\xi^1)d\W$ the measures $\xi^2_{(f,s_1)}$ define a disintegration of $\tilde r_{2,1}(\xi^2)$ wrt $r_1(\xi^1).$ We have to show that $r_1(\xi^1)$ a.s.\ $\xi^2_{(f,s_1)}$ is a randomized stopping time. We will show property (2) in Theorem \ref{thm:equiv RST}, where now $\mathsf{X}=S_\R, m=r_1(\xi)$ and accordingly $\cG^0_t=\mathcal{B}(S_\R)\otimes\cF^0_t$ with usual augmentation $\cG_t^a$ (cf.~ Section \ref{sec:notation}).

To this end, fix $t\geq 0$ and let $g:S_\R\times\CRo\to\R$ be bounded and measurable and set $h=\E_m[g|\cG_t^a].$ Then, it holds that $\E_\Q[g(r_1(B,\tau_1),\theta_{\tau_1}B)|\mathcal H_{\tau_1+t}]=h(r_1(B,\tau_1),\theta_{\tau_1}B)$. Using rightcontinuity of the filtration $\mathcal H$ in the third step to conclude that $\tau_2-\tau_1$ is an $(\mathcal H_{\tau_1+t})_{t\geq 0}$ stopping time, this implies
\begin{align*}
& \int g((f,s),\omega)~\xi^2_{(f,s),\omega}([0,t])~r_1(\xi^1)(d(f,s))\W(d\omega)\\
=&\, \E_\Q\left[g(r_1(B,\tau_1),\theta_{\tau_1} B)\1_{\tau_2-\tau_1\leq t}\right]\\
=&\,\E_\Q\left[\E_\Q\left[g(r_1(B,\tau_1),\theta_{\tau_1} B)|\mathcal H_{\tau_1+t}\right]\1_{\tau_2-\tau_1\leq t}\right]\\
=&\,\E_\Q\left[h(r_1(B,\tau_1),\theta_{\tau_1} B)\1_{\tau_2-\tau_1\leq t}\right]\\
=& \int h((f,s),\omega)~\xi^2_{(f,s),\omega}([0,t])~r_1(\xi^1)(d(f,s))\W(d\omega).
\end{align*}
This shows the first part of the lemma. 

To show the second part of the lemma we start by constructing an increasing sequence of stopping times on the extended canonical probability space  $(\bar\X,\bar\cG,(\bar \cG_t)_{t\geq 0},\bar\P)$. By Theorem \ref{thm:equiv RST} and the assumption that $\xi^{1}\in\RST(\sX,m)$ there is a $\bar\cG$ stopping time $\rho^1(x,\omega,u)=\rho^1(x,\omega,u_1)$ defining a disintegration of $\xi^{1}$ wrt $\P$ via
$$\int \delta_{\rho^1}(ds_1)~du~.$$
By assumption, $\tilde r_{2,1}(\xi^2)\in\RST(S_\sX,r_1(\xi^1))$. Hence, writing $s_2'=s_2-s_1$ we can disintegrate 
$$\xi^2(d(x,\omega),ds_1,ds_2)=\int \xi^2_{r_\sX(x,\omega,\rho^1(x,\omega,u_1))}(\theta_{\rho^1(x,\omega,u_1)}\omega,ds_2')\delta_{\rho^1(x,\omega,u_1)}(ds_1) du_1$$
such that for $r_1(\xi^1)$ a.e. $(x,f,s_1)$ the  disintegration $\xi^2_{(x,f,s_1)}$ is a randomized stopping time. Again by Theorem \ref{thm:equiv RST} there is a stopping time $\tilde\rho^2_{x,f,s_1}(\tilde\omega,u_2)$ representing $\xi^2_{(x,f,s_1)}$ as in \eqref{eq:rhodefn}. Then,
$$\rho^2(x,\omega,u_1,u_2):= \rho^1(x,\omega,u_1)+\tilde\rho^2_{r_\sX(x,\omega,\rho_1(x,\omega,u_1))}(\theta_{\rho_1(x,\omega,u_1)}\omega,u_2)$$
defines a $\bar\cG$ stopping time such that 
\begin{align*}
 (x,\omega)\mapsto \int_{[0,1]^d} \delta_{\rho^1(x,\omega,u)}(dt_1) \delta_{\rho^2(x,\omega,u)}(dt_2)\ du
\end{align*}
defines a $\cG^a$- measurable disintegration of $\xi^2$ w.r.t.\ $\P$. We proceed inductively. To finish the proof, let $U$ be the $[0,1]^d$--valued uniform $\mathcal H_0$--measurable random variable. Then $\tau_i:=\rho^i(B,U)$ define the required increasing family of $\mathcal H$ stopping times.
\end{proof}
\begin{remark}
Lemma \ref{lem:rmst} shows that optimizing over an increasing family of stopping times on a rich enough probability space in  \eqref{eq:Pgamma} is equivalent to optimizing over randomized multi-stopping times on the Wiener space.
\end{remark}

\begin{corollary}\label{cor:canonicalRMST}
Let $\xi\in\RMST$. On the extended canonical probability space  $(\bar\X,\bar\cG,(\bar \cG_t)_{t\geq 0},\bar\P)$ there exists an increasing sequence $(\rho^i)_{i=1}^d$ of $\bar\cG$- stopping times such that 
\begin{enumerate}
 \item for $u=(u_1,\ldots,u_d)\in [0,1]^d$ and for each $1\leq i\leq d$ we have $\rho^i(x,\omega,u)=\rho^i(x,\omega,u_1,\ldots,u_i)$;
\item \begin{align}\label{eq:STdisint}
 (x,\omega)\mapsto \int_{[0,1]^d} \delta_{\rho^1(x,\omega,u)}(dt_1)\cdots \delta_{\rho^d(x,\omega,u)}(dt_d)\ du
\end{align}
defines a $\cG^a$- measurable disintegration of $\xi$ w.r.t.\ $\P$.
\end{enumerate}
\end{corollary}

Next we introduce some notation to state  another straightforward corollary.
It is easy to see that
$q^{d,i} \circ \tilde{r}_{d,i} = \tilde{r}_i\circ p^{d,i}$, where
$q^{d,i}$ is the projection from
$\S{i}_\sX\times C(\R_+) \times \Xi^{d-i}$ to $\S{i}_\sX\times
C(\R_+)$, and $p^{d,i}$ is the projection from $\X\times\Xi^d$ to
$\X\times\Xi^i$ defined by 
$$(x,\omega,s_1,\ldots,s_d)\mapsto (x,\omega,s_1,\ldots,s_i).$$ 
Recalling that $\xi^i = p^{d,i}(\xi)$, it follows that there
exists a disintegration of $\tilde r_{d,i}(\xi)$ with respect to
$\tilde{r}_i(\xi^i)$, which we denote by:
$$ \xi_{(x,f,s_1,\ldots,s_i),\omega}(ds_{i+1},\ldots,ds_d) \in \mathcal P( \Xi^{d-i}).$$
Moreover, we set
$$ \xi_{(x,f,s_1,\ldots,s_i)}(d\omega,ds_{i+1},\ldots,ds_d):=\xi_{(x,f,s_1,\ldots,s_i),\omega}(ds_{i+1},\ldots,ds_d) \ \W(d\omega) \in \mathcal P(C(\R_+) \times \Xi^{d-i}).$$
The map $(x,f,s_1,\ldots,s_i)\mapsto \xi_{(x,f,s_1,\ldots,s_i)}$ inherits
measurability from the joint measurability of
$((x,f,s_1,\ldots,s_i),\omega)\mapsto \xi_{(x,f,s_1,\ldots,s_i),\omega}.$
In particular, $\xi_{(x,f,s_1,\ldots,s_i)}$
defines a disintegration of $\tilde r_{d,i}(\xi)$ w.r.t.\ $r_i(\xi^i)$, since $d\tilde r_i(\xi^i) = d\W dr_i(\xi^i)$ by the same calculation as \eqref{eq:disint0}. Following exactly the line of reasoning as in the first part of the proof of Lemma \ref{lem:rmst} yields
\begin{corollary}\label{cor:xii rmst}
 Let $\xi\in\RMST_d(\sX,m)$ and $1\leq i<d$. Then, for $r_i(\xi^i)$ a.e.\ $(x,f,s_1,\ldots,s_i)$ we have $\xi_{(x,f,s_1,\ldots,s_i)}\in\RMST_{d-i}(\{0\},\delta_0).$
\end{corollary}
\begin{remark}\label{rem:i=0}
 We note that the last Corollary still holds for $i=0$ by setting $\S{0}_\R=\R, r_0(\xi^j)=m$. Then, the result says that for a disintegration $(\xi_x)_x$ of $\xi$ w.r.t.\ $m$ for $m$-a.e.\ $x\in \sX$ we have
$\xi_x\in\RMST_d.$ Of course this can also trivially be seen as a consequence of $\P=m\otimes\W.$
\end{remark}

An important property of $\RMST$ is the following Lemma.

\begin{lemma}\label{lem:RSTclosed}
 $\RMST$ is closed w.r.t.\ the weak topology induced by the continuous and bounded functions on $\X\times\Xi^d.$
\end{lemma}
\begin{proof}
 We fix $0\leq i\leq d-1$ and consider  the Polish space $\tilde \sX=\S{i}_\sX$ with corresponding $\tilde \X=\tilde \sX\times\CRo$ and $\P=r_i(\xi^i)\otimes\W$. To show the defining property \eqref{eq:defrmst} in Definition \ref{def:RMST} we consider condition (2) in Theorem \ref{thm:equiv RST}; the goal is to express measurability of $Z_t(x,\omega):= \xi^{i+1}_{x,\omega}(f), f\in  C_b([0,t]), x\in \S{i}_\sX, \omega\in\CRo$ in a different fashion. Note that a bounded Borel function $h$ is $\cG_t^0$-measurable iff for all   bounded Borel functions $G:\tilde\X\to\R$ 
$$ \E[h G]= \E [h \E[G|\cG_t^0]],$$
of course this does not rely on our particular setup. 
By a functional monotone class argument, for $\cG_t^0$-measurability of $Z_t$ it is sufficient to check 
that 
\begin{align}\label{MeasurabilityCheck}
\E[Z_t (G-\E[G|\cG_t^0])]=0
\end{align}
for all $G\in C_b(\tilde\X)$. 
In terms of $\xi^{i+1}$, \eqref{MeasurabilityCheck} amounts to 
\begin{align*}
0=\E[Z_t (G-\E[G|\cG_t^0])] \ & =\int \, \P (dx,d\omega) \int \xi^{i+1}_{x,\omega}(ds) f(s) (G-\E[G|\cG_t^0])(x,\omega)\\
& =\int f(s) (G-\E[G|\cG_t^0])(x,\omega) \, \tilde r_{i+1,i}(\xi^{i+1})(dx,d\omega, ds),
\end{align*}
which is a closed condition by Proposition~\ref{prop:very cont}. 
\end{proof}

Given $\xi\in\M^d$ and $s\geq 0$ we define the random measure $\xi\wedge s$  on $\Xi^d$ by setting for $A\subset \Xi^d$ and each $(x,\omega)\in \X$  
$$(\xi\wedge s)_{x,\omega}(A)= \int \1_A(s_1\wedge s,\ldots,s_d\wedge s)~\xi_{x,\omega}(ds_1,\ldots,ds_d).$$

Assume that $(M_s)_{s\geq 0}$ is a process on $\X$. Then  $(M_s^\xi)_{s\geq 0}$ is defined to be the \emph{probability measure on} $\R^{d+1}$ such that for all bounded and measurable functions $F:\R^{d+1}\to \R$
$$\int_{\R^{d+1}} F(y)\ M_s^\xi(dy)=\int_{\X\times\Xi^d} F(M_0(x,\omega),M_{s_1}(x,\omega),\ldots,M_{s_d}(x,\omega))\ (\xi\wedge s)(dx,d\omega,ds_1,\ldots,ds_d).$$
This means that $M_s^\xi$ is the image measure of $\xi\wedge s$ under the map $M:\X\times\Xi^d\to \R^{d+1}$ defined by
$$(x,\omega,s_1,\ldots,s_d)\mapsto (M_0(x,\omega),M_{s_1}(x,\omega),\dots,M_{s_d}(x,\omega)).$$
We write $\lim_{s\to\infty}M^\xi_s=M_\xi$ if it exists.

\subsection{The set $\RMST(\mu_0,\mu_1,\ldots,\mu_n)$, compactness and existence of optimisers.} \label{sec:optimExists}

In this subsection, we specialise our setup to $\sX=\R, m=\mu_0\in\mathcal P(\R)$ and $d=n$.
Let $\mu_0,\mu_1,\ldots,\mu_n\in\mathcal P(\R)$ be centered, in convex order and with finite 
second moment\footnote{It is possible to relax this, see \cite[Section 8]{BeCoHu14}} $\int x^2 \mu_i(dx)=V_i<\infty$ for all $i\leq n$. In particular $V_i\leq V_{i+1}.$ For $t\geq 0$ we set $B_t(x,\omega)=x+\omega_t~.$ We extend $B$ to the extended probability space $\bar\Xs$ by setting $\bar B(x,\omega,u)=B(x,\omega)$. By considering the martingale $\bar B_t^2-t$ we immediately get (see the proof of Lemma 3.12 in \cite{BeCoHu14} for more details)

\begin{lemma}\label{lem:ui}
 Let $\xi\in\RMST_n$ and assume that $B_\xi=(\mu_0,\mu_1,\ldots,\mu_n)$. Let $(\rho_1,\ldots,\rho_n)$ be any representation of $\xi$ granted by Lemma~\ref{lem:rmst}. Then, the following are equivalent 
\begin{enumerate}
 \item $\bar\E[\rho^i]<\infty$ for all $1\leq i\leq n$
\item $\bar\E[\rho^i]=V_i-V_0$ for all $1\leq i\leq n$
\item $ (\bar B_{\rho^i\wedge t})_{t\geq0}  $ is uniformly integrable for all $1\leq i\leq n$.
\end{enumerate}
Of course it is sufficient to test any of the above quantities for $i=n$.
\end{lemma}

\begin{definition}
 We denote by $\RMST(\mu_0,\mu_1,\ldots,\mu_n)$ the set of all randomised multi-stopping times satisfying one of the conditions in Lemma \ref{lem:ui}.
\end{definition}

By pasting solutions to the one marginal Skorokhod embedding problem one can see that the set $\RMST(\mu_0,\mu_1,\ldots,\mu_n)$ is non-empty. However, the most important property is

\begin{proposition}\label{prop:RMSTcompact}
 The set $\RMST(\mu_0,\mu_1,\ldots,\mu_n)$ is compact wrt to the topology induced by the continuous and bounded functions on $\CR\times\Xi^d$.
\end{proposition}
\begin{proof}
 This is a direct consequence of the compactness of $\RST(\mu_n)$ established in \cite[Theorem 3.14]{BeCoHu14} as the set $\RMST(\mu_0,\mu_1,\ldots,\mu_n)$  is closed.
\end{proof}

This result allows us to deduce one of the critical results for our optimisation problem:

\begin{proof}[Proof of Theorem \ref{thm:exists}]
 This follows from Proposition \ref{prop:RMSTcompact} and the Portmanteau theorem.
\end{proof}

\subsection{Joinings of Stopping times}
We now introduce the notion of a joining; these will be used later to
define new stopping times which are candidate competitors for our
optimisation problem.

\begin{definition}\label{def:joinings}
 Let $(\Ys,\sigma)$ be a Polish probability space. The set $\TRT(m,\sigma)$ of joinings between $\P=m\otimes \W$ and $\sigma$ is defined to consist of all subprobability measures $\pi\in\mathcal P^{\leq 1}(\X \times \R_+\times \Ys)$ such that
\begin{itemize}
 \item $\proj_{\X\times\R_+}(\pi_{\llcorner \X\times\R_+\times B})\in\RST(\sX,m)$ for all $B\in\mathcal B(\Ys)$;
\item $\proj_{\X}(\pi)=\P$
\item $\proj_\Ys(\pi)\leq\sigma\;.$
\end{itemize}
\end{definition}

\begin{example}\label{ex:joinings}
 An important example in the sequel will be the probability space $(\X,\P)$ constructed from  $\sX=\S{i}_\R$ and $ m=r_{i}(\xi^{i})$ for $\xi\in\RMST^1_n(\R,\mu_0)$ and $0\leq i<n$, where we set $\S{0}=\R, r_0(\xi^0)=\mu_0$ leading to $\X=\S{i}_\R\times C(\R_+)$ and $\P=r_i(\xi^i)\W=\tilde r_i(\xi^i)$ (cf.\ Corollary \ref{cor:xii rmst}). 
\end{example}

\section{Colour Swaps, Multi-colour Swaps and Stop-Go pairs}
\label{sec:CSMCSSGPairs}

In this section, we will define the general notion of stop-go pairs which was already introduced in a weaker form in Section \ref{sec:probint}. We will do so in two steps. First we define colour swap pairs and then we combine several colour swaps to get multi-colour swaps. Together, they build the stop-go pairs.

 Our basic intuition for the different swapping rules
  comes from the following picture. We imagine that each of the
  measures $\mu_1,\ldots,\mu_n$ carries a certain colour, i.e.\ the
  measure $\mu_i$ carries colour $i$. The Brownian motion will be
  thought of being represented by a particle of a certain colour: at
  time zero the Brownian particle has colour 1 and when it is stopped
  for the $i$-th time it changes its colour from $i$ to $i+1$ (cf.\
  Figure \ref{fig:BadPairs} in Section \ref{sec:probint}).

  In identifying a stop-go pair, we want to consider two sub-paths,
  $(f, s_1, \dots, s_i)$ and $(g,t_1, \dots, t_{i})$, and imagine the
  future stopping rules, which will now be a sequence of colour
  changes, obtained by concatenating a path $\omega$ onto the two
  paths. The simplest way of creating a new stopping rule is simply to
  exchange the coloured tails. This will preserve the marginal law of the
  stopped process, while generating a new multi-stopping time.
  A generalisation of this rule would be to try and swap back to the
  original colour rule at the $j$th colour change, where $i < j$. In
  this case, one would swap the colours until the first time one of
  the paths would stop for the $j$th time, after which one attempts to
  revert to the previous stopping rule. Note however that this may not
  be possible: if the other path has not yet reached the $j-1$st
  colour change, then the rules cannot be stopped, since one would
  have to switch from the $j$th colour to the $j-1$st colour, which is
  not allowed. Instead, in such a case, we simply keep the swapped
  colourings. We call recolouring rules of this nature \emph{colour swaps}
  (or $i \leftrightarrow j$ colour swaps). We will define such colour
  swap pairs in Section~\ref{sec:CS}.

  After consideration of these colour swaps, it is clear that the
  determination of when to revert to the original stopping rule could
  be determined in a more sophisticated manner. For example, instead
  of trying to revert only on the $j$th colour change, one could
  instead try to revert on every colour change, and revert the first
  time it is possible to revert. This recolouring rule gives us a
  second set of possible path swaps, and we call such pairs
  \emph{multi-colour swaps}. We will define these recolouring rules in
  Section~\ref{sec:MCS}. Of course, a multitude of other rules can
  easily be created. For our purposes, colour swaps and multi-colour
  swaps will be sufficient, but other generalisations could easily be
  considered. 

  An important aspect of the recolouring rules are that they provide a
  recipe to map from one stopping rule to another, and an important
  aspect that needs to be verified is that the new stopping rule does
  indeed define a randomised multi-stopping time.

We fix $\xi\in \RMST_n^1(\R,\mu_0)$ and $\gamma:\S{n}_\R\to\R.$ As in the previous section, we denote $\xi^i=\xi^{(1,\ldots,i)}=\proj_{\X\times (1,\ldots,i)}(\xi)$. For $(x,f,s_1,\ldots,s_i)\in\S{i}_\R$ we write $(f,s_1,\ldots,s_i)$ and agree on $f(0)=x\in\R.$ For $(f,s_1,\ldots,s_i)\in\S{i}_\R$ and $(h,s)\in S$ we will often write $(f,s_1,\ldots,s_i)|(h,s)$ instead of $(f,s_1,\ldots,s_i)\cat (h,s)\in\S{i+1}_\R$ to stress the probabilistic interpretation of conditioning the continuation of $(f,s_1,\ldots,s_i)$ on $(h,s).$

\subsection{Coloured particles and conditional randomised multi-stopping times.}\label{sec:condRMST}

By Corollary \ref{cor:xii rmst} and Remark \ref{rem:i=0} (for $i=0$), for each $0\leq i\leq n$ the measure $\xi_{(f,s_1,\ldots,s_i)}$ is $r_i(\xi^i)-$a.s.\ a randomised multi-stopping time. For each $0\leq i\leq n-1$ we  fix a disintegration $(\xi_{(f,s_1,\ldots,s_i),\omega})_{(f,s_1,\ldots,s_i),\omega}$  of $\tilde{r}_{n,i}(\xi)$ w.r.t.\ $\tilde r_i(\xi^i)$ and set $\xi_{(f,s_1,\ldots,s_i)}=\xi_{(f,s_1,\ldots,s_i),\omega}\W(d\omega).$
We will need to consider randomised multi-stopping times conditioned on not yet having stopped the particle of colour $i+1$. To this end, observe that
$$\xi^{i+1}_{(f,s_1,\ldots,s_i)}:= \proj_{C(\R_+)\times \Xi^1}(\xi_{(f,s_1,\ldots,s_i)})$$
defines a disintegration of $\tilde r_{i+1,i}(\xi^{i+1})$ wrt $r_i(\xi^i)$. By Definition \ref{def:RMST}, $\xi^{i+1}_{(f,s_1,\ldots,s_i)}\in\RST$ a.s.\ and we set 
$$ A^\xi_{(f,s_1,\ldots,s_{i})}(\omega,t):= A^\xi_{(f,s_1,\ldots,s_{i})}(\omega_{\llcorner[0,t]},t):=(\xi^{i+1}_{(f,s_1,\ldots,s_{i})})_\omega([0,t])$$
which is well defined for $r_i(\xi^i)$-almost every
$(f,s_1,\ldots,s_i)$ by Theorem \ref{thm:equiv RST}.

For $(f,s_1,\ldots,s_i)\in \S{i}_\R$ we define the conditional randomised multi-stopping time given $(h,s)\in S$ to be the (sub) probability measure $\xi_{(f,s_1,\ldots,s_{i})|(h,s)}$  on $C(\R_+)\times \Xi^{n-i}$ given by
\begin{align}\label{eq:cRMST}
&(\xi_{(f,s_1,\ldots,s_{i})|(h,s)})_\omega([0,T_{i+1}]\times\ldots\times[0,T_n])\\
=& \begin{cases}
(\xi_{(f,s_1,\ldots,s_{i})})_{h\oplus\omega}((s,s+T_{i+1}]\times\ldots\times(s,s+T_n]) & \text{ if } A^\xi_{(f,s_1,\ldots,s_i)}(h,s)<1\\
    \Delta A^\xi_{(f,s_1,\ldots,s_i)}(h,s) (\xi_{(f\oplus h,s_1,\ldots,s_{i},s_i+s)})_{\omega}([s,s+T_{i+2}]\times\ldots\times [s,s+T_n]) & \text{ if } A^\xi_{(f,s_1,\ldots,s_i)}(h,s)=1,
   \end{cases} \nonumber
\end{align}
where $\Delta A^\xi_{(f,s_1,\ldots,s_i)}(h,s)= A^\xi_{(f,s_1,\ldots,s_i)}(h,s)- A^\xi_{(f,s_1,\ldots,s_i)}(h,s-)$. The second case in \eqref{eq:cRMST} corresponds to $(\xi^{i+1}_{(f,s_1,\ldots,s_i)})_{h\oplus\omega}$ having an atom at time $s$ eating up `all the remaining (positive) mass' which is of course independent of $\omega$. This causes a $\delta_0$ to appear in \eqref{eq:ncRMST} below. Moreover, in this case  it is possible that also all particles of colour $j\in \{i+2,\ldots,n\}$ are stopped at time $s$ by $(\xi^{i+1}_{(f,s_1,\ldots,s_i)})_{h\oplus\omega}$. This is the reason for the closed intervals in the second line on the right hand side of \eqref{eq:cRMST}. Using Lemma \ref{lem:rmst} resp.\ Corollary \ref{cor:canonicalRMST} it is not hard to see that \eqref{eq:cRMST} indeed defines a randomized multi-stopping times (you simply have to consider the stopping times $\rho^l(\omega,u_1,\ldots,u_l)$ representing  $\xi_{(f,s_1,\ldots,s_i})$ with $u_1> A^\xi_{(f,s_1,\ldots,s_i)}$ for the first case and the second case is immediate).

Accordingly, we define the normalised conditional randomised
multi-stopping times, by
\begin{align}\label{eq:ncRMST}
  \bar\xi_{(f,s_1,\ldots,s_i)|(h,s)}:=\begin{cases}
    \frac1{1-A^\xi_{(f,s_1,\ldots,s_i)}(h,s)} \cdot \xi_{(f,s_1,\ldots,s_i)|(h,s)} & \text{ if } A^\xi_{(f,s_1,\ldots,s_i)}(h,s)<1,\\
    \delta_{0}\xi_{(f\oplus h,s_1,\ldots,s_i,s_i+s)} & \text{ if } A^\xi_{(f,s_1,\ldots,s_i)}(h,s-)<1 \text{ and }A^\xi_{(f,s_1,\ldots,s_i)}(h,s)=1,\\
    \delta_0\cdots\delta_0 & \text{ else.}
  \end{cases}
\end{align}
We emphasize that the construction of $\bar\xi_{(f,s_1,\ldots,s_i)|(h,s)}$ and $\xi_{(f,s_1,\ldots,s_i)|(h,s)}$ only relies on the fixed disintegration of $\tilde r_{n,i}(\xi)$ w.r.t.\ $\tilde r_i(\xi)$. In particular, the map
\begin{equation}\label{eq:ximeas}
  ((f,s_1,\ldots,s_i),(h,s))\mapsto \bar\xi_{(f,s_1,\ldots,s_i)|(h,s)}
\end{equation}
is measurable. 

Recall the connection of Borel sets of $S_\sX$ and optional sets in $\Xs\times\R_+$ given by Proposition \ref{S2F}.

\begin{definition}
 Let $(\sX,m)$ be a Polish probability space. A set $F\subset S_\sX$ is called $m$-evanescent iff $r_\sX^{-1}(F)\subset \Xs\times\R_+$ is evanescent (wrt the probability space $(\Xs,\P)$) iff there exists $A\subset \Xs$ such that $\P(A)=(m\otimes \W)(A)=1$ and $r_\sX(A\times \R_+)\cap F=\emptyset.$
\end{definition}

By Corollary \ref{cor:xii rmst}, $\xi_{(f,s_1,\ldots,s_i)}\in\RMST_{n-i}$ for $r_i(\xi^i)$ a.e.\ $(f,s_1,\ldots,s_i)\in\S{i}$. The next lemma says that for typical $(f,s_1,\ldots,s_i)|(h,s)\in\S{i+1}$ this still holds for $\bar\xi_{(f,s_1,\ldots,s_i)|(h,s)}$.

\begin{lemma}\label{lem:taming}
 Let $\xi\in\RMST_n^1$ and fix $0\leq i <n.$
\begin{enumerate}
 \item $\bar\xi_{(f,s_1,\ldots,s_i)|(h,s)}\in\RMST^1_{n-i}$ outside a $r_i(\xi^i)$-evanescent set.  
\item If $F:\S{n}\to\R$ satisfies $\xi(F\circ r_n)<\infty,$ then the set $\{(f,s_1,\ldots,s_i)|(h,s) : \bar\xi_{(f,s_1,\ldots,s_i)|(h,s)}(F^{(f,s_1,\ldots,s_i)|(h,s)\oplus}\circ r_{n-i})=\infty\}$ is $r_i(\xi^i)$-evanescent. In particular, this applies to $F(f,s_1,\ldots,s_n)=s_n$ if $\xi\in\RMST(\mu_0,\ldots,\mu_n).$ 
\end{enumerate}
\end{lemma}
\begin{remark}\label{rem:taming}
 Observe that a direct consequence of Corollary~\ref{cor:xii rmst}, assuming $\xi(F\circ r_n)<\infty$, is that $\{(f,s_1,\ldots,s_i) : \xi_{(f,s_1,\ldots,s_i)}({F}^{(f,s_1,\ldots,s_i)\otimes}\circ r_{n-i})=\infty\}$ is a $r_i(\xi^i)$ null set. 
\end{remark}

\begin{proof}[Proof of Lemma \ref{lem:taming}]
It is apparent that $\xi_{(f,s_1\ldots,s_i)|(h,s)}\in\RMST$.
 By Corollary \ref{cor:xii rmst}, \eqref{eq:ncRMST}, and Remark \ref{rem:taming} it is sufficient to show the claims under the additional hypothesis that $A^\xi_{(f,s_1\ldots,s_i)}(h,s)<1$. Hence, consider
\begin{align*}
 U_1&=\{(f,s_1,\ldots,s_i)|(h,s) : A^\xi_{(f,s_1\ldots,s_i)}(h,s)<1, \bar\xi_{(f,s_1,\ldots,s_i)|(h,s)}\notin\RMST^1_{n-i}\},\\
 U_2&=\{(f,s_1,\ldots,s_i)|(h,s) : A^\xi_{(f,s_1\ldots,s_i)}(h,s)<1, \bar\xi_{(f,s_1,\ldots,s_i)|(h,s)}(F^{(f,s_1,\ldots,s_i)|(h,s)\oplus}\circ r_{n-i})=\infty\}.
\end{align*}
Set $A^\xi_{(f,s_1,\ldots,s_i)}(\omega):=\lim_{s\to\infty} A^\xi_{(f,s_1,\ldots,s_i)}(r(\omega,s))$. Then, $(f,s_1,\ldots,s_i)|(h,s)\in U_1$ is equivalent to $\textstyle \int A^\xi_{(f,s_1,\ldots,s_i)}(h\oplus\omega)~\W(d\omega)<1.$ Set $X=\S{i}$ and $m=r_i(\xi^i)$ and recall that the natural coordinate process on $\Xs$ is denoted by $Y$. Given a $\cG^0$-stopping time $\tau$ on $(\Xs,\cG,\P)$ we have $r_i(\xi^i)$ a.s.~ by the strong Markov property and the fact that $\xi$ is almost surely a finite stopping time: \begin{align*}
 1= & \int A^\xi_{(f,s_1,\ldots,s_i)}(\omega)~\W(d\omega) \\
= & \int \left[\1_{\tau(\omega)=\infty} A^\xi_{(f,s_1,\ldots,s_i)}(\omega)  + \int \1_{\tau(\omega)<\infty} A^\xi_{(f,s_1,\ldots,s_i)}(\omega_{\llcorner [0,\tau]}\oplus\tilde\omega) ~\W(\tilde\omega)\right] ~\W(d\omega),
\end{align*}
implying that $\P[((Y_s)_{s\leq\tau},\tau)\in U_1]=0.$ Hence, the first part follows from the optional section Theorem.

Additionally, setting $\alpha(d(x,\omega),dt)=\delta_{\tau(x,\omega)}(dt) \, \P(d(x,\omega))$ we have
\begin{align*}
 \int_{U_2} dr_\sX(\alpha)(x,h,s)~(1-A^\xi_x(h,s))  ~\bar\xi_{x|(h,s)}(F^{x|(h,s)\oplus})\leq \xi(F)<\infty,
\end{align*}
implying $r_\sX(\alpha)(U_2)=0$. Hence, we have $\P[((Y_s)_{s\leq\tau},\tau)\in U_2]=0$ proving the claim by the optional section theorem, e.g.\ \cite[Theorems IV 84 and IV 85]{DeMeA} (see also Remark \ref{rem:section}).
\end{proof}

\subsection{Colour swaps}
\label{sec:CS}
As a first step towards the definition of stop-go pairs we introduce
an important building block, the colour swap pairs. 
 
By Corollary \ref{cor:canonicalRMST} and Corollary \ref{cor:xii rmst}, for $r_i(\xi^i)$ a.e.\ $(g,t_1,\ldots,t_{i})$ there is an increasing sequence $(\rho^j_{(g,t_1,\ldots,t_i)})_{j=i+1}^n$ of $\bar\cF^a$-stopping times such that
\begin{align*}%\label{eq:STdisint}
 \omega\mapsto \int_{[0,1]^{n-i}} \delta_{\rho_{(g,t_1,\ldots,t_{i})}^{i+1}(\omega,u)}(dt_{i+1})\cdots \delta_{\rho_{(g,t_1,\ldots,t_{i})}^n(\omega,u)}(dt_n)\ du
\end{align*}
defines an $\cF^a$- measurable disintegration of $\xi_{(g,t_1,\ldots,t_{i})}$ w.r.t.\ $\W_{\mu_0}$. Similarly, by Lemma \ref{lem:taming}, outside an $r_{i-1}(\xi^{i-1})$ evanescent set, for $(f,s_1,\ldots,s_{i-1})|(h,s) \in \S{i}_\R$ such that $\bar\xi_{(f,s_1,\ldots,s_{i-1})|(h,s)}\neq \delta_0 \cdots \delta_0$  there is an increasing sequence $(\rho^j_{(f,s_1,\ldots,s_{i-1})|(h,s)})_{j=i}^n$ of $\bar\cF^a$-stopping times such that
\begin{align*}%\label{eq:STdisint}
 \omega\mapsto \int_{[0,1]^{n-i+1}} \delta_{\rho_{(f,s_1,\ldots,s_{i-1})|(h,s)}^{i}(\omega,u)}(ds_{i})\cdots \delta_{\rho_{(f,s_1,\ldots,s_{i-1})|(h,s)}^n(\omega,u)}(ds_n)\ du
\end{align*}
defines an $\F^a$- measurable disintegration of $\bar\xi_{(f,s_1,\ldots,s_{i-1})|(h,s)}$ w.r.t.\ $\W_{\mu_0}$. We make the important observation that, if $A^\xi_{(f,s_1,\ldots,s_{i-1})}(h,s)=1$ (hence in this situation $\Delta A^\xi_{(f,s_1,\ldots,s_{i-1})}(h,s)>0$), we have $\rho_{(f,s_1,\ldots,s_{i-1})|(h,s)}^i\equiv \delta_{0}$. 

This representation allows us to couple the two stopping rules by taking realizations of the $\rho^j_{(g,t_1,\ldots,t_i)}$ stopping times and $\rho_{(f,s_1,\ldots,s_i)|(h,s)}^k$ stopping times on the \emph{same probability space} $\bar\Omega^{f\otimes h,g}:=C(R_+)\times [0,1]^{n-i+1}$ where of course one of the $u$-coordinates is superfluous for the $\rho^j_{(g,t_1,\ldots,t_i)}$ stopping times.  For $(f,s_1,\ldots,s_{i-1}),(h,s)$ and $(g,t_1,\ldots,t_i)$ as above and $n>j\geq i$ we define
\begin{align}\label{eq:Afhg}
\begin{split}
\Lambda_j^{f\otimes h,g}:=\Big\{ (\omega,u)\in \bar\Omega^{f\otimes h,g}~:\  & %\nonumber 
 \rho^j_{(f,s_1,\ldots,s_{i-1})|(h,s)}(\omega,u) \vee \rho^{j}_{(g,t_1,\ldots,t_i)}(\omega,u) \\
 & \le \rho^{j+1}_{(f,s_1,\ldots,s_{i-1})|(h,s)}(\omega,u) \wedge \rho^{j+1}_{(g,t_1,\ldots,t_i)}(\omega,u)\Big\}.
\end{split}
\end{align}
Note that this is the set where it is possible to swap the stopping rules from colour $i$ up to colour $j$ and \emph{not swap} the stopping rule for colours greater than $j$.

The set of colour swap pairs between colour $i$ and $j$, $i \le j < n$, denoted by $\CS^\xi_{i\leftrightarrow j}$ is defined to consist of all $(f,s_1,\ldots,s_{i-1})\in \S{i-1}_\R$, $(h,s)\in S$ and $(g,t_1,\ldots,t_i)\in  \S{i}_\R$ such that $f\oplus h(s_{i-1}+s)=g(t_i)$, $1-A^\xi_{(f,s_1,\ldots,s_{i-1})}(h,s)+\Delta A^\xi_{(f,s_1,\ldots,s_{i-1})}(h,s)>0$, and
\begin{align*}
& \int \gamma^{(f,s_1,\ldots,s_{i-1})|(h,s)\oplus}(\omega,s_i,\ldots,s_n)~\bar\xi_{(f,s_1,\ldots,s_{i-1})|(h,s)}(d\omega,ds_i,\ldots,ds_n)\\
& + \int \gamma^{(g,t_1,\ldots,t_i)\otimes}(\omega,t_{i+1},\ldots,t_{n})~\xi_{(g,t_1,\ldots,t_i)}(d\omega,dt_{i+1},\ldots,dt_n)\\
>& \int \W(d\omega)du  \1_{\Lambda_j^{f\otimes h,g}}(\omega,u) \left[ \int \gamma^{(f,s_1,\ldots,s_{i-1})|(h,s)\otimes}(\omega,t_{i+1},\ldots,t_j,s_{j+1},\ldots,s_n)\right. \\
& \delta_{\rho^{i+1}_{(g,t_1,\ldots,t_i)}(\omega,u)}(dt_{i+1})\cdots \delta_{\rho^{j}_{(g,t_1,\ldots,t_i)}(\omega,u)}(dt_{j})~ \delta_{\rho^{j+1}_{(f,s_1,\ldots,s_{i-1})|(h,s)}(\omega,u)}(ds_{j+1})\cdots \delta_{\rho^{n}_{(f,s_1,\ldots,s_{i-1})|(h,s)}(\omega,u)}(ds_{n})\\
&+ \int \gamma^{(g,t_1,\ldots,t_i)\oplus}(\omega,s_i,\ldots,s_j,t_{j+1},\ldots,t_n)\\
&\delta_{\rho^{i}_{(f,s_1,\ldots,s_{i-1})|(h,s)}(\omega,u)}(ds_{i})\cdots \delta_{\rho^{j}_{(f,s_1,\ldots,s_{i-1})|(h,s)}(\omega,u)}(ds_{j})~\delta_{\rho^{j+1}_{(g,t_1,\ldots,t_i)}(\omega,u)}(dt_{j+1})\cdots \delta_{\rho^{n}_{(g,t_1,\ldots,t_i)}(\omega,u)}(dt_{n}) \Bigg]\\
&+ \int \W(d\omega)du \left(1-\1_{\Lambda_j^{f\otimes h,g}}(\omega,u) \right) \left[ \int  \gamma^{(f,s_1,\ldots,s_{i-1})|(h,s)\otimes}(\omega,t_{i+1},\ldots,t_n) \right.  \\
& \qquad \delta_{\rho^{i+1}_{(g,t_1,\ldots,t_i)}(\omega,u)}(dt_{i+1})\cdots \delta_{\rho^{n}_{(g,t_1,\ldots,t_i)}(\omega,u)}(dt_{n})   \\
&+ \int \gamma^{(g,t_1,\ldots,t_i)\oplus}(\omega,s_i,\ldots,s_n)~\delta_{\rho^{i}_{(f,s_1,\ldots,s_{i-1})|(h,s)}(\omega,u)}(ds_{i})\cdots \delta_{\rho^{n}_{(f,s_1,\ldots,s_{i-1})|(h,s)}(\omega,u)}(ds_{n}) \Bigg]. \numberthis \label{eq:CS}
\end{align*} 
Moreover, we agree that \eqref{eq:CS} holds in each of the following cases
\begin{enumerate}
 \item $\bar\xi_{(f,s_1,\ldots,s_{i-1})|(h,s)} \notin\RMST^1_{n-i+1}, \xi_{(g,t_1,\ldots,t_i)}\notin\RMST^1_{n-i}$;
\item the left hand side is infinite;
\item any of the integrals appearing is not well-defined.
\end{enumerate}

Then we set $\CS_{i}^\xi = \bigcup_{j \ge i} \CS_{i\leftrightarrow j}^\xi$.

\begin{remark}\label{rem:cs}
 \begin{enumerate}
  \item In case that $\Lambda_j^{f\oplus h,g} \neq \bar\Omega^{f\otimes h,g}$ it is not sufficient to only change the colours/stopping rules from colour $i$ to $j$. On the complement of $\Lambda^{f\oplus h,g}$, we have to switch the whole stopping rule from colour $i$ up to colour $n$ in order to stay within the class of randomised multi-stopping times. This is precisely the reason for the two big integrals appearing on the right hand side of the inequality.
\item Recall that $\rho^i_{(f,s_1,\ldots,s_{i-1})|(h,s)}=\delta_0$ is possible so that it might happen that on both sides of \eqref{eq:CS} the stopping rule of colour $i$ is in fact \emph{the same} and we only change the stopping rule from colour $i+1$ onwards.
\item In case of $\CS^\xi_{i\leftrightarrow i}$ the condition $1-A^\xi_{(f,s_1,\ldots,s_{i-1})}(h,s)+\Delta A^\xi_{(f,s_1,\ldots,s_{i-1})}(h,s)>0$ is not needed since there is no colour swap pair (with finite well defined integrals) not satisfying this condition.
 \end{enumerate}
\end{remark}

\subsection{Multi-colour swaps}
\label{sec:MCS}
Having introduced colour swap pairs we can now proceed and combine
different colour swaps into multi-colour swap pairs.  As described
above, the aim is now to swap back \emph{as soon as possible}.
To this end, we consider for fixed $i<n$ the
following partition of $\bar\Omega^{f\otimes h,g}$ defined in such a
way that modifications of stopping rules in accordance to this
partition transform randomised multi-stopping times into randomised
multi-stopping times (c.f.~\eqref{eq:tildetau}).

\begin{align}\label{eq:partition}
 \bar\Omega^{f\otimes h,g}=\bigcup_{j=i}^{n}\left(\Lambda_j^{f\otimes h,g}\setminus\cup_{k=i}^{j-1}\Lambda_k^{f\otimes h,g}\right),
\end{align}
where
 $$\Lambda_n^{f\otimes h,g}:=\left\{\rho_{(g,t_1,\ldots,t_i)}^{i+1}<\rho^i_{(f,s_1,\ldots,s_{i-1})|(h,s)}, \rho_{(g,t_1,\ldots,t_i)}^n < \rho^n_{(f,s_1,\ldots,s_{i-1})|(h,s)}\right\}.$$
This is indeed a partition: The different sets are disjoint by construction. Hence, the right hand side of \eqref{eq:partition} is contained in the left hand side. We need to show that also the converse conclusion holds. Take $(\omega,u)\in\bar\Omega^{f\otimes h,g}$. If 
$$ \rho^{i+1}_{(g,t_1,\ldots,t_i)}(\omega,u)\geq \rho^{i}_{(f,s_1,\ldots,s_{i-1})|(h,s)}(\omega,u),$$
then $(\omega,u)\in \Lambda^{f\oplus h,g}_i$. Otherwise, it holds that
$$ \rho^{i+1}_{(g,t_1,\ldots,t_i)}(\omega,u) < \rho^{i}_{(f,s_1,\ldots,s_{i-1})|(h,s)}(\omega,u)\leq \rho^{i+2}_{(f,s_1,\ldots,s_{i-1})|(h,s)}(\omega,u)$$ 
and either
$$ \rho^{i+2}_{(g,t_1,\ldots,t_i)}(\omega,u)\geq \rho^{i+1}_{(f,s_1,\ldots,s_{i-1})|(h,s)}(\omega,u) $$
or
$$\rho^{i+2}_{(g,t_1,\ldots,t_i)}(\omega,u)< \rho^{i+1}_{(f,s_1,\ldots,s_{i-1})|(h,s)}(\omega,u) \leq \rho^{i+3}_{(f,s_1,\ldots,s_{i-1})|(h,s)}(\omega,u) .$$
In the former case, we have $(\omega,u)\in \Lambda^{f\oplus h,g}_{i+1}\setminus \Lambda^{f\oplus h,g}_i$ and in the latter case we have either
$$ \rho^{i+3}_{(g,t_1,\ldots,t_i)}(\omega,u)\geq \rho^{i+2}_{(f,s_1,\ldots,s_{i-1})|(h,s)}(\omega,u) $$
or
$$\rho^{i+3}_{(g,t_1,\ldots,t_i)}(\omega,u)< \rho^{i+2}_{(f,s_1,\ldots,s_{i-1})|(h,s)}(\omega,u) \leq \rho^{i+4}_{(f,s_1,\ldots,s_{i-1})|(h,s)}(\omega,u) .$$
By induction, the claim follows.  We put $\bar \Lambda_j^{f\otimes h,g}=\Lambda_j^{f\otimes h,g}\setminus\cup_{k=i}^{j-1}\Lambda_k^{f\otimes h,g}.$ Then, the set of all \emph{multi-colour swap pairs starting at colour $i$}, denoted by $\MCS^\xi_i$, is defined to consist of all $(f,s_1,\ldots,s_{i-1}) \in \S{i-1}_\R,(h,s)\in S, (g,t_1,\ldots,t_i)\in \S{i}_\R$ such that $f\oplus h(s_{i-1}+s)=g(t_i)$ and
\begin{align*}
& \int \gamma^{(f,s_1,\ldots,s_{i-1})|(h,s)\oplus}(\omega,s_i,\ldots,s_n)~\bar\xi_{(f,s_1,\ldots,s_{i-1})|(h,s)}(d\omega,ds_i,\ldots,ds_n)\\
& + \int \gamma^{(g,t_1,\ldots,t_i)\otimes}(\omega,t_{i+1},\ldots,t_{n})~\xi_{(g,t_1,\ldots,t_i)}(d\omega,dt_{i+1},\ldots,dt_n)\\
>& \int \W(d\omega)du \sum_{j=i}^{n-1} \left[ \1_{\bar \Lambda_j^{f\otimes h,g}}(\omega,u) \right.  \int \gamma^{(f,s_1,\ldots,s_{i-1})|(h,s)\otimes}(\omega,t_{i+1},\ldots,t_j,s_{j+1},\ldots,s_n) \\
& \delta_{\rho^{i+1}_{(g,t_1,\ldots,t_i)}(\omega,u)}(dt_{i+1})\cdots \delta_{\rho^{j}_{(g,t_1,\ldots,t_i)}(\omega,u)}(dt_{j})~ \delta_{\rho^{j+1}_{(f,s_1,\ldots,s_{i-1})|(h,s)}(\omega,u)}(ds_{j+1})\cdots \delta_{\rho^{n}_{(f,s_1,\ldots,s_{i-1})|(h,s)}(\omega,u)}(ds_{n})\\
&+ \int \gamma^{(g,t_1,\ldots,t_i)\oplus}(\omega,s_i,\ldots,s_j,t_{j+1},\ldots,t_n)\\
&\delta_{\rho^{i}_{(f,s_1,\ldots,s_{i-1})|(h,s)}(\omega,u)}(ds_{i})\cdots \delta_{\rho^{j}_{(f,s_1,\ldots,s_{i-1})|(h,s)}(\omega,u)}(ds_{j})~\delta_{\rho^{j+1}_{(g,t_1,\ldots,t_i)}(\omega,u)}(dt_{j+1})\cdots \delta_{\rho^{n}_{(g,t_1,\ldots,t_i)}(\omega,u)}(dt_{n}) \Bigg]\\
&+ \int \W(d\omega)du \left(1-\sum_{j=i}^{n-1}\1_{\bar \Lambda_j^{f\otimes h,g}}(\omega,u) \right) \left[ \int  \gamma^{(f,s_1,\ldots,s_{i-1})|(h,s)\otimes}(\omega,t_{i+1},\ldots,t_n) \right.  \\
& \qquad \delta_{\rho^{i+1}_{(g,t_1,\ldots,t_i)}(\omega,u)}(dt_{i+1})\cdots \delta_{\rho^{n}_{(g,t_1,\ldots,t_i)}(\omega,u)}(dt_{n})   \\
& + \int \gamma^{(g,t_1,\ldots,t_i)\oplus}(\omega,s_i,\ldots,s_n)~\delta_{\rho^{i}_{(f,s_1,\ldots,s_{i-1})|(h,s)}(\omega,u)}(ds_{i})\cdots \delta_{\rho^{n}_{(f,s_1,\ldots,s_{i-1})|(h,s)}(\omega,u)}(ds_{n}) \Bigg].\numberthis \label{eq:MCS}
\end{align*}
As in the case of colour swaps we agree that \eqref{eq:MCS} holds in each of the following cases
\begin{enumerate}
 \item $\bar\xi_{(f,s_1,\ldots,s_{i-1})|(h,s)} \notin\RMST^1_{n-i+1}, \xi_{(g,t_1,\ldots,t_i)}\notin\RMST^1_{n-i}$;
\item the left hand side is infinite;
\item any of the integrals appearing is not well-defined.
\end{enumerate}

\begin{remark}
\begin{enumerate}
 \item Note that when $\rho^i_{(f,s_1,\ldots,s_{i-1})|(h,s)}\equiv\delta_0$ we have $\bar\Omega^{f\otimes h,g}=\Lambda_i^{f\otimes h,g}$. Inserting this case into \eqref{eq:MCS} we see that both sides agree so that there are no multi-colour swap pairs satisfying $\rho^i_{(f,s_1,\ldots,s_{i-1})|(h,s)}\equiv\delta_0$.
\item Observe that in the definition of $\MCS^\xi_i$ we do not need to impose the condition $1-A^\xi_{(f,s_1,\ldots,s_{i-1})}(h,s)+\Delta A^\xi_{(f,s_1,\ldots,s_{i-1})}(h,s)>0$ by Remark  \ref{rem:cs}.
\end{enumerate}

\end{remark}

\subsection{Stop-go pairs}\label{sec:SGPairs}
Finally, we combine the previous two notions.

\begin{definition}
  \label{def:SG}
  Let $\xi \in \RMST_n^1(\R,\mu_0)$. The set of stop-go pairs of colour $i$ relative to $\xi$, $\SG^\xi_i$, is defined by $\SG^\xi_i = \CS_i^\xi \cup \MCS_i^\xi$. We define the stop-go pairs of colour $i$ in the wide sense by $\widehat{\SG}^\xi_i=\SG^\xi_i\cup \{(f,s_1,\ldots,s_{i-1})|(h,s)\in \S{i}_\R:A^\xi_{(f,s_1,\ldots,s_{i-1})}(h,s)=1\}\times\S{i}_\R.$

 The set of stop-go pairs relative to $\xi$ is defined by $\SG^\xi := \bigcup_{1\leq i\leq n} \SG_i^\xi$. The stop-go pairs in the wide sense are $\widehat{\SG}^\xi:= \bigcup_{1\leq i\leq n} \widehat{\SG}_i^\xi$.
\end{definition}

\begin{lemma}
 Every stop-go pair is a stop-go pair in the wide sense, i.e.\  $\SG_i\subset \widehat{\SG}_i^\xi$ for any $1\leq i\leq n$.
\end{lemma}
\begin{proof}
By loading notation, this follows using exactly the same argument as for the proof of \cite[Lemma 5.4]{BeCoHu14}.
\end{proof}

\begin{remark}
  As in \cite{BeCoHu14}, we observe that the sets $\SG^\xi$ and $\widehat{\SG}^\xi$
  are both Borel subsets of $\S{i} \times \S{i}$, since the maps given
  in e.g.\ \eqref{eq:ximeas} are measurable. In contrast, the set $\SG$ is in general just co-analytic.
\end{remark}

\section{The monotonicity principle}\label{sec:monotonicity}

The aim of this section is to prove the main results, Theorem~\ref{thm:monprin2} and the closely related Theorem~\ref{thm:monotonicity principle}. The structure of this section follows closely the structure of the proof of the corresponding results, Theorem~5.7 (resp. Theorem~5.16), in \cite{BeCoHu14}. For the benefit of the reader, and to keep our presentation compact, we concentrate on those aspects of the proof where additional insight is needed to account for the multi-marginal aspects of the problem. We refer the reader to \cite{BeCoHu14} for other details.

The essence of the proof is to first show that if we have a candidate optimiser $\xi$, and a joining rule $\pi$ which identifies stop-go pairs, we can construct an infinitesimal improvement $\xi^\pi$, which will also be a candidate solution, but which will improve the objective. It will follow that the joining $\pi$ will place no mass on the set of stop-go pairs. The second part of the proof shows that we can strengthen this to give a pointwise result, where we can exclude any stop-go pair from a set related to the support of the optimiser. 

\textbf{Important convention:} Throughout this section, we fix a function $\gamma:\S{n}\to\R$ and a measure $\xi\in\RMST^1(\mu_0,\mu_1,\ldots,\mu_n)$. Moreover, for each $0\leq i\leq n-1$ we fix a disintegration $(\xi_{(f,s_1,\ldots,s_i),\omega})_{(f,s_1,\ldots,s_i),\omega}$  of $\tilde{r}_{n,i}(\xi)$ w.r.t.\ $\tilde r_i(\xi^i)$ and set $\xi_{(f,s_1,\ldots,s_i)}=\xi_{(f,s_1,\ldots,s_i),\omega}\W(d\omega).$

Recall the map $\proj_{\S{i}}$ from Section \ref{sec:probint}.

\begin{definition}
 A family of Borel
 sets $\Gamma=(\Gamma_1,\ldots,\Gamma_n)$ with $\Gamma_i\subset \S{i}_\R$ is called $(\gamma,\xi)$-monotone iff for all $1\leq i \leq n$
$$ \widehat{\SG}^{\xi}_i\cap\left(\Gamma_i^<\times\Gamma_i\right)=\emptyset,$$
where
$$\Gamma_i^<=\{(f,s_1,\ldots,s_{i-1},s): \text{ there exists } (g,s_1,\ldots,s_{i-1},t)\in\Gamma_i, s_{i-1}\leq s <t, g_{\llcorner [0,s]}=f\},$$
and $\proj_{\S{i-1}}(\Gamma_i)\subset\Gamma_{i-1}$.
\end{definition}

\begin{theorem}\label{thm:monotonicity principle}
Assume that $\gamma:\S{n}\to\R$ is Borel measurable. Assume that \eqref{eq:Pgamma} is well posed and that $\xi\in\RMST(\mu_0,\ldots,\mu_1)$ is an optimizer. Then there exists a  $(\gamma,\xi)$-monotone family of Borel sets $\Gamma=(\Gamma_1,\ldots,\Gamma_n)$ such that $r_i(\xi)(\Gamma_i)=1$ for each $1\leq i\leq n$.
\end{theorem}

The proof of Theorem \ref{thm:monotonicity principle} is based on the following two propositions.

\begin{proposition}\label{prop:modification}
 Let $\gamma:\S{n}_\R\to\R$ be Borel. Assume that \eqref{eq:Pgamma} is well posed and that $\xi\in\RMST(\mu_0,\ldots,\mu_1)$ is an optimizer. Fix $1\leq i \leq n$ and set $\sX=\S{i-1}_\R, m=r_{i-1}(\xi^{i-1})$. Then
$$(r_\sX\otimes \id)(\pi)(\SG_i^\xi)=0$$
for all $\pi\in \TRT(r_{i-1}(\xi^{i-1}),r_i(\xi^i)).$
\end{proposition}

\begin{proposition}\label{prop:kele}
 Let $(\sX,m)$ and $(\Ysf, \nu)$ be  Polish probability spaces and  $E\subseteq S_\sX\times \Ysf$ a Borel set. Then the following are equivalent:
  \begin{enumerate}
  \item $(r_\sX \otimes \id)(\pi)(E)=0$ for all $\pi\in \TRST^1(m, \nu )$.
  \item $E \subseteq (F \times \Ys)\ \cup\ (S_\sX\times N)$ for some evanescent set $F\subset S_\sX$ and a $\nu$-null set $N\subseteq \Ysf$.
  \end{enumerate}
\end{proposition}
\begin{proof}
 This is a straightforward modification of \cite[Proposition 5.9]{BeCoHu14} to the case of a general starting law (see also the proof of \cite[Theorem~7.4]{BeCoHu14}).
\end{proof}

\begin{remark}\label{rem:section}
 Note that Proposition \ref{prop:kele} is closely related to the classical section theorem (cf.\ \cite[Theorems IV 84 and IV 85]{DeMeA}) which in our setup implies the following statement:

Let $(X,\mathcal B, m)$ be a Polish probability space. $E\subset S_\sX$ be Borel. Then the following are equivalent:
\begin{enumerate}
 \item $r_\sX(\alpha)(E)=0$ for all $\alpha\in\RST(\sX,m)$
\item $E$ is $m$-evanescent
\item $\P(((Y_s)_{s\leq\tau},\tau)\in E)=0$ for every $\cG^0$-stopping time $\tau$.
\end{enumerate}

\end{remark}

\begin{proof}[Proof of Theorem \ref{thm:monotonicity principle}]
 Fix $1\leq i\leq n$. Set $\sX=\S{i-1}_\R, m=r_{i-1}(\xi^{i-1})$ and consider the corresponding probability space $(\Xs,\P).$ By Proposition \ref{prop:modification} $(r_\sX\otimes \id)(\pi)(\SG_i^\xi)=0$ for all $\pi\in\TRST^1(r_{i-1}(\xi^{i-1}),r_i(\xi^i))$. Applying Proposition \ref{prop:kele} with $\Ysf=\S{i}_\R, \nu=r_i(\xi^i)$ we deduce that there exists a $r_{i-1}(\xi^{i-1})$-evanescent set $\tilde F_i$ and a $r_i(\xi^i)$-null set $N_i$ such that 
$$ \SG_i^\xi \subseteq (\tilde F_i \times \S{i}_\R) \cup (\S{i}_\R \times N_i).$$
Put $F_i:=\{(g,t_1,\ldots,t_i)\in \S{i}_\R : \exists (f,t_1,\ldots,t_{i-1},s_i)\in \tilde F_i,t_i\geq s_i, g\equiv f \text{ on } [0,s_i]\}.$ Then, $F_i$ is $r_{i-1}(\xi^{i-1})$-evanescent and  
$$ \SG_i^\xi \subseteq (F_i \times \S{i}_\R) \cup (\S{i}_\R \times N_i).$$ Setting $\tilde \Gamma_i=\S{i}_\R\setminus (F_i\cup N_i)$ we have $r_i(\xi^i)(\tilde \Gamma_i)=1$ as well as $\SG^\xi_i\cap (\tilde\Gamma_i^<\times\tilde\Gamma_i)=\emptyset.$ Define
$$\Gamma_i:=\tilde\Gamma_i\cap\{(g,t_1,\ldots,t_i)\in \S{i}_\R : A^\xi_{(g,t_1,\ldots,t_{i-1})}(\theta_{t_{i-1}}(g)_{\llcorner [0,s]},s)<1 \text{ for all } s<t_i-t_{i-1}\},$$
where $\theta_u(g)(\cdot)=g(\cdot+u)-g(u)$ as usual. Then $r_i(\xi^i)(\Gamma_i)=1$ and $\Gamma_i^<\cap \{(g,t_1,\ldots,t_i)\in \S{i}_\R : A^\xi_{(g,t_1,\ldots,t_{i-1})}(\theta_{t_{i-1}}(g)_{\llcorner [0,t_i-t_{i-1}]},t_i-t_{i-1})=1\}=\emptyset$ so that $\widehat{\SG}^\xi_i\cap(\Gamma_i^< \times\Gamma_i)=\emptyset.$ Finally, we can take a Borel subset of $\Gamma_i$ with full measure and taking suitable intersections we can assume that $\proj_{\S{i-1}_\R}(\Gamma_i)\subseteq \Gamma_{i-1}.$
\end{proof}

\subsection{Proof of Proposition \ref{prop:modification}}

For notational convenience we will only prove the statement for the colour swap pairs $\CS_i^\xi$. As the colour swap pairs are the main building block for the multi-colour swap pairs $\MCS^\xi_i$ it will be immediate how to adapt the proof for the general case. Moreover, it is clearly sufficient to show that for every $j \geq i$ we have $(r_\sX\otimes \id)(\pi)(\CS^\xi_{i\leftrightarrow j})=0$ for each $\pi \in \TRT(r_{i-1}(\xi^{i-1}),r_i(\xi^i)).$

Working towards a contradiction, we assume that there is an index $i\leq j\leq n$ and $\pi\in \TRT(r_{i-1}(\xi^{i-1}),r_i(\xi^i))$ such that $(r_\sX\otimes \id)(\pi)(\CS^\xi_{i\leftrightarrow j})>0$. By the definition of joinings (Definition \ref{def:joinings}) also $\pi_{\llcorner (r_\sX \otimes \id)^{-1}(E)}\in\TRT(r_{i-1}(\xi^{i-1}),r_i(\xi^i))$ for any $E\subset \S{i-1}_\R\times \S{i}_\R.$ Hence, by considering $(r_\sX\otimes \id)(\pi)_{\llcorner \CS^\xi_{i \leftrightarrow j}}$ we may assume that $(r_\sX\otimes \id)(\pi)$ is concentrated on $\CS^\xi_{i\leftrightarrow j}$. Recall that (by definition) there is no colour swap pair $((f,s_1,\ldots,s_{i-1})|(h,s),(g,t_1,\ldots,t_i))$ with $A^\xi_{(f,s_1,\ldots,s_{i-1})}(h,s)=1$ and $\Delta A^\xi_{(f,s_1,\ldots,s_{i-1})}(h,s)=0$. Hence,
\begin{align}\label{eq:pi conc}
\pi[((f,s_1,\ldots,s_{i-1})|(h,s),(g,t_1,\ldots,t_i)) : A^\xi_{(f,s_1,\ldots,s_{i-1})}(h,s)=1 \text{ and } \Delta A^\xi_{(f,s_1,\ldots,s_{i-1})}(h,s)=0]=0. 
\end{align}

We argue by contradiction and define two modifications of $\xi$, $\xi^E$ and $\xi^L$, based on the definition of $\CS_{i \leftrightarrow j}^\xi$ such that their convex combination yields a randomised multi-stopping time embedding the same measures as $\xi$ and leading to strictly less cost. The stopping time $\xi^E$ will stop paths earlier than $\xi$, and $\xi^L$ will stop paths later than $\xi$.

By Lemma \ref{lem:taming} and Corollary~\ref{cor:canonicalRMST}, for $(f,s_1,\ldots,s_{i-1})|(h,s)$ outside an $r_{i-1}(\xi^{i-1})$ - evanescent set and $(g,t_1,\ldots,t_i)$ outside an $r_i(\xi^i)$ null set there are  increasing sequences of $\bar\cG^a$-stopping times  $(\rho^j_{(f,s_1,\ldots,s_{i-1})|(h,s)})_{j=i}^n$ and $(\rho^j_{(g,t_1,\ldots,t_i)})_{j=i+1}^n$ defining $\bar\cG^a$-measurable disintegrations of $\xi_{(f,s_1,\ldots,s_{i-1})|(h,s)}$ and $\xi_{(g,t_1,\ldots,t_i)}$ as in \eqref{eq:STdisint}.

For $B\subset C(\R_+)\times \Xi^n$ and $(g,t_1,\ldots,t_i)\in \S{i}_\R$ we set 
$$ B^{(g,t_1,\ldots,t_i)\oplus}:= \{(\omega,T_i,\ldots,T_n)\in C(\R_+)\times \Xi^{n-i+1}~:~(g\oplus \omega,t_1,\ldots,t_{i-1},t_i+T_i,\ldots,t_i+T_n)\in B\}$$
and
$$B^{(g,t_1,\ldots,t_i)\otimes}:=\{(\omega,T_{i+1},\ldots,T_n)\in C(\R_+)\times \Xi^{n-i}~:~(g\oplus\omega,t_1,\ldots,t_i,t_i+T_{i+1},\ldots,t_i+T_n)\in B\}.$$
Observe that both $B^{(g,t_1,\ldots,t_i)\oplus}$ and $B^{(g,t_1,\ldots,t_i)\otimes}$ are then Borel.  Note that if we define $F(g,t_1, \dots, t_n) = \indic_{B}(g,t_1, \dots, t_n)$, then $F^{(f,s_1,\ldots,s_i)\oplus} (\eta,t_{i+1},\ldots,t_n) = \indic_{B^{(f,s_1,\ldots,s_i)\oplus}}(\eta,t_{i+1},\ldots,t_n)$, and similarly for $B^{(g,t_1,\ldots,t_i)\otimes}$.  Observe, that for $(f,s_1,\ldots,s_{i-1})|(h,s)$ with $A^\xi_{(f,s_1,\ldots,s_{i-1})}(h,s)=1$ and $\Delta A^\xi_{(f,s_1,\ldots,s_{i-1})}(h,s)>0$ it follows from \eqref{eq:cRMST} that
$$ \xi_{(f,s_1,\ldots,s_{i-1})|(h,s)}(B^{(g,t_1,\ldots,t_i) \oplus}) = \Delta A^\xi_{(f,s_1,\ldots,s_{i-1})}(h,s)  \xi_{(f\oplus h,s_1,\ldots,s_{i-1},s_{i-1}+s)}(B^{(g,t_1,\ldots,t_i)\otimes}).$$
Then, we define the measure $\xi^E$ by setting for $B\subset C(\R_+)\times \Xi^n$ (recall $\Lambda_j^{f\otimes h,g}$ from \eqref{eq:Afhg}) 
\begin{align*}
& \xi^E(B):= \numberthis \label{eq:xiE} \\
&\xi(B)
 - \int \xi_{(f,s_1,\ldots,s_{i-1})|(h,s)}(B^{(f\oplus h,s_1,\ldots,s_{i-1},s_{i-1}+s)\oplus})~ (r_\sX\otimes \id)(\pi)(d((f,s_1,\ldots,s_{i-1})|(h,s)),d(g,t_1,\ldots,t_i))\\
& + \int  (r_\sX\otimes \id)(\pi)(d((f,s_1,\ldots,s_{i-1})|(h,s)),d(g,t_1,\ldots,t_i))\\
& \left(1-A^\xi_{(f,s_1,\ldots,s_{i-1})}(h,s) +   \1_{A^\xi_{(f,s_1,\ldots,s_{i-1})}(h,s)=1}\Delta A^\xi_{(f,s_1,\ldots,s_{i-1})}(h,s)\right)\\ &\quad \left[\int_{C(\R_+)}\int_{[0,1]^{n-i+1}} \1_{\Lambda_j^{f\otimes h,g}}(\omega,u) \1_{B^{(f\oplus h,s_1,\ldots,s_{i-1},s_{i-1}+s)\otimes}}\left(\omega,\rho^{i+1}_{(g,t_1,\ldots,t_i)}(\omega,u),\ldots,\right.\right. \\
&\hspace{2cm} \left. \rho^j_{(g,t_1,\ldots,t_i)}(\omega,u),\rho^{j+1}_{(f,s_1,\ldots,s_{i-1})|(h,s)}(\omega,u),\ldots,\rho^{n}_{(f,s_1,\ldots,s_{i-1})|(h,s)}(\omega,u)\right)~\W(d\omega)~du\\
&+  \int_{C(\R_+)}\int_{[0,1]^{n-i+1}}\left(1- \1_{\Lambda_j^{f\otimes h,g}}(\omega,u) \right) \\
&\quad \1_{B^{(f\oplus h,s_1,\ldots,s_{i-1},s_{i-1}+s)\otimes}}\left(\omega,\rho^{i+1}_{(g,t_1,\ldots,t_i)}(\omega,u),\ldots,\rho^n_{(g,t_1,\ldots,t_i)}(\omega,u)\right)\W(d\omega)~du\Bigg].
\end{align*}
Similarly we define the measure $\xi^L$ by  setting for $B\subset C(\R_+)\times \Xi^n$
\begin{align*}
& \xi^L(B):= \numberthis \label{eq:xiL} \\
&\xi(B)
 - \int \left(1-A^\xi_{(f,s_1,\ldots,s_{i-1})}(h,s) +   \1_{A^\xi_{(f,s_1,\ldots,s_{i-1})}(h,s)=1}\Delta A^\xi_{(f,s_1,\ldots,s_{i-1})}(h,s)\right) \xi_{(g,t_1,\ldots,t_i)}( B^{(g,t_1,\ldots,t_i)\otimes})\\
&\qquad (r_\sX\otimes \id)(\pi)(d((f,s_1,\ldots,s_{i-1})|(h,s)),d(g,t_1,\ldots,t_i))\\
& + \int \left(1-A^\xi_{(f,s_1,\ldots,s_{i-1})}(h,s) +   \1_{A^\xi_{(f,s_1,\ldots,s_{i-1})}(h,s)=1}\Delta A^\xi_{(f,s_1,\ldots,s_{i-1})}(h,s)\right)\\
&\quad \left[\int_{C(\R_+)}\int_{[0,1]^{n-i+1}} \1_{\Lambda_j^{f\otimes h,g}}(\omega,u) \1_{B^{(g,t_1,\ldots,t_{i})\oplus}}\left(\omega,\rho^{i}_{(f,s_1,\ldots,s_{i-1})|(h,s)}(\omega,u),\ldots,\right.\right. \\
&\hspace{2cm} \left. \rho^j_{(f,s_1,\ldots,s_{i-1})|(h,s)}(\omega,u),\rho^{j+1}_{(g,t_1,\ldots,t_i)}(\omega,u),\ldots,\rho^{n}_{(g,t_1,\ldots,t_i)}(\omega,u)\right)~\W(d\omega)~du\\
&+  \int_{C(\R_+)}\int_{[0,1]^{n-i+1}}\left(1- \1_{\Lambda_j^{f\otimes h,g}}(\omega,u) \right)\\
&\quad \left. \1_{B^{(g,t_1,\ldots,t_i)\oplus}}\left(\omega,\rho^{i}_{(f,s_1,\ldots,s_{i-1})|(h,s)}(\omega,u),\ldots,\rho^n_{(f,s_1,\ldots,s_{i-1})|(h,s)}(\omega,u)\right)\W(d\omega)~du\right]\\
&\quad (r_\sX\otimes \id)(\pi)(d((f,s_1,\ldots,s_{i-1})|(h,s)),d(g,t_1,\ldots,t_i)).
\end{align*}

Then, we define a competitor of $\xi$ by  $\xi^\pi:=\frac12(\xi^E+\xi^L)$. We will show that $\xi^\pi\in \RMST(\mu_0,\ldots,\mu_n)$ and $\int \gamma ~d\xi > \int \gamma ~ d\xi^\pi$ which contradicts optimality of $\xi.$ 

First of all note that from the definition of $\Lambda_j^{f\oplus h, g}$ in \eqref{eq:Afhg} both $\xi^E,\xi^L\in\RMST$ (also compare \eqref{eq:tildetau}). Hence, also $\xi^\pi\in\RMST$. Next we show that $\xi^\pi\in\RMST(\mu_0,\ldots,\mu_n)$.

For bounded and measurable $F: C(\R)\times \Xi^n\to\R$ \eqref{eq:xiE} and \eqref{eq:xiL} imply by using \eqref{eq:cRMST} and \eqref{eq:ncRMST}
\begin{align*}
& 2\int F ~d(\xi-\xi^\pi)\numberthis \label{eq:xipi F}\\
=& \int (r_\sX\otimes r_i)(\pi)(d((f,s_1,\ldots,s_{i-1})|(h,s)),d(g,t_1,\ldots,t_i))\\
& \quad \left(1-A^\xi_{(f,s_1,\ldots,s_{i-1})}(h,s) + \1_{A^\xi_{(f,s_1,\ldots,s_{i-1})}(h,s)=1} \Delta A^\xi_{(f,s_1,\ldots,s_{i-1})}(h,s)\right) \times\\
&\left(\int F^{(f,s_1,\ldots,s_{i-1})|(h,s) \oplus}(\omega,S_i,\ldots,S_n)~\bar\xi_{(f,s_1,\ldots,s_{i-1})|(h,s)}(d\omega,dS_i,\ldots,dS_n)\right.\\
& +\int F^{(g,t_1,\ldots,t_i)\otimes}(\omega,T_{i+1},\ldots,T_n)~\xi_{(g,t_1,\ldots,t_i)}(d\omega,dT_{i+1},\ldots,dT_n)\\
&- \int_{C(\R_+)}\int_{[0,1]^{n-i+1}} \W(d\omega)du\1_{\Lambda_j^{f\otimes h,g}}(\omega,u) \\
&\quad\left[\int  F^{(f,s_1,\ldots,s_{i-1})|(h,s)\otimes}(\omega,T_{i+1},\ldots, T_j,S_{j+1},\ldots,S_n) \right.\\
&\qquad \delta_{\rho^{i+1}_{(g,t_1,\ldots,t_i)}(\omega,u)}(dT_{i+1})\cdots \delta_{\rho^{j}_{(g,t_1,\ldots,t_i)}(\omega,u)}(dT_{j})~ \delta_{\rho^{j+1}_{(f,s_1,\ldots,s_{i-1})|(h,s)}(\omega,u)}(dS_{j+1})\cdots \delta_{\rho^{n}_{(f,s_1,\ldots,s_{i-1})|(h,s)}(\omega,u)}(dS_{n})\\
&- \int F^{(g,t_1,\ldots,t_i)\oplus}(\omega,S_i,\ldots,S_j,T_{j+1},\ldots,T_n)\\
&\delta_{\rho^{i}_{(f,s_1,\ldots,s_{i-1})|(h,s)}(\omega,u)}(dS_{i})\cdots \delta_{\rho^{j}_{(f,s_1,\ldots,s_{i-1})|(h,s)}(\omega,u)}(dS_{j})~\delta_{\rho^{j+1}_{(g,t_1,\ldots,t_i)}(\omega,u)}(dT_{j+1})\cdots \delta_{\rho^{n}_{(g,t_1,\ldots,t_i)}(\omega,u)}(dT_{n}) \Bigg]\\
&- \int\W(d\omega)du \left(1-\1_{\Lambda_j^{f\otimes h,g}}(\omega,u) \right) \left[ \int  F^{(f,s_1\ldots,s_{i-1}|(h,s)\otimes}( \omega,T_{i+1},\ldots,T_n) \right.  \\
& \qquad \delta_{\rho^{i+1}_{(g,t_1,\ldots,t_i)}(\omega,u)}(dT_{i+1})\cdots \delta_{\rho^{n}_{(g,t_1,\ldots,t_i)}(\omega,u)}(dT_{n})   \\
 &- \int F^{(g,t_1,\ldots,t_i)\oplus}(\omega,S_i,\ldots,S_n)~\delta_{\rho^{i}_{(f,s_1,\ldots,s_{i-1})|(h,s)}(\omega,u)}(dS_{i})\cdots \delta_{\rho^{n}_{(f,s_1,\ldots,s_{i-1})|(h,s)}(\omega,u)}(dS_{n}) \Bigg]\Bigg) ~.
 \end{align*} 
Next we show that \eqref{eq:xipi F} extends to nonnegative $F$ satisfying $\xi(F)<\infty$ in the sense that it is well defined with a value in $[-\infty,\infty)$.
To this end, we will show that 
\begin{align*}
 & \int (r_\sX\otimes r_i)(\pi)(d((f,s_1,\ldots,s_{i-1})|(h,s)),d(g,t_1,\ldots,t_i))\\
& \quad \left(1-A^\xi_{(f,s_1,\ldots,s_{i-1})}(h,s) + \1_{A^\xi_{(f,s_1,\ldots,s_{i-1})}(h,s)=1} \Delta A^\xi_{(f,s_1,\ldots,s_{i-1})}(h,s)\right) \times \numberthis \label{eq:integrable}\\
&\left[\int F^{(f,s_1,\ldots,s_{i-1})|(h,s) \oplus}( \omega,S_i,\ldots,S_n)~\bar\xi_{(f,s_1,\ldots,s_{i-1})|(h,s)}(d\omega,dS_i,\ldots,dS_n)\right.\\
& +\int F^{(g,t_1,\ldots,t_i)\otimes}(\omega,T_{i+1},\ldots,T_n)~\xi_{(g,t_1,\ldots,t_i)}(d\omega,dT_{i+1},\ldots,dT_n)\Bigg]<\infty
\end{align*}
Since $\pi\in\TRT(r_{i-1}(\xi^{i-1}),r_i(\xi^i))$ the integral of the second term in the square brackets in \eqref{eq:integrable} is bounded by $\xi(F)$, and hence is finite. To see that the first term is also finite, write $\proj_{\Xs\times\R_+}(\pi)=:\pi_1$ and note that $\pi_1\in\RST(\S{i-1}_\R,r_{i-1}(\xi^{i-1}))$. Hence, the disintegration $(\pi_1)_{(f,s_1,\ldots,s_{i-1})}$  of $\pi_1$ wrt $r_{i-1}(\xi^{i-1})$ is a.s.\ in $\RST$. Fix $(f,s_1,\ldots,s_{i-1})\in \S{i-1}_\R$ and assume $\alpha:=(\pi_1)_{(f,s_1,\ldots,s_{i-1})}\in\RST$ (which holds on a set of measure one). In case that $\alpha_\omega(\R_+)<1$ we extend it to a probability on $[0,\infty]$ by adding an atom at $\infty$. We denote the resulting randomized stopping time still by $\alpha$.  Then, we can calculate using the strong Markov property of the Wiener measure for the first equality and $F\geq 0$ for the first inequality
\begin{align*}
 & \int F^{(f,s_1,\ldots,s_{i-1})\otimes}(\omega,s_i,\ldots,s_n)~ \xi_{(f,s_1,\ldots,s_{i-1})}(d\omega,ds_i,\ldots,ds_n) \\
=&  \iint F^{(f,s_1,\ldots,s_{i-1})\otimes}(\omega_{\llcorner [0,t]}\oplus \theta_t\omega,s_i,\ldots,s_n)~ (\xi_{(f,s_1,\ldots,s_{i-1})})_{\omega_{\llcorner [0,t]}\oplus \theta_t\omega}(ds_i,\ldots,ds_n)  \alpha_{\omega}(dt) \W(d\omega) \\
=&  \iint  F^{(f,s_1,\ldots,s_{i-1})\otimes}(\omega_{\llcorner [0,t]}\oplus \tilde\omega,s_i,\ldots,s_n)~ (\xi_{(f,s_1,\ldots,s_{i-1})})_{\omega_{\llcorner [0,t]}\oplus \tilde\omega}(ds_i,\ldots,ds_n)  \alpha_{\omega}(dt) \W(d\omega) \W(d\tilde\omega) \\
\geq &   \iint \1_{\{(\omega,t) : t \leq s_i<\infty\}} F^{(f,s_1,\ldots,s_{i-1})\otimes}(\omega_{\llcorner [0,t]}\oplus \tilde\omega,s_i,\ldots,s_n)~ (\xi_{(f,s_1,\ldots,s_{i-1})})_{\omega_{\llcorner [0,t]}\oplus \tilde\omega}(ds_i,\ldots,ds_n)  \alpha_{\omega}(dt) \W(d\omega) \W(d\tilde\omega) \\
=& \iint \1_{\{(\omega,t) : t <\infty\}} F^{(f,s_1,\ldots,s_{i-1})|r(\omega,t)\oplus}( \tilde \omega,s'_i,\ldots,s'_n)~ \xi_{(f,s_1,\ldots,s_{i-1})|r(\omega,t)}(d\tilde \omega,ds_i',\ldots,ds_n')  \alpha(d\omega,dt) 
\end{align*}
 Hence,
\begin{align*}
& \int (r_\sX\otimes r_i)(\pi)(d((f,s_1,\ldots,s_{i-1})|(h,s)),d(g,t_1,\ldots,t_i))\\
& \left(1-A^\xi_{(f,s_1,\ldots,s_{i-1})}(h,s) + \1_{A^\xi_{(f,s_1,\ldots,s_{i-1})}(h,s)=1} \Delta A^\xi_{(f,s_1,\ldots,s_{i-1})}(h,s)\right) \times \\
&\int F^{(f,s_1,\ldots,s_{i-1})|(h,s) \oplus}( \omega,S_i,\ldots,S_n)~\bar\xi_{(f,s_1,\ldots,s_{i-1})|(h,s)}(d\omega,dS_i,\ldots,dS_n)\\
= & \int (r_\sX\otimes r_i)(\pi)(d((f,s_1,\ldots,s_{i-1})|(h,s)),d(g,t_1,\ldots,t_i)) \times\\
&\int F^{(f,s_1,\ldots,s_{i-1})|(h,s) \oplus}( \omega,S_i,\ldots,S_n)~\xi_{(f,s_1,\ldots,s_{i-1})|(h,s)}(d\omega,dS_i,\ldots,dS_n) \\
 \leq& \iint F^{(f,s_1,\ldots,s_{i-1})\otimes}(\omega,s_i,\ldots,s_n)~ \xi_{(f,s_1,\ldots,s_{i-1})}(d\omega,ds_i,\ldots,ds_n) ~ r_{i-1}(\xi^{i-1})(d((f,s_1,\ldots,s_{i-1}))\\
= &\  \xi(F) < \infty~.
\end{align*}
Applying \eqref{eq:xipi F} to $\mathfrak t_n(\omega,t_1,\ldots,t_n)=t_n$, and observing that all the terms on the right-hand side cancel implies that $\xi^\pi(\mathfrak t_n)=\xi(\mathfrak t_n)<\infty.$ Taking $F(\omega,s_1,\ldots,s_n)=G(\omega(s_j))$ for $0\leq j\leq n$ with $s_0:=0$ for bounded and measurable $G:\R\to\R$ the right hand side vanishes. For $j<i$ this follows since $\xi^{i-1}=(\xi^\pi)^{i-1}$ as we have not changed any stopping rule for colours prior to $i.$ For $j\geq i$ this follows from the fact that
 $\pi$ is concentrated on pairs $((f,s_1,\ldots,s_{i-1})|(h,s),(g,t_1,\ldots,t_i))$ satisfying $f\oplus h(s_i+s)=g(t_i)$. Hence, we have shown that $\xi^\pi\in\RMST(\mu_0,\ldots,\mu_n).$

Using that $\xi^\pi(\gamma^-),\xi(\gamma^-)<\infty$, by well posedness of \eqref{eq:Pgamma}, we can apply \eqref{eq:xipi F} to $F=\gamma$ to obtain by the Definition of $\CS^\xi_{i\leftrightarrow j}$ that 
\begin{align*}
&\int (\gamma\circ r_n)^{(f,s_1,\dots,s_{i-1})|(h,s) \oplus}(\omega,S_i,\ldots,S_n)~\bar\xi_{(f,s_1,\ldots,s_{i-1})|(h,s)}(d\omega,dS_i,\ldots,dS_n)\\
& +\int (\gamma\circ r_n)^{(g,t_1,\ldots,t_i)\otimes}(\omega,T_{i+1},\ldots,T_n)~\xi_{(g,t_1,\ldots,t_i)}(d\omega,dT_{i+1},\ldots,dT_n)\\
&- \int_{C(\R_+)}\int_{[0,1]^{n-i+1}} \1_{\Lambda_j^{f\otimes h,g}}(\omega,u) \left[\int  (\gamma\circ r_n)^{(f,s_1,\ldots,s_{i-1})|(h,s)\otimes}(\omega,T_{i+1},\ldots, T_j,S_{j+1},\ldots,S_n) \right.\\
&\qquad \delta_{\rho^{i+1}_{(g,t_1,\ldots,t_i)}(\omega,u)}(dT_{i+1})\cdots \delta_{\rho^{j}_{(g,t_1,\ldots,t_i)}(\omega,u)}(dT_{j})~ \delta_{\rho^{j+1}_{(f,s_1,\ldots,s_{i-1})|(h,s)}(\omega,u)}(dS_{j+1})\cdots \delta_{\rho^{n}_{(f,s_1,\ldots,s_{i-1})|(h,s)}(\omega,u)}(dS_{n})\\
&- \int (\gamma\circ r_n)^{(g,t_1,\ldots,t_i)\oplus}(\omega,S_i,\ldots,S_j,T_{j+1},T_n)\\
&\delta_{\rho^{i}_{(f,s_1,\ldots,s_{i-1})|(h,s)}(\omega,u)}(dS_{i})\cdots \delta_{\rho^{j}_{(f,s_1,\ldots,s_{i-1})|(h,s)}(\omega,u)}(dS_{j})~\delta_{\rho^{j+1}_{(g,t_1,\ldots,t_i)}(\omega,u)}(dT_{j+1})\cdots \delta_{\rho^{n}_{(g,t_1,\ldots,t_i)}(\omega,u)}(dT_{n}) \Bigg]\W(d\omega)du\\
&- \int \left(1-\1_{\Lambda_j^{f\otimes h,g}}(\omega,u) \right) \left[ \int  (\gamma\circ r_n)^{(f,s_1,\ldots,s_{i-1})|(h,s)\otimes}( \omega,T_{i+1},\ldots,T_n) \right.  \\
& \qquad \delta_{\rho^{i+1}_{(g,t_1,\ldots,t_i)}(\omega,u)}(dT_{i+1})\cdots \delta_{\rho^{n}_{(g,t_1,\ldots,t_i)}(\omega,u)}(dT_{n})   \\
& - \int (\gamma\circ r_n)^{(g,t_1,\ldots,t_i)\oplus}(\omega,S_i,\ldots,S_n)~\delta_{\rho^{i}_{(f,s_1,\ldots,s_{i-1})|(h,s)}(\omega,u)}(dS_{i})\cdots \delta_{\rho^{n}_{(f,s_1,\ldots,s_{i-1})|(h,s)}(\omega,u)}(dS_{n}) \Bigg]\W(d\omega)du 
\end{align*}
is $(r_\sX\otimes\id)(\pi)$ a.s.\ strictly positive applying Lemma \ref{lem:taming}. Hence, we arrive at the contradiction $\int \gamma~d\xi^\pi < \int \gamma d\xi.$ \qed

\subsection{Secondary optimization (and beyond)}\label{sec:sec min}

We set  $\bar\mu=(\mu_0,\ldots,\mu_n)$ and denote by $\Opt(\gamma,\bar\mu)$ the set of optimizers of \eqref{eq:Pgamma}. If $\pi\mapsto \int \gamma d\pi$ is lower semicontinuous then $\Opt(\gamma,\bar\mu)$ is a closed subset of $\RMST(\mu_0,\mu_1,\ldots,\mu_n)$ and therefore also compact.

\begin{definition}[Secondary stop-go pairs]
 Let $\gamma,\gamma':\S{n}_\R\to\R$ be Borel measurable. The set of secondary stop-go pairs of colour $i$ relative to $\xi$, short $\SG^\xi_{2,i}$, consists of all $(f,s_1,\ldots,s_{i-1})|(h,s)\in\S{i}_\R,(g,t_1,\ldots,t_i)\in \S{i}_\R$ such that $f\oplus h(s_{i-1}+s)=g(t_i)$ and either $((f,s_1,\ldots,s_{i-1})|(h,s),(g,t_1,\ldots,t_i))\in\SG^\xi_i$ or equality holds in \eqref{eq:CS} for $\gamma$, and strict inequality holds in \eqref{eq:CS} for $\gamma'$, or equality holds in \eqref{eq:MCS} for $\gamma$, and strict inequality holds in \eqref{eq:MCS} for $\gamma'$. As before we agree that $((f,s_1,\ldots,s_{i-1})|(h,s),(g,t_1,\ldots,t_i))\in\SG^\xi_i$ if either of the integrals in \eqref{eq:CS} or \eqref{eq:MCS} is infinite or not well defined.

We also define the secondary stop-go pairs of colour $i$ relative to $\xi$ in the wide sense, $\widehat{\SG}^\xi_{2,i}$, by $\widehat{\SG}^\xi_{2,i}=\SG^\xi_{2,i}\cup \{(f,s_1,\ldots,s_{i-1})|(h,s)\in \S{i}_\R:A^\xi_{(f,s_1,\ldots,s_{i-1})}(h,s)=1\}\times\S{i}_\R.$

The set of secondary stop-go pairs relative to $\xi$ is defined by $\SG^\xi := \bigcup_{1\leq i\leq n} \SG_i^\xi$. The secondary stop-go pairs in the wide sense are $\widehat{\SG}^\xi:= \bigcup_{1\leq i\leq n} \widehat{\SG}_i^\xi$.
\end{definition}

\begin{theorem}[Secondary minimisation]\label{thm:sec min}
 Let $\gamma,\gamma':\S{n}_\R\to\R$ be Borel measurable. Assume that $\Opt(\gamma,\bar\mu)\neq\emptyset$ and that $\xi\in\Opt(\gamma,\bar\mu)$ is an optimiser for 
\begin{align}\label{eq:sec min}
 P_{\gamma'|\gamma}(\bar\mu)=\inf_{\pi\in\Opt(\gamma,\bar\mu)}\int \gamma'\ d\pi.
\end{align}
Then, for any $i\leq n$ there exists a Borel set $\Gamma_i\subset \S{i}$ such that $r_i(\xi^i)(\Gamma_i)=1$ and 
\begin{align}
\widehat{\SG}_{2,i}^{\xi}\cap(\Gamma_i^{<}\times\Gamma_i)=\emptyset.
\end{align}
\end{theorem}
Theorem \ref{thm:sec min} follows from a straightforward modification of Proposition \ref{prop:modification} by the same proof as for Theorem \ref{thm:monotonicity principle} using Proposition \ref{prop:kele}. We omit further details. 

\begin{remark}
 Of course  the previous theorem can be applied repeatedly to a sequence of functions $\gamma,\gamma',\gamma'',\ldots$ and sets $\Opt(\gamma,\bar\mu),\Opt(\gamma'|\gamma,\bar\mu),\Opt(\gamma''|\gamma'|\gamma,\bar\mu),\ldots\ $ leading to $\SG^\xi_3,\SG^\xi_4,\ldots .$  We omit further details.
\end{remark}

\subsection{Proof of Main Result}

We are now able to conclude, by observing that our main result is now a simple consequence of previous results.

\begin{proof}[Proof of Theorem \ref{thm:monprin2}]
 Since any $\xi\in\RMST(\mu_0,\ldots,\mu_n)$ induces via Lemma \ref{lem:taming} and Corollary~\ref{cor:canonicalRMST} a sequence of stopping times as used for the definition of stop-go pairs in Section \ref{sec:probint} the result follows from Theorem \ref{thm:sec min}.
\end{proof}

\bibliography{joint_biblio}{}
\bibliographystyle{plain}

\end{document}